\documentclass[10pt,reqno]{amsart} 
\usepackage[a4paper, margin=2.3cm]{geometry}

\usepackage{amsthm}

\usepackage{bbm}
\usepackage{mathrsfs}

\usepackage[utf8]{inputenc}
\usepackage{xypic}

\usepackage{enumitem}

\usepackage{float}
\usepackage{amsbsy}
\usepackage[all]{xy}\usepackage{xypic}
\usepackage{braket}
\usepackage[colorlinks = true,citecolor=black]{hyperref}

\hypersetup{
     colorlinks=true,
     linkcolor=blue,
     filecolor=blue,
     citecolor = blue,      
     urlcolor=cyan,
     }

\newtheorem{theorem}{Theorem}[section]
\newtheorem*{theorem*}{Theorem}
\newtheorem{theorem-non}{Theorem}
\newtheorem{lemma-non}{Lemma}

\theoremstyle{definition}

\theoremstyle{definition}

\newtheorem{conjecture-non}{Conjecture}

\newtheorem{corollary-non}{Corollary}
\newtheorem{proposition}[theorem]{Proposition}
\newtheorem{lemma}[theorem]{Lemma}
\newtheorem*{lemma*}{Lemma}
\newtheorem{corollary}[theorem]{Corollary}

\newtheorem*{conjecture*}{Conjecture}

\theoremstyle{definition}
\newtheorem{definition}[theorem]{Definition}
\newtheorem{example}[theorem]{Example}

\theoremstyle{remark}
\newtheorem{remark}[theorem]{Remark}

\DeclareMathOperator{\rank}{rank}

\numberwithin{equation}{section}

\begin{document}

\linespread{1.0}

\title[Principal elliptic bundles and compact homogeneous l.c.K. manifolds]{Principal elliptic bundles and compact homogeneous l.c.K. manifolds}

\author{Eder M. Correa}

\address{{UFMG, Avenida Ant\^{o}nio Carlos, 6627, 31270-901 Belo Horizonte - MG, Brazil}}
\address{E-mail: {\rm edermc@ufmg.br.}}

\begin{abstract} 
\normalsize{In this paper, we provide a systematic and constructive description of Vaisman structures on certain principal elliptic bundles over complex flag manifolds. From this description we explicitly classify homogeneous l.c.K. structures on compact homogeneous Hermitian manifolds using elements of representation theory of complex simple Lie algebras. Moreover, we also describe using Lie theory all homogeneous solutions of the Hermitian-Einstein-Weyl equation on compact homogeneous Hermitian-Weyl manifolds. As an application, we provide a huge class of explicit (nontrivial) examples of such structures on homogeneous Hermitian manifolds, these examples include elliptic bundles over full flag manifolds, elliptic bundles over Grassmannian manifolds, and 8-dimensional locally conformal hyperK\"{a}hler compact manifolds.}

\end{abstract}

\maketitle

\hypersetup{linkcolor=black}

%\tableofcontents

\hypersetup{linkcolor=black}

\section{Introduction}

A locally conformally K\"{a}hler (l.c.K.) manifold is a conformal Hermitian manifold $(M,[g],J)$ such that for one (and hence for all) metric $g$ in the conformal class the corresponding fundamental 2-form $\Omega = g(J \otimes {\text{id}})$ satisfies $d\Omega = \theta \wedge \Omega$, where $\theta$ is a closed 1-form ({\it{Lee form}}). This is equivalent to the existence of an atlas such that the restriction of $g$ to any chart is conformal to a K\"{a}hler metric. The study of l.c.K. manifolds has been developing mainly since the 1970s, although, there are early contributions by P. Libermann (going back to 1954), see for instance \cite{Libermann1}, \cite{Libermann2}. The recent treatment of the subject was initiated by I. Vaisman in 1976. In a long series of papers (see \cite{Dragomir} and references therein) he established the main properties of l.c.K. manifolds, demonstrated a connection with P. Gauduchon's standard metrics \cite{Gauduchon}, and recognized the Boothby metric as l.c.K. \cite{Boothby}. From these, he introduced the notion of {\textit{generalized Hopf manifolds}} (g.H.m.), e.g. \cite{Vaisman}. As pointed out in \cite{Shells}, the name ``generalized Hopf manifold'' was already used by Brieskorn and van de Ven (see \cite{Brieskorn}) for some products of homotopy spheres which do not bear Vaisman's structure. Thus, the generalized Hopf manifolds in the sense of I. Vaisman are known nowadays as {\textit{Vaisman manifolds}}. A l.c.K. manifold $(M,[g],J)$ is said to be Vaisman if its Lee form $\theta $ is parallel with respect to the Levi-Civita connection of $g$.  The structure of compact strongly regular Vaisman manifolds is well-understood, in fact, they are principal flat $S^{1}$-bundles over Sasakian manifolds \cite{Vaisman}. Moreover, it was recently proved that every compact
homogeneous l.c.K. manifold is Vaisman \cite{Gauduchon1}. Another important feature of compact Vaisman manifolds is their relation with Einstein-Weyl geometry. Recall that a Weyl manifold is a conformal manifold equipped with a torsion free connection preserving the conformal structure, called Weyl connection, see for instance \cite{Folland}. It is said to be Einstein-Weyl if the symmetric trace-free part of the Ricci tensor of this connection vanishes cf. \cite{Higa}. As observed by Pedersen, Poon and Swann in \cite{PedersenPoonSwann}, compact l.c.K. manifolds with parallel Lee form describe Hermitian-Einstein-Weyl spaces. Furthermore, under the Vaisman-Sasaki correspondence provided in \cite{VaismanI}, the Einstein-Weyl (strongly regular) compact Vaisman manifolds correspond to compact (regular) Sasaki-Einstein manifolds. 

In the homogeneous setting, combining \cite{Vaisman} with \cite{VaismanII} one can prove that compact homogeneous Vaisman manifolds are defined by flat principal circle bundles over compact homogeneous Sasakian manifolds, and these ones are total spaces of Boothby-Wang fibrations over compact homogeneous K\"{a}hler manifolds \cite{BW}. In this last case, by considering the (strong) regularity of the underlying canonical foliation (see for instance \cite{VaismanII}, \cite{Chen}), it can be shown that compact homogeneous Vaisman manifolds can be realized as principal elliptic bundles over compact homogeneous K\"{a}hler manifolds. In this paper, by following a recent result provided in \cite{Gauduchon1}, which states that any compact homogeneous l.c.K. manifold is Vaisman (i.e. a principal elliptic bundle over a complex flag manifold), our main goal is providing a complete description of Vaisman structures on compact homogeneous l.c.K. manifolds by employing elements of representation theory of complex simple Lie algebras.

\subsection{Main results} In order to give an overview of our main results, let us recall some basic facts on the geometry of flag manifolds. A complex flag manifold $X$ is a compact simply connected homogeneous complex manifold defined as 
\begin{equation}
X = G^{\mathbbm{C}}/P = G/G \cap P,
\end{equation}
where $G^{\mathbbm{C}}$ is a complex simple Lie group with a compact real form given by $G$, and $P \subset G^{\mathbbm{C}}$ is a parabolic Lie subgroup. Considering ${\text{Lie}}(G^{\mathbbm{C}}) = \mathfrak{g}^{\mathbbm{C}}$, if we choose a Cartan subalgebra $\mathfrak{h} \subset \mathfrak{g}^{\mathbbm{C}}$, and a simple root system $\Sigma \subset \mathfrak{h}^{\ast}$, up to conjugation, we have that $P = P_{\Theta}$, for some $\Theta \subset \Sigma$, where $P_{\Theta}$ is a parabolic Lie subgroup determined by $\Theta$, see for instance \cite{Alekseevsky}. From this, in order to emphasize the underlying parabolic Cartan geometry of the pair $(G^{\mathbbm{C}},P)$, let us denote $X = X_{P}$. It is a well-known fact that $X_{P}$ is a compact simply connected homogeneous Hodge manifold (cf. \cite{BorelK}) whose {\textit{Picard group}} can be described by
\begin{equation}
{\text{Pic}}(X_{P}) = \bigoplus_{\alpha \in \Sigma \backslash \Theta}\mathbbm{Z}c_{1}\big ( \mathscr{O}_{\alpha}(1) \big ),
\end{equation}
such that $\mathscr{O}_{\alpha}(1) = G^{\mathbbm{C}}\times_{P}\mathbbm{C}$, $\forall \alpha \in \Sigma \backslash \Theta$, is a (homogeneous) holomorphic line bundle obtained as an associated holomorphic vector bundle to the principal $P$-bundle $P \hookrightarrow G^{\mathbbm{C}} \to X_{P}$, whose underlying $\mathbbm{C}$-linear holomorphic representation of $P$ is defined by the holomorphic character $\chi_{\omega_{\alpha}} \colon P \to {\rm{GL}}(1,\mathbbm{C}) = \mathbbm{C}^{\times}$, associated to the fundamental weight $\omega_{\alpha} \in \mathfrak{h}^{\ast}$, see for instance \cite{BorelH}, \cite{CONTACTCORREA}. By following the Vaisman-Sasaki correspondence provided in \cite{VaismanI}, and the results provided in \cite{Tsukada}, \cite{CONTACTCORREA}, we obtain the following theorem:
\begin{theorem-non}
\label{Theorem1}
Let $X_{P}$ be a complex flag manifold, associated to some parabolic Lie subgroup $P = P_{\Theta} \subset G^{\mathbbm{C}}$, and let $L \in {\text{Pic}}(X_{P})$ be a negative line bundle. Then, for every $\lambda \in \mathbbm{C}^{\times}$, such that $|\lambda| <1$, we have that the $T_{\mathbbm{C}}^{1}$-principal bundle
\begin{equation}
%\label{quotient}
M = {\rm{Tot}}(L^{\times})/\Gamma_{\lambda}, \ \ {\text{such that}}  \ \ \Gamma_{\lambda} = \big \{ \lambda^{n} \in \mathbbm{C}^{\times} \ \big | \ n \in \mathbbm{Z}  \big \},
\end{equation}
admits a Vaisman structure completely determined by the $T_{\mathbbm{C}}^{1}$-principal connection $\Psi \in \Omega^{1}(M;\mathbbm{C})$ locally described by
\begin{equation}
\displaystyle \Psi = \frac{1}{\sqrt{-1}}\bigg [\partial \log \Big ( \big | \big |s_{U}v_{\mu(L)}^{+} \big | \big |^{2}\Big ) + \frac{dw}{w}\bigg],
\end{equation}
for some local section $s_{U} \colon U \subset X_{P} \to G^{\mathbbm{C}}$, where $v_{\mu(L)}^{+}$ is the highest weight vector of weight $\mu(L)$ associated to the irreducible $\mathfrak{g}^{\mathbbm{C}}$-module $V(\mu(L)) = H^{0}(X_{P},L^{-1})^{\ast}$.
\end{theorem-non}

As it can be seen, Theorem \ref{Theorem1} provides a constructive way to describe Vaisman structures on principal elliptic bundles over complex flag manifolds. In particular, it enables us to classify homogeneous l.c.K. structures on compact homogeneous manifolds. Actually, recall that a compact l.c.K. manifold $(M,J,g)$ is said to be homogeneous l.c.K. manifold if it admits an effective and transitive smooth (left) action of a compact connected Lie group $K$ with Lie algebra $\mathfrak{k}$, which preserves the metric $g$ and the complex structure $J$. In this setting, we have
\begin{equation}
K = K_{{\text{ss}}}Z(K)_{0},
\end{equation}
where $K_{{\text{ss}}}$ is a closed, connected, and semisimple Lie subgroup with ${\text{Lie}}(K_{{\text{ss}}}) = [\mathfrak{k},\mathfrak{k}]$, and $Z(K)_{0}$ is the closed connected subgroup defined by the connected component of the identity of the center of $K$ (cf. \cite{Knapp}). Recently, it was shown in \cite{Gauduchon1} that compact homogeneous l.c.K. manifolds are in fact Vaisman, and since compact homogeneous Vaisman manifolds are defined by principal elliptic bundles over flag manifolds \cite{VaismanII}, it follows that for any compact homogeneous l.c.K. manifold $(M,g,J)$ we have an underlying elliptic fibration $T_{\mathbbm{C}}^{1} \hookrightarrow M \to X_{P}$, such that $X_{P} = K_{{\text{ss}}}^{\mathbbm{C}}/P$ is a complex flag manifold defined by some parabolic Lie subgroup $P \subset K_{{\text{ss}}}^{\mathbbm{C}}$. From this characterization, and from Theorem \ref{Theorem1}, we have following result:  

\begin{theorem-non}
\label{Theorem2}
Let $(M,[g],J)$ be a compact homogeneous l.c.K. manifold, and let $K$ be the compact connected Lie group which acts on $M$ by preserving the l.c.K. structure. Suppose that the semisimple Lie subgroup $K_{{\text{ss}}}$ is simply connected and has a unique simple component. Then, the Lee form associated to the l.c.K. structure $([g],J)$ is (up to scale) completely determined by the $1$-form
\begin{equation}
\displaystyle \theta = -\bigg [d\log \Big ( \big | \big |s_{U}v_{\mu(L)}^{+} \big | \big |^{2}\Big ) + \frac{\overline{w}dw + wd\overline{w}}{|w|^{2}}\bigg],
\end{equation}
such that $s_{U} \colon U \subset X_{P} \to K_{\text{ss}}^{\mathbbm{C}}$ is some local section, $v_{\mu(L)}^{+}$ is the highest weight vector of weight $\mu(L)$ associated to the irreducible $\mathfrak{k}_{\text{ss}}^{\mathbbm{C}}$-module $V(\mu(L)) = H^{0}(X_{P},L^{-1})^{\ast}$, and $L \in {\text{Pic}}(X_{P})$ is a negative line bundle.

\end{theorem-non}

The above theorem provides a concrete way to describe homogeneous l.c.K. structures by using elements of representation theory of complex simple Lie algebras. It is worth pointing out that, under the hypotheses of Theorem \ref{Theorem2}, since $b_{1}(M) = 1$, we obtain a precise description, up to homothety, for the standard Gauduchon metric \cite{Gauduchon} on any compact homogeneous l.c.K. manifold, see for instance \cite{Vaisman}.

Given a compact l.c.K. manifold $(M,[g],J)$, such that $\dim_{\mathbbm{R}}(M) = n \geq 6$, without loss of generality, one can suppose that the associated Lee form $\theta$ is harmonic \cite{Gauduchon}. By considering the Levi-Civita connection $\nabla$ associated to the Riemannian metric $g$, one can define a Weyl connection on $M$ by setting 
\begin{equation}
D = \nabla - \frac{1}{2} \Big (\theta \odot {\text{id}} - g \otimes A \Big ),
\end{equation}
such that $A = \theta^{\sharp}$, see for instance \cite[page 125]{Weyl}, \cite{Folland} and \cite{Higa}. It follows from \cite{Gauduchon2} that, if the underlying Hermitian-Weyl manifold $(M,[g],D,J)$ is Hermitian-Einstein-Weyl, i.e., the symmetrised Ricci tensor of $D$ is a multiple of $g$ at each point, then the Ricci tensor ${\text{Ric}}^{\nabla}$ associate to $\nabla$ satisfies
\begin{equation}
{\text{Ric}}^{\nabla} = (n-2) \Big (||\theta||^{2}g - \theta \otimes \theta \Big ),
\end{equation}
 cf. \cite[Equation 40]{Gauduchon2}. Moreover, in this last setting we also have that $\nabla \theta \equiv 0$, which impiles that a compact Hermitian-Einstein-Weyl manifold $(M,[g],D,J)$ is particularly a Vaisman manifold. By following \cite{PedersenPoonSwann}, we have that every compact homogeneous Hermitian-Einstein-Weyl manifold can be obtained from a discrete quotient of the maximal root of the canonical bundle of a complex flag manifold $X_{P}$, i.e., a discrete quotient of the complex manifold underlying the total space of the holomorphic line bundle
\begin{equation}
 \mathscr{O}_{X_{P}}(-1) := \textstyle{\frac{1}{I(X_{P})}}K_{X_{P}},
\end{equation}
here $I(X_{P}) \in \mathbbm{Z}$ denotes the Fano index of $X_{P}$. Combining Theorem \ref{Theorem2} with \cite[Theorem 4.2]{PedersenPoonSwann}, we obtain the following result:

\begin{theorem-non}
\label{Theorem3}
Let $(M,[g],D,J)$ be a compact homogeneous Hermitian-Einstein-Weyl manifold, such that $\dim_{\mathbbm{R}}(M) \geq 6$, and let $K$ be the compact connected Lie group which acts on $M$ by preserving the Hermitian-Einstein-Weyl structure. Suppose that $K_{{\text{ss}}}$ is simply connected and has a unique simple component. Then, the Lee form $\theta_{g} \in \Omega^{1}(M)$ associated to the Hermitian-Einstein-Weyl metric $g$ is locally described by
\begin{equation}
\theta_{g} = -\bigg [ \frac{\ell}{I(X_{P})}d\log \Big ( \big | \big |s_{U}v_{\delta_{P}}^{+} \big | \big |^{2}\Big ) +  \frac{\overline{w}dw + wd\overline{w}}{|w|^{2}} \bigg],
\end{equation}
for some $\ell \in \mathbbm{Z}_{>0}$, such that $s_{U} \colon U \subset X_{P} \to K_{\text{ss}}^{\mathbbm{C}}$ is some local section, $v_{\delta_{P}}^{+}$ is the highest weight vector of weight $\delta_{P}$ associated to the irreducible $\mathfrak{k}_{\text{ss}}^{\mathbbm{C}}$-module $V(\delta_{P}) = H^{0}(X_{P},K_{X_{P}}^{-1})^{\ast}$, and $I(X_{P})$ is the Fano index of $X_{P}$. 
\end{theorem-non}

The result above provides a concrete description for $K$-invariant Hermitian-Einstein-Weyl metrics on compact homogeneous manifolds. Moreover, since in the setting above we also have that $b_{1}(M) = 1$, and the corresponding {\it{Higgs field}} $\theta_{g}$ is harmonic $K$-invariant, it follows that, up to homothety, the homogeneous Hermitian Eisntein-Weyl metric described in Theorem \ref{Theorem3} through $\theta_{g}$ is unique on $(M,J)$. It is worth mentioning that the result provided by Theorem \ref{Theorem3} is consistent with the classification of homogeneous K\"{a}hler-Einstein metrics on compact homogeneous manifolds (in the sense of Y. Matsushima \cite{MATSUSHIMA}), i.e., a compact homogeneous complex manifold admits at most one, up to scale, $K$-invariant Hermitian-Einstein-Weyl structure, cf. \cite[Conjecture 5.1]{OVconjecture}. 

\subsection{Outline of the paper} The content and main ideas of this paper are organized as follows: In Section 2, we shall cover the basic material related to the Vaisman-Sasaki correspondence. In Section 3, we present some basic results on strongly regular Vaisman manifolds realized as principal elliptic bundles over K\"{a}hler manifolds. In Section 4, we present some classical results on Hermitian-Weyl geometry and on Hermitian-Einstein-Weyl geometry. In Section 5, we cover some basic results about the realization of compact homogeneous l.c.K. manifolds as principal elliptic bundles over complex flag manifolds. In Section 6, we prove Theorem \ref{Theorem1}, Theorem \ref{Theorem2}, and Theorem \ref{Theorem3}. In Section 7, we provide several examples which illustrate the concrete applications of our main results.

\section{Vaisman and Sasaki manifolds}
In what follows, we shall cover some basic generalities about the relationship between compact Vaisman manifolds and compact Sasaki manifolds. Further details about the results presented in this section can be found in \cite{Gini}, \cite{Blair}.

\subsection{Vaisman manifolds from K\"{a}hler cones}
Let $(M,g,J)$ be a connected complex Hermitian manifold such that $\dim_{\mathbbm{C}}(M) \geq 2$. Let us denote by $\Omega = g(J \otimes {\text{id}})$ the associated fundamental $2$-form of $(M,g,J)$. 

\begin{definition}
\label{DEFLCK}
A Hermitian manifold $(M,g,J)$ is called locally conformally K\"{a}hler (l.c.K.) if it satisfies one of the following equivalent conditions: 

\begin{enumerate}

\item There exists an open cover $\mathscr{U}$ of $M$ and a family of smooth functions $\{f_{U}\}_{U \in \mathscr{U}}$, $f_{U} \colon U \to \mathbbm{R}$, such that each local metric $g_{U} = {\mathrm{e}}^{-f_{U}}g|_{U}$, is K\"{a}hlerian, $\forall U \in \mathscr{U}$.

\item There exists a globally defined closed $1$-form $\theta \in \Omega^{1}(M)$ (Lee form), such that $d\Omega = \theta \wedge \Omega$.

\end{enumerate}
\end{definition}

\begin{remark}
Throughout this paper, unless otherwise stated, we shall assume that $\theta$ is not exact and $\theta \not \equiv 0$.
\end{remark}

An important subclass of l.c.K. manifolds is defined by the parallelism of the Lee form with respect to the Levi-Civita connection of $g$. Being more precise, we have the following definition.

\begin{definition}
A l.c.K. manifold $(M,g,J)$ is called a Vaisman manifold if $\nabla \theta = 0$, where $\nabla$ is the Levi-Civita connection of $g$.
\end{definition}

In order to establish some results which relate Vaisman manifolds to Sasakian manifolds, let us recall some generalities on Sasaki geometry.

\begin{definition}
\label{ContMetric}
Let $(Q,g_{Q})$ be a Riemannian manifold of dimension $2n+1$. A contact metric structure on $(Q,g_{Q})$ is a triple $(\phi,\xi,\eta)$, where $\phi$ is a $(1,1)$-tensor, $\xi$ is a vector field and $\eta$ is a $1$-form, such that: 

\begin{enumerate}

    \item $\eta \wedge (d\eta)^{n} \neq 0$, \  $\eta(\xi) = 1$, \ $\phi \circ \phi = - {\rm{id}} + \eta \otimes \xi$,
    
    \item $g_{Q}(\phi \otimes \phi) = g_{Q} - \eta \otimes \eta$, \ $d\eta = 2g_{Q}(\phi \otimes {\rm{id}})$.
       
\end{enumerate}
A contact metric structrure $(\phi,\xi,\eta)$ on $(Q,g_{Q})$ is called {\rm{K}}-contact structure if $\mathscr{L}_{\xi}g_{Q} = 0$.
\end{definition}

Given a contact metric structure $(g_{Q},\phi,\xi,\eta)$ on a smooth manifold $Q$, one can consider the manifold defined by its metric cone 
\begin{equation}
\mathscr{C}(Q) = Q \times \mathbbm{R}^{+}.
\end{equation}
Taking the coordinate $r$ on $\mathbbm{R}^{+}$ one can define the warped product Riemannian metric on $\mathscr{C}(Q)$ by setting
\begin{equation}
g_{\mathscr{C}} = r^{2}g_{Q} + dr \otimes dr.
\end{equation}
Furthermore, from $(\phi,\xi,\eta)$ one has an almost-complex structure defined on $\mathscr{C}(Q)$ by setting
\begin{equation}
\label{complexcone}
J_{\mathscr{C}}(Y) = \phi(Y) - \eta(Y)r\partial_{r}, \ \ \ \ \ J_{\mathscr{C}}( r\partial_{r}) = \xi.
\end{equation}
From the above comments, a Sasaki manifold can be defined as follows. 
\begin{definition}
\label{sasakicondition}
A contact metric structure $(g_{Q},\phi,\xi,\eta)$ on a smooth manifold $Q$ is called Sasaki if $(\mathscr{C}(Q),g_{\mathscr{C}},J_{\mathscr{C}})$ is a K\"{a}hler manifold.
\end{definition}
%\begin{remark}
%\label{realhomolomrphic}
%It is worth mentioning that, in the setting of the last definition, by considering the Levi-Civita connection $\nabla$ associated to the Riemannian metric $g_{\mathscr{C}}$ on the warped product $\mathbbm{R}^{+} \times_{r} Q$, it follows from \cite[Page 206]{ONEILL} that
%\begin{enumerate}
%\item $\nabla_{r\partial_{r}} ( r\partial_{r}) = r\partial_{r}$, $\nabla_{r\partial_{r}} (X ) = \nabla_{X}( r\partial_{r}) = X$;

%\item $\nabla_{X} (Y) = \nabla^{Q}_{X} (Y ) - g_{Q}(X,Y)r\partial_{r}$.
%\end{enumerate}
%Here $X$ and $Y$ denote vector fields on $Q$, appropriately interpreted also as vector fields on $\mathscr{C}(Q)$, and $\nabla^{Q}$ is the Levi-Civita connection of $g_{Q}$.
%\end{remark}

Given a Sasakian manifold $(Q,g_{Q})$ with structure tensors $(g_{Q},\phi,\xi,\eta)$, for every $\lambda \in \mathbbm{R}^{+}$ we can define a smooth map $\tau_{\lambda} \colon \mathscr{C}(Q) \to  \mathscr{C}(Q)$, such that 
\begin{equation}
\label{dilatation}
\tau_{\lambda}\colon (x,s) \to (x,\lambda s).
\end{equation}
It is straightforward to show that $\tau_{\lambda}^{\ast}g_{\mathscr{C}}  = \lambda^{2}g_{\mathscr{C}}$, i.e. $\tau_{\lambda}$ is a dilatation (homothety). Moreover, the map defined above is also holomorphic with respect to the complex structure given in Eq. (\ref{complexcone}). We observe that the K\"{a}hler form $\omega_{\mathscr{C}} = g_{\mathscr{C}}(J_{\mathscr{C}} \otimes \text{id})$ can be written as
\begin{equation}
\displaystyle \omega_{\mathscr{C}} = d \Big (\frac{r^{2}\eta}{2} \Big) = rdr \wedge \eta + r^{2}\frac{d\eta}{2}.
\end{equation}
Thus, by setting $\psi = \log(r) \colon \mathscr{C}(Q) \to \mathbbm{R}$, we obtain 

\begin{equation}
\displaystyle \omega_{\mathscr{C}} = {\mathrm{e}}^{2\psi} \Big ( d\psi \wedge \eta + \frac{d\eta}{2}\Big).
\end{equation}
If we consider $\widetilde{\Omega} = {\mathrm{e}}^{-2\psi}\omega_{\mathscr{C}}$, from the identification $\text{pr}_{1} \times \psi \colon Q \times \mathbbm{R}^{+} \to Q \times \mathbbm{R}$, it follows that the map $\tau_{\lambda} \colon \mathscr{C}(Q) \to  \mathscr{C}(Q)$ transforms the coordinate $\psi$ by $\psi \mapsto \psi + \log(\lambda)$, so we have $\tau_{\lambda}^{\ast} \widetilde{\Omega} = \widetilde{\Omega}$. Therefore, if we take the manifold $M = \mathscr{C}(Q) /  \sim_{\lambda}$, where $``  \sim_{\lambda}"$ stands for the equivalence relation $(x,s) \sim (x,\lambda s)$, we obtain a l.c.K. structure $(g,J)$ on $M = Q \times S^{1}$, where $J$ is a complex structure induced by $J_{\mathscr{C}}$, and $\wp^{\ast}g = \widetilde{\Omega}({\rm{id}}\otimes J_{\mathscr{C}})$, where $\wp \colon \mathscr{C}(Q) \to  M$ is the projection map. Observing that 
\begin{equation}
\widetilde{\Omega} = \displaystyle d\psi \wedge \eta + \frac{d\eta}{2},
\end{equation}
if one sets $\widetilde{\theta} =  -2d\psi$, it follows that $ d \widetilde{\Omega} = \widetilde{\theta} \wedge \widetilde{\Omega}$, which implies that $\wp^{\ast}\theta =  -2d\psi$, where $\theta $ is the Lee form associated to the l.c.K. structure $(g,J)$ described above on  $M = Q \times S^{1}$. One can also easily verify that for this last l.c.K. structure the associated Lee form $\theta$ satisfies $\nabla \theta \equiv 0$. 

\begin{remark}
\label{globkahler}
In the description above we implicitly used the fact that there exists a smooth function $f \colon \mathscr{C}(Q) \to \mathbbm{R}$, such that $\widetilde{\Omega}  = {\mathrm{e}}^{f}\omega_{\mathscr{C}}$, i.e. the metric induced by $\widetilde{\Omega} $ on $\mathscr{C}(Q)$ is globally conformal K\"{a}hler, e.g. \cite[Lemma 5.3]{Gini}. 

\end{remark}

The previous construction is a quite natural way to obtain Vaisman structures by means of K\"{a}hler structures defined on cones over Sasakian manifolds. More generally, given a Sasakian manifold $Q$, if we consider an automorphism of the Sasakian structure $\varphi \colon Q \to Q$, we can define $\varphi_{\lambda} \colon \mathscr{C}(Q) \to \mathscr{C}(Q)$,  for some $\lambda \in \mathbbm{R}^{+}$, such that $\varphi_{\lambda} \colon (x,s) \to (\varphi(x),\lambda s)$. The map above also satisfies $\varphi_{\lambda}^{\ast} g_{\mathscr{C}}  = \lambda^{2}g_{\mathscr{C}}$, and we can consider the manifold 
\begin{equation}
M_{\varphi,\lambda} = \mathscr{C}(Q) / \sim_{\varphi,\lambda},
\end{equation}
where ``$\sim_{\varphi,\lambda}$" stands for the equivalence relation $(x,s) \sim (\varphi(x),\lambda s)$. From this, we can equip the manifold $M_{\varphi,\lambda}$ with l.c.K. structure in a similar way as we have done before. Moreover, we have the following result.

\begin{proposition}[\cite{Gini}, \cite{Kamishima}]
\label{fromsasakitovaisman}
Let $Q$ be a compact Sasakian manifold. Then the l.c.K. manifold $M_{\varphi,\lambda}$ obtained from $\varphi_{\lambda} \colon \mathscr{C}(Q) \to \mathscr{C}(Q)$ as above is Vaisman.
\end{proposition}

Notice that under the hypotheses of the last proposition, if we denote 
\begin{equation}
\label{grouppresentation}
\Gamma_{\lambda,\varphi} = \big \{ \varphi_{\lambda}^{n} \ \ \big | \ \ n \in \mathbbm{Z} \big\},
\end{equation}
it follows that $M_{\varphi,\lambda} = \mathscr{C}(Q) / \Gamma_{\lambda,\varphi}$ can be endowed with a structure of Vaisman manifold which comes from the globally conformal K\"{a}hler structure of the cone  $\mathscr{C}(Q) = Q \times \mathbbm{R}^{+}$.
\begin{remark}
\label{Vaismanfromsasaki}
It is worth pointing out that the aforementioned group $\Gamma_{\lambda,\varphi}$ is a subgroup of the group $\mathcal{H}(\mathscr{C}(Q))$ of biholomorphisms $f \colon \mathscr{C}(Q) \to \mathscr{C}(Q)$, which satisfies $f^{\ast}g_{\mathscr{C}} = ag_{\mathscr{C}}$, for some $a \in \mathbbm{R}^{+}$ (scale factor). In the general setting, i.e., when $Q$ is not necessarily compact, we can consider the homomorphism
\begin{equation}
\label{scalehomomorphism}
\rho_{Q} \colon \mathcal{H}(\mathscr{C}(Q)) \to \mathbbm{R}^{+},
\end{equation}
which assigns to each element of $\mathcal{H}(\mathscr{C}(Q))$ its scale factor. Given a subgroup $\Gamma \subset \mathcal{H}(\mathscr{C}(Q))$ which satisfies $\rho_{Q}(\Gamma) \neq 1$, if we suppose that $f \circ \Phi_{t} = \Phi_{t} \circ f$, $\forall f \in \Gamma$, where $\Phi_{t}$ denotes the flow of $\frac{\partial}{\partial \psi}$, and that $\Gamma$ acts freely and
properly discontinuously on $\mathscr{C}(Q)$, it follows that $M_{\Gamma} = \mathscr{C}(Q)/\Gamma$ is a Vaisman manifold. The l.c.K. structure on $M_{\Gamma}$ is obtained from the globally conformal K\"{a}hler structure which we have on $\mathscr{C}(Q)$, see for instance \cite[Proposition 5.4]{Gini}. 
\end{remark}

\begin{definition}
Given a Vaisman manifold $M_{\Gamma} = \mathscr{C}(Q)/\Gamma$, obtained as in Remark \ref{Vaismanfromsasaki}, we call the pair $(\mathscr{C}(Q),\Gamma)$ a presentation of $M_{\Gamma}$.
\end{definition}

Notice that every compact Vaisman manifold $M$ has a natural presentation $(\widetilde{M} = \mathscr{C}(\widetilde{Q}),\Gamma)$, where $\widetilde{M}$ is the universal covering space of $M$, and $\Gamma = \pi_{1}(M)$, e.g. \cite{Kamishima}. The universal covering presentation $(\widetilde{M},\pi_{1}(M))$ is also called the maximal presentation of $M$. The minimal presentation $(M_{\text{min}},\Gamma_{\text{min}})$ of $M$ is defined by $M_{\text{min}} = \widetilde{M}/\big (\pi_{1}(M) \cap {\text{Iso}}(\widetilde{M})\big )$ and $\Gamma_{\text{min}} = \pi_{1}(M)/\big(\pi_{1}(M) \cap {\text{Iso}}(\widetilde{M})\big)$, where ${\text{Iso}}(\widetilde{M}) = \ker(\rho_{\widetilde{Q}})$, for more details we suggest \cite{GiniI}.

\section{Vaisman manifolds and elliptic fibrations over K\"{a}hler manifolds}
\label{Vaismantorusfibration}
This section is devoted to present some basic results related to the canonical foliation of compact Vaisman manifolds. Further details on this subject can be found in \cite{Vaisman}, \cite{Chen}, \cite{Tsukada}, \cite{Dragomir}.

\subsection{The canonical foliation of a compact Vaisman manifold} Given a compact Vaisman manifold $(M,g,J)$ one can consider the vector field $A = \theta^{\sharp} \in \Gamma(TM)$ (Lee vector field), where $g(X,A) = \theta(X)$, $\forall X \in \Gamma(TM)$. From this, one can take $B = JA$, and consider the anti-Lee form $\vartheta = -\theta \circ J$. The vector fields $A,B \in \Gamma(TM)$ are infinitesimal automorphism of the Hermitian structure, i.e., the following hold
\begin{equation}
\label{equationfoliation}
\mathscr{L}_{A}J = \mathscr{L}_{B}J = 0, \ {\text{and}} \ \mathscr{L}_{A}g = \mathscr{L}_{B}g = 0.
\end{equation}
Moreover, one can suppose also that $||\theta|| = ||\vartheta|| = 1$, and $[A,B] = 0$, see for instance \cite{Vaisman}, \cite[pages 37-39]{Dragomir}. The distribution generated by $A,B \in \Gamma(TM)$ defines a foliation $\mathcal{F}$ which is completely integrable, this foliation is called {\textit{canonical ({\text{or}} characteristic) foliation}} of $(M,g,J)$, cf. \cite{Vaisman}, \cite{Tsukada}. The foliation $\mathcal{F}$ plays an important role when we investigate topological and complex analytic properties of a compact Vaisman manifold. As we shall see bellow, the leaves of $\mathcal{F}$ are totally geodesic and locally flat complex submanifolds. Moreover, $\mathcal{F}$ is transversally K\"{a}hlerian.

Before we state the main theorem which ensures the aforementioned facts about $\mathcal{F}$, let us recall some basic results.

\begin{lemma}[Palais, \cite{Palais}] Let $\mathcal{F}$ be a regular foliation on a compact manifold $M$, then we have the following facts: 

\begin{enumerate}

\item The leaves of $\mathcal{F}$ are compact submanifolds of $M$;

\item The leaf space $M / \mathcal{F}$ is a (Hausdorff) compact manifold;

\item The natural projection is a smooth map.

\end{enumerate}

\end{lemma}

\begin{remark}
In general, given a foliated manifold $(M,\mathcal{F})$ ($\rank{(\mathcal{F})} < \dim(M)$), the topology of the leaf space $M / \mathcal{F}$ might be very complicated, possibly non-Hausdorff. Since we are interested in a more restrictive setting, unless otherwise stated, we shall suppose that the leaf space $M / \mathcal{F}$ of a regular foliation is Hausdorff.
\end{remark}

Given a compact Vaisman manifold $(M,g,J)$ with associated Lee form $\theta \in \Omega^{1}(M)$, we denote 
\begin{equation}
\mathcal{F} =  \mathcal{F}_{A} \oplus \mathcal{F}_{B},
\end{equation}
where $\mathcal{F}_{A}$ is the foliation generated by $A = \theta^{\sharp} \in \Gamma(TM)$, and $\mathcal{F}_{B}$ is the foliation generated by $B = JA \in \Gamma(TM)$.

\begin{definition}
\label{regvaisman}
A compact Vaisman manifold $(M,g,J)$ is called regular, if $\mathcal{F}_{A}$ is a regular foliation, and strongly regular, if $\mathcal{F}_{A}$ and $\mathcal{F}$ are both regular.
\end{definition}

Now, from \cite{VaismanI}, see also \cite[Theorem 6.2]{Dragomir}, we have the following result.

\begin{theorem}
\label{flatvaisman}
A compact l.c.K. manifold $(M,g,J)$ is a regular Vaisman manifold if and only if there exists a flat $S^{1}$-principal bundle $\pi \colon M \to M/\mathcal{F}_{A}$, with $Q = M/\mathcal{F}_{A}$ being a connected and compact Sasakian manifold. If this is the case, the restrictions of $\pi$ to the leaves of $\mathcal{F}_{A}$ are covering maps.
\end{theorem}

\begin{remark}
\label{monodromyvaisman}
It is worth pointing out that from the above theorem it follows that for every flat connection $\beta \in \Omega^{1}(M,\mathfrak{u}(1))$ on a compact regular Vaisman manifold $M$, we can consider 
\begin{equation}
\label{holonomyrep}
\text{Hol}_{\beta} \colon \pi_{1}(Q,x_{0}) \to S^{1}, 
\end{equation}
where $\text{Hol}_{\beta} \colon [\gamma] \mapsto {\mathrm{e}}^{\int_{\widetilde{\gamma}}\beta}$, denotes the holonomy representation induced by $\beta$, here $\widetilde{\gamma} \colon S^{1} \to M$ denotes the horizontal lift of $\gamma \colon S^{1} \to Q$ with respect to $\beta \in \Omega^{1}(M,\mathfrak{u}(1))$. From this, we obtain that 

\begin{center}
$M \cong \widetilde{Q} \times_{\text{Hol}_{\beta}} S^{1} =  \big (\widetilde{Q} \times S^{1} \big )/\pi_{1}(Q,x_{0})$,
\end{center}
such that $p \colon \widetilde{Q} \to Q$ is the universal covering space of $Q$, notice that the twisted product above is obtained by considering the action of $\pi_{1}(Q,x_{0})$ on $\widetilde{Q}$ by deck transformations and the action of $\pi_{1}(Q,x_{0})$ on $S^{1}$ through the holonomy representation (Eq. (\ref{holonomyrep})) induced by $\beta$. Since we have a canonical flat connection $\beta = \sqrt{-1} \theta$, induced by the Lee form $\theta \in \Omega^{1}(M)$, we can realize $M$ as an associated bundle to the (trivial) $S^{1}$-principal bundle defined by the Vaisman manifold $\widetilde{Q} \times S^{1}$, see for instance \cite[Chapter 4]{Tondeur}.
\end{remark}

The next result describes the properties of the leaves of the characteristic foliation $\mathcal{F}$ of a compact Vaisman manifold.

\begin{theorem}[Vaisman, \cite{Vaisman}]
\label{blikekahler}
The characteristic foliation $\mathcal{F}$ is a complex analytic 1-dimensional foliation on $(M,g,J)$ with complex parallelizable leaves which are totally geodesic and locally flat submanifolds of $M$. The metric of $M$ is a bundle-like transversally K\"{a}hlerian metric with respect to $\mathcal{F}$.

\end{theorem}

%The result above follows essentially from Equation \ref{equationfoliation} together with the identities 
%\begin{center}
%$\nabla_{A}A = \nabla_{A}B = \nabla_{B}A = \nabla_{B}B = 0,$
%\end{center}
%the last statement of Theorem \ref{blikekahler} follows from Equation \ref{equationfoliation}, and \cite[Theorem 6.2]{Chen}.

%An important consequence of Theorem \ref{blikekahler} is that one can describe locally the Riemannian metric $g$ associated to the underlying Hermitian structure $h = g -\sqrt{-1}\Omega$ by \footnote{Here we consider $\beta_{1} \odot \beta_{2} = \beta_{1} \otimes \beta_{2} + \beta_{2} \otimes \beta_{1}, \forall \beta_{1},\beta_{2} \in \Omega^{1}(M)$. Thus, we have $\beta \odot \overline{\beta} = 2\mathfrak{Re}\big ( \beta \otimes \overline{\beta}\big )$, $\forall \beta \in \Omega^{1}(M) \otimes \mathbbm{C}.$}
%\begin{equation}
%\label{localvaisman}
%g = h_{a\overline{b}}dz^{a} \odot d \overline{z}^{b} + \frac{1}{2} \Big (\theta + \sqrt{-1}\vartheta \Big) \odot \Big (\theta - \sqrt{-1} \vartheta \Big),
%\end{equation}
%here we consider the $(1,0)$ cobasis $\{dz^{a}, \theta + \sqrt{-1}\vartheta\}$, where $z^{a}$, $a = 1, \ldots, n-1$, are complex coordinates which defines $\mathcal{F}$ (locally) by $dz^{a}=0$, $a = 1, \ldots, n-1$. The bundle-like character of the metric \ref{localvaisman} means that the coefficients $h_{a\overline{b}}$ depend only on $z^{a},\overline{z}^{a}$, $a = 1, \ldots, n-1$, see for instance \cite{Vaisman}.

\begin{remark}[Canonical Vaisman structure]
\label{canonicalvaisman}
As we have seen previously, we have a canonical globally conformal K\"{a}hler structure $(\widetilde{\Omega}, J_{\mathscr{C}})$ on a K\"{a}hler cone $\mathscr{C}(Q)$ over a Sasakian manifold $Q$, such that
\begin{equation}
\label{globalconformal1}
\widetilde{\Omega} =  \frac{d\eta}{2} + \displaystyle d\psi \wedge \eta \ \ \ \ \ {\text{and}} \ \ \ \ \ \widetilde{g} =  g_{Q} + d\psi \otimes d\psi.
\end{equation}
By considering
\begin{equation}
\label{structure}
\widetilde{\theta} = -2d\psi \ \ \ \ \ {\text{and}} \ \ \ \ \ \displaystyle g_{Q} = \frac{1}{2}d\eta({\rm{id}}\otimes \phi) + \eta \otimes \eta, 
\end{equation}
it follows that 
\begin{equation}
\label{globalconformal}
\widetilde{g} = \frac{1}{2}d\eta({\rm{id}}\otimes \phi) + \frac{1}{2} \Big (\eta - \sqrt{-1}d\psi \Big) \odot \Big (\eta + \sqrt{-1} d\psi \Big).
\end{equation}
Therefore, if we suppose $Q$ as being compact, from Proposition \ref{fromsasakitovaisman} we have that the l.c.K. manifold $M_{\varphi,\lambda} = (\mathscr{C}(Q), \Gamma_{\varphi,\lambda})$, presented by some $\Gamma_{\varphi,\lambda} \subset \mathcal{H}(\mathscr{C}(Q))$, is a compact Vaisman manifold with Vaisman structure $(\Omega, J, \theta)$ induced from $(\widetilde{\Omega}, J_{\mathscr{C}}, \widetilde{\theta})$, notice that $\widetilde{\nabla}\widetilde{\theta} \equiv 0$, where $\widetilde{\nabla}$ is the Levi-Civita connection associated to $\widetilde{g}$.

\end{remark}

The next result provides a precise description of compact strongly regular Vaisman manifolds in terms of principal toroidal bundles, see for instance \cite{Vaisman}, \cite{Chen}.

\begin{theorem}
\label{vaismanclass}
Let $(M,g,J)$ be a compact connected strongly regular Vaisman manifold. Then the projection

\begin{center}

$\pi \colon M \to M/\mathcal{F},$

\end{center}
defines a $T_{\mathbbm{C}}^{1}$-principal bundle over a compact Hodge manifold $N =  M/\mathcal{F}$. Moreover, $M$ also defines a flat $S^{1}$-principal bundle over a compact Sasaki manifold $Q = M/\mathcal{F}_{A}$, which in turn can be realized as a $S^{1}$-principal bundle over $N$ whose Chern class differs only by a torsion element from the Chern class of the induced $T_{\mathbbm{C}}^{1}$-principal bundle defined by $M$.

\end{theorem}

The converse of Theorem \ref{vaismanclass} is given by the following result.

\begin{theorem}[\cite{Tsukada}]
\label{vaismanbundle}
Let $\pi \colon M \to N$ be a holomorphic $T_{\mathbbm{C}}^{1}$-principal bundle over a compact K\"{a}hler manifold. If $M$ admits a connection $\Psi$ with the following properties:

\begin{enumerate}

\item $\Psi \circ J = \sqrt{-1} \Psi$,

\item there exists a K\"{a}hler form $\omega$ on $N$, such that $d\Psi = c \pi^{\ast} \omega$ for a nonzero real constant $c$,
\end{enumerate}
then $M$ admits a Vaisman structure $(\Omega,J,\theta)$, such that the leaves of the associate characteristic foliation $\mathcal{F}$ coincides with the fibers of the bundle.

\end{theorem}

Notice that, under the hypotheses of Theorem \ref{vaismanbundle}, one can write $\Psi \in \Omega^{1}(M;\mathbbm{C})$ as 
\begin{equation}
\label{connectiontorus}
\Psi = \mathfrak{Re}(\Psi) + \sqrt{-1} \mathfrak{Im}(\Psi),
\end{equation}
such that $\mathfrak{Re}(\Psi), \mathfrak{Im}(\Psi) \in \Omega^{1}(M)$. It follows from condition (1) of Theorem \ref{vaismanbundle} that $\mathfrak{Re}(\Psi) \circ J = - \mathfrak{Im}(\Psi)$. Moreover, from condition (2) of Theorem \ref{vaismanbundle} it can be easily derived that 
\begin{equation}
\label{connectioncond2}
d\big (\mathfrak{Re}(\Psi) \big) = c\pi^{\ast} \omega \ \ \ \ {\text{and}} \ \ \ \ d\big (\mathfrak{Im}(\Psi) \big) = 0,
\end{equation}
for a basic K\"{a}hler form $\omega \in \Omega^{1,1}(N)$. Thus, from $\Psi \in \Omega^{1}(M;\mathbbm{C})$ one can define a Hermitian metric on $M$ by setting
\begin{equation}
\label{bundlemetric}
g = \frac{1}{2} \Big (d\Psi({\text{id}} \otimes J) + \Psi \odot \overline{\Psi} \Big).
\end{equation}
It is straightforward to see that the fundamental form $\Omega = g(J\otimes {\rm{id}})$ satisfies
\begin{equation}
\Omega = \frac{c}{2}\pi^{\ast}\omega - \mathfrak{Im}(\Psi) \wedge \mathfrak{Re}(\Psi) \Longrightarrow d\Omega =  2\mathfrak{Im}(\Psi) \wedge \Omega.
\end{equation}
Hence, from the aforementioned fundamental $2$-form, it follows that the Lee form $\theta \in \Omega^{1}(M)$ associated to the l.c.K. structure defined by the connection $\Psi$ (Eq. ($\ref{connectiontorus}$)) is given by
\begin{equation}
\label{Leefrombundle}
\theta = 2\mathfrak{Im}(\Psi).
\end{equation}
Also, it can be easily shown that $\nabla \theta = 0$, where $\nabla$ is the Levi-Civita connection associated to the Riemannian metric defined in Eq. (\ref{bundlemetric}), see \cite[Theorem 4.1]{Tsukada} for more details.

\subsection{Negative line bundles and principal elliptic bundles}
\label{negativeelliptic}
In this subsection, following \cite[p. 23]{Tsukada}, we shall discuss in details how to construct Vaisman manifolds by means of negative line bundles over compact Hodge manifolds. In order to do this, let us recall some basic facts and generalities related to holomorphic line bundles over compact complex manifolds.

\begin{definition}
\label{Griffithsnegative}
Let $L \in {\text{Pic}}(X)$ be a holomorphic line bundle over a compact complex manifold $X$. We say that $L$ is (Griffiths) negative if there exists a system of trivializations $f_{i} \colon L|_{U_{i}} \to U_{i} \times \mathbbm{C}$ (with transition functions $g_{ik}$) and a Hermitian metric $H$ on $L$, such that 

\begin{enumerate}

\item $H(f_{i}^{-1}(z,w),f_{i}^{-1}(z,w)) = h_{i}(z)|w|^{2}$, where $h_{i} \colon U_{i} \to \mathbbm{R}^{+}$ is a smooth function,

\item ${\rm{K}}_{H} \colon {\text{Tot}}(L^{\times}) \to \mathbbm{R}^{+}$, ${\rm{K}}_{H}(u) = H(u,u)$, is strictly plurisubharmonic.

\end{enumerate}

\end{definition}

Under the hypotheses of the definition above the following characterization will be important for our approach.
\begin{proposition}[\cite{NEGATIVELINEBUNDLE}]
\label{Gnegative}
A line bundle $L \in {\text{Pic}}(X)$ is (Griffiths) negative if only if there exists a system of positive smooth functions $\varrho_{i} \colon U_{i} \to \mathbbm{R}^{+}$, such that: 

\begin{enumerate}

\item $-\log(\varrho_{i})$ is strictly plurisubharmonic on $U_{i}$,

\item $\varrho_{i} = |g_{ik}| \cdot \varrho_{k}$ on $U_{i} \cap U_{k}$.

\end{enumerate}
\end{proposition}
In the setting of the above proposition, we have the Chern class of $L \in {\text{Pic}}(X)$ is represented by the real $(1,1)$-form $\omega \in \Omega^{1,1}(X)$ defined locally by
\begin{equation}
\label{connection}
\displaystyle \omega|_{U_{i}} = \frac{ \sqrt{-1}}{\pi}  \partial \overline{\partial}\log(\varrho_{i}).
\end{equation}
Thus, since for $L \in {\text{Pic}}(X)$ Griffiths negative we have $-\log(\varrho_{i})$ strictly plurisubharmonic on $U_{i}$, it follows that $\omega$ is a negative $(1,1)$-form, which implies that $L^{-1}$ is positive in the sense of Kodaira (e.g. \cite{Lazarsfeld}). It is straightforward to show that a positive line bundle $L$ (in the sense of Kodaira) has its inverse $L^{-1}$ Griffiths negative.

\begin{remark}
\label{griffithsequiv}
Notice that, given a Griffiths negative line bundle $L \in {\text{Pic}}(X)$ with a Hermitian structure $H$, since on each open set $L|_{U_{i}}$ we have

\begin{center}
$H(f_{i}^{-1}(z,w),f_{i}^{-1}(z,w)) = h_{i}(z)|w|^{2},$
\end{center}
by setting $h_{i} = \frac{1}{\varrho_{i}^{2}}$, from the Chern connection $\nabla |_{U_{i}} = d + H^{-1}\partial H $ induced form the Hermitian structure $H$, it follows that $c_{1}(L,\nabla) = [\omega]$, where $\omega$ is given in Eq. (\ref{connection}). It is worth pointing out that this description of $c_{1}(L)$ does not depend on the choice of the Hermitian structure. For further details on this subject we suggest \cite{NEGATIVELINEBUNDLE}, \cite{DANIEL}, \cite{Lazarsfeld}.

\end{remark}

Given a compact complex manifold $X$, and a negative line bundle $L \in {\text{Pic}}(X)$, by fixing a Hermitian metric on $L$ as in Remark \ref{griffithsequiv}, one can consider the sphere bundle\footnote{Here we have $L^{\times} := L \backslash \{{\text{zero section}}\}$}
\begin{equation}
\label{spherebundle}
Q(L) = \Big \{ u \in L^{\times} \ \Big | \ \sqrt{H(u,u)} = 1 \Big \}.
\end{equation}
By taking a Chern connection $\nabla$ on $L$ induced from $H$, we obtain a system of $(1,0)$-forms $A_{i} \in \Omega^{1,0}(U_{i})$, such that $\nabla |_{U_{i}} = d + A_{i} $. If one sets $\mathcal{A}_{i} = \frac{1}{2}(A_{i} - \overline{A}_{i})$, on each $U_{i} \subset X$, then one obtains a connection $1$-form $\eta' \in \Omega^{1}(Q(L);\mathfrak{u}(1))$, locally given by 
\begin{equation}
\label{principalconnection}
\displaystyle \eta' = \pi^{\ast}\mathcal{A}_{i} + \sqrt{-1}d\sigma_{i},
\end{equation}
such that $\pi \colon Q(L) \to X$ is the natural $S^{1}$-principal bundle projection, see \cite[Proposition 1.4]{KN}. Moreover, it follows that 
\begin{equation}
\frac{\sqrt{-1}}{2 \pi} d\eta' = \frac{\sqrt{-1}}{2\pi} \pi^{\ast} F_{\nabla},
\end{equation}
where $F_{\nabla}$ is the curvature of $\nabla$. Since $L$ is negative, one can define a contact structure on $Q(L)$ by setting $\eta = -\sqrt{-1} \eta'$. 

\begin{remark}
Let us recall that an almost contact manifold is a (2n+1)-dimensional smooth manifold $M$ endowed with structure tensors $(\phi, \xi,\eta)$, such that $\phi \in {\text{End}}(TM)$, $\xi \in \Gamma(TM)$, and $\eta \in \Omega^{1}(M)$, satisfying
\begin{equation}
\label{almostcontact}
 \phi \circ \phi = - {\rm{id}} + \eta \otimes \xi, \ \ \eta(\xi) = 1.    
\end{equation}
An almost contact structure is said to be normal if additionally satisfies $[\phi,\phi]+ d\eta \otimes \xi = 0$, where $[\phi,\phi]$ denotes the Nijenhuis torsion of $\phi$, see for instance \cite[Chapter 6]{Blair}.
\end{remark}

Now we consider the following result.

\begin{theorem}[\cite{MORIMOTO}, \cite{HATAKEYMA}]
\label{almostcircle}
Let $Q$ be a principal ${\rm{U}}(1)$-bundle over a complex manifold $(N,J')$. Suppose we have a connection $1$-form $\sqrt{-1}\eta$ on $Q$, such that $d\eta = \pi^{\ast}\omega$. Here $\pi$ denotes the projection of $Q$ onto $N$, and $\omega$ is a $2$-form on $N$ satisfying 

\begin{center}

$\omega(J'X,J'Y) = \omega(X,Y),$

\end{center}
for $X,Y \in \Gamma(TN)$. Then, we can define a $(1,1)$-tensor field $\phi$ on $Q$ and a vector field $\xi$ on $Q$, such that $(\phi,\xi,\eta)$ is a normal almost contact structure on $Q$.
\end{theorem}
If we apply Theorem \ref{almostcircle} above on $Q(L)$ (Eq. (\ref{spherebundle})) we obtain a normal almost contact structure $(\phi,\xi,\eta)$ on $Q(L)$, such that 
\begin{equation}
\label{sasakistructure}
\eta = -\sqrt{-1} \eta', \ \  \xi = \displaystyle \frac{\partial}{\partial \sigma}, \ \  \phi(X) := \begin{cases}
    (J'\pi_{\ast}X)^{H}, \ \ \ {\text{if}}  \ \ X \bot \xi,\\
    \ \ \ \ \  0 \  \  \ \ \ \ \ ,  \ \  \ {\text{if}} \ \ X \parallel \xi,                   \\
  \end{cases}    
\end{equation}
here we denote by $(J'\pi_{\ast}X)^{H}$ the horizontal lift of $J'\pi_{\ast}X$ relative to the connection $\sqrt{-1}\eta \in \Omega^{1}(Q(L);\mathfrak{u}(1))$. Moreover, observing that $d\eta = \pi^{\ast}\omega_{X}$, where $\omega_{X} \in \Omega^{1,1}(X)$, such that $-\omega_{X} \in 2\pi c_{1}(L)$, i.e., $\omega_{X}$ is a K\"{a}hler form on $X$, we can define a Riemannian metric on $Q(L)$ by setting
\begin{equation}
\label{Sasakimetric}
g_{Q(L)} = \frac{1}{2}d\eta(\text{id}\otimes \phi) + \eta \otimes \eta.
\end{equation}
It is straightforward to show that $(g_{Q(L)},\phi,\xi,\eta)$ defines a contact metric structure on $Q(L)$, and from the normality of $(\phi,\xi,\eta)$ we conclude that $(g_{Q(L)},\phi,\xi,\eta)$ is in fact a Sasaki structure on $Q(L)$, see \cite{BoyerGalicki} for more details. Hence, one can associate to every negative line bundle $L$ over a compact complex manifold $X$ a Sasaki manifold $Q(L)$, with structure tensors $(g_{Q(L)},\phi,\xi,\eta)$ completely determined by a principal connection $\sqrt{-1}\eta \in \Omega^{1}(Q(L); \mathfrak{u}(1))$. 

In the above setting, by considering the K\"{a}hler cone $(\mathscr{C}(Q(L)),g_{\mathscr{C}},J_{\mathscr{C}})$, we have the following identification
\begin{equation}
\label{coneandtot}
{\text{Tot}}(L^{\times}) \cong Q(L) \times \mathbbm{R}^{+}.  
\end{equation}
The above identification is defined by 
\begin{equation}
u \in {\text{Tot}}(L^{\times}) \mapsto \bigg (\frac{u}{\sqrt{{\rm{K}}_{H}(u)}}, \sqrt{{\rm{K}}_{H}(u)} \bigg) \in Q(L) \times \mathbbm{R}^{+},
\end{equation}
where ${\rm{K}}_{H} \colon {\text{Tot}}(L^{\times}) \to \mathbbm{R}^{+}$, such that ${\rm{K}}_{H}(u) = H(u,u)$, see Definition \ref{Griffithsnegative}. Under the above identification we can realize $(\mathscr{C}(Q(L)),J_{\mathscr{C}})$ as a holomorphic $\mathbbm{C}^{\times}$-principal bundle over $X$ for which the underlying $\mathbbm{C}^{\times}$-action is defined by the flow (up to scale) of the holomorphic vector field $\xi -\sqrt{-1}J_{\mathscr{C}}(\xi)$\footnote{Notice that $r\partial_{r}$ is real holomorphic, i.e. $\mathscr{L}_{r\partial_{r}}J_{\mathscr{C}} = 0$, which implies that that $\xi = J_{\mathscr{C}}(r\partial_{r})$ is real holomorphic, so we have that $\xi - \sqrt{-1}J_{\mathscr{C}}(\xi)$ is holomorphic, see for instance \cite{Moroianukahler}.}, so the identification given in Eq. (\ref{coneandtot}) is an isomorphism of holomorphic $\mathbbm{C}^{\times}$-principal bundles. Further, by taking $r^{2} = {\rm{K}}_{H}$, from the aforementioned identification we obtain
\begin{equation}
\label{gpotential}
 \frac{\sqrt{-1}}{2} \partial \overline{\partial} {\rm{K}}_{H} = d\Big(\frac{r^{2}\eta}{2}\Big),
\end{equation}
which implies that ${\rm{K}}_{H} \colon {\text{Tot}}(L^{\times}) \to \mathbbm{R}^{+}$ defines a global K\"{a}hler potential. In fact, locally we have 
\begin{center}
${\rm{K}}_{H}(z,w) = {\rm{K}}_{H}(f_{i}^{-1}(z,w)) = h_{i}(z)w\overline{w} = {\mathrm{e}}^{\log(h_{i}(z))}w\overline{w}$,

\end{center}
where $f_{i} \colon L|_{U_{i}} \to U_{i} \times \mathbbm{C}$ is a local trivialization. Thus, it follows that 
$$\displaystyle \partial \overline{\partial} {\rm{K}}_{H} = {\mathrm{e}}^{\log(h_{i})}|w|^{2}\bigg [ \partial \log(h_{i}) \wedge \overline{\partial} \log(h_{i}) + \Big (\frac{dw}{w} \Big) \wedge \overline{\partial}\log(h_{i}) + \partial \log(h_{i}) \wedge \Big (\frac{d \overline{w}}{\overline{w}} \Big ) + \partial \overline{\partial} \log(h_{i}) + \Big (\frac{dw}{w} \Big) \wedge \Big(\frac{d \overline{w}}{\overline{w}} \Big ) \bigg].$$
Since we have $r^{2} = {\rm{K}}_{H}$, we obtain 
\begin{equation}
\label{polar}
w = r {\mathrm{e}}^{-\frac{\log(h_{i})}{2}}{\mathrm{e}}^{\sqrt{-1}\sigma_{i}},
\end{equation}
which implies that 
\begin{center}
$\displaystyle \frac{dw}{w} = \frac{dr}{r} - \frac{d\log(h_{i})}{2} + \sqrt{-1}d\sigma_{i}$, \ \ and \ \  $\displaystyle \frac{d \overline{w}}{\overline{w}} = \frac{dr}{r} - \frac{d\log(h_{i})}{2} - \sqrt{-1}d\sigma_{i}$,
\end{center}
so we obtain

\begin{center}
$\displaystyle \partial \overline{\partial} {\rm{K}}_{H} = r^{2} \bigg [ \frac{dr}{r} \wedge \big ( (\overline{\partial} - \partial)\log(h_{i}) -2 \sqrt{-1}d\sigma_{i} \big)  + \partial \overline{\partial} \log(h_{i}) \bigg]$.
\end{center}
Since $h_{i} = \frac{1}{\varrho_{i}^{2}}$, see Remark \ref{griffithsequiv}, it follows from Eq. (\ref{connection}), and Eq. (\ref{principalconnection}), that 
\begin{center}

$\displaystyle \frac{\sqrt{-1}}{2} \partial \overline{\partial} {\rm{K}}_{H} = rdr \wedge \eta + \frac{r^{2}}{2}d\eta = d\Big(\frac{r^{2}\eta}{2} \Big)$.

\end{center}
Hence, we have the equality given in Eq. (\ref{gpotential}). From these last computations, we can use the characterization given in Eq. (\ref{coneandtot}) to construct elliptic fibrations over $X$ as follows. Given $\lambda \in \mathbbm{C}^{\times}$, such that $|\lambda| < 1$, we can consider the infinite cyclic group 
\begin{equation}
\label{cyclicgroup}
\Gamma_{\lambda} = \big \{ \lambda^{n} \in \mathbbm{C}^{\times} \ \big | \ n \in \mathbbm{Z}  \big \}.
\end{equation}
Since $L^{\times}$ defines a $\mathbbm{C}^{\times}$-principal bundle over $X$, we can consider the action of $\Gamma_{\lambda} \subset \mathbbm{C}^{\times}$ on ${\text{Tot}}(L^{\times})$ by restriction. From the identification ${\text{Tot}}(L^{\times}) = \mathscr{C}(Q(L))$, it follows that
\begin{equation}
\label{principalholomorphic}
u \cdot a \mapsto  \bigg ( \frac{u}{\sqrt{{\rm{K}}_{H}(u)}} \frac{a}{|a|}, |a|\sqrt{{\rm{K}}_{H}(u)} \bigg),
\end{equation}
$\forall a \in \Gamma_{\lambda}$, and $\forall u \in {\text{Tot}}(L^{\times})$. Therefore, considering the polar form $\lambda = |\lambda|{\mathrm{e}}^{ \sqrt{-1}\sigma(\lambda)}$, we have that the action of $\Gamma_{\lambda}$ on ${\text{Tot}}(L^{\times})$ corresponds to the action of $\Gamma_{|\lambda|,\sigma(\lambda)}$ on $\mathscr{C}(Q(L))$ (cf. Eq. (\ref{grouppresentation})), where $\sigma(\lambda) \colon Q(L) \to Q(L)$ is defined by
\begin{center}
$\sigma(\lambda) \colon x \mapsto x \cdot {\mathrm{e}}^{ \sqrt{-1}\sigma(\lambda)}$,
\end{center}
$\forall x \in Q(L)$, i.e. $\sigma(\lambda)$ is an automorphism which preserves the Sasakian structure of $Q(L)$ induced by the flow of $\xi$. From this, we can consider $\varphi_{\lambda} \colon \mathscr{C}(Q(L)) \to \mathscr{C}(Q(L))$, such that $\varphi_{\lambda}(x,s) = (\sigma(\lambda)(x), |\lambda|s)$, $\forall (x,s) \in \mathscr{C}(Q(L))$, and describe the group $\Gamma_{\lambda}$ by

\begin{center}

$\Gamma_{\lambda} = \Big \{  \varphi_{\lambda}^{n} \in {\rm{Diff}}\big(\mathscr{C}(Q(L)) \big) \ \ \Big | \ \ n \in \mathbbm{Z}\Big\}$.

\end{center}
It is straightforward to check that \footnote{It is a consequence of the fact that $\xi$ and $r\partial_{r}$ are real holomorphic vector fields, so their associated flow are biholomorphism of $(\mathscr{C}(Q(L)),J_{\mathscr{C}})$. From this, it is not difficult to see that $\varphi_{\lambda} \colon \mathscr{C}(Q(L)) \to \mathscr{C}(Q(L))$ is obtained from the flow of $\xi$ and $r\partial_{r}$.} $ \Gamma_{\lambda} \subset \mathcal{H}(\mathscr{C}(Q(L)))$, and also that $\rho_{Q(L)}(\Gamma_{\lambda}) \neq 1$, see Remark \ref{Vaismanfromsasaki}. Further, if we consider $\psi := \log(r) \colon \mathscr{C}(Q(L)) \to \mathbbm{R}$, we can check that every $\varphi_{\lambda}^{n} \in \Gamma_{\lambda}$ transforms the coordinate $\psi$ by
\begin{center}
$\psi \mapsto \psi + n \log(|\lambda|)$.
\end{center}
Thus, by taking the flow $\Phi_{t}$ of $\frac{\partial}{\partial \psi}$, it follows that 

\begin{center}
$\Phi_{t} \big (\varphi_{\lambda}^{n}(x,\psi) \big) = \big(x\cdot {\mathrm{e}}^{ \sqrt{-1}\sigma(\lambda)}, \psi + n \log(|\lambda|) + t \big) = \varphi_{\lambda}^{n}\big (\Phi_{t}(x,\psi) \big)$,
\end{center}
which implies that $\Phi_{t} \circ \varphi_{\lambda}^{n} = \varphi_{\lambda}^{n} \circ \Phi_{t}$. Therefore, we have that
\begin{equation}
\label{conepresentation}
M_{\Gamma_{\lambda}} = {\text{Tot}}(L^{\times})/\Gamma_{\lambda} = \big (\mathscr{C}(Q(L)),\Gamma_{\lambda} \big ),
\end{equation}
defines a Vaisman manifold, see Remark \ref{Vaismanfromsasaki} or \cite[Proposition 5.4]{Gini}. Now let us briefly discuss the aspects of $M_{\Gamma_{\lambda}}$ as a principal bundle over $X$. At first, consider 
\begin{equation}
\displaystyle \varpi = \frac{\log(\lambda)}{2\pi \sqrt{-1}},
\end{equation}
here, as before, we have $\lambda \in \mathbbm{C}^{\times}$, such that $|\lambda| < 1$. From this, it follows that the imaginary part of $\varpi$ is a positive real number. Now, we can take the lattice
\begin{equation}
\Lambda = \mathbbm{Z} + \varpi \mathbbm{Z},
\end{equation}
and consider the complex elliptic curve
\begin{equation}
\mathbbm{E}(\Lambda) = \mathbbm{C} / \Lambda.
\end{equation}
From the natural projection map $\pi_{\Lambda} \colon \mathbbm{C} \to \mathbbm{E}(\Lambda)$, we can define a homomorphism $h \colon \mathbbm{C}^{\times} \to \mathbbm{E}(\Lambda)$, such that 
\begin{equation}
h(w) = \pi_{\Lambda} \bigg ( \frac{\log(w)}{2\pi \sqrt{-1}}\bigg),
\end{equation}
$\forall w \in  \mathbbm{C}^{\times}$. Since $\ker(h) = \Gamma_{\lambda}$, it follows that
\begin{equation}
\mathbbm{E}(\Lambda) \cong \mathbbm{C}^{\times}/\Gamma_{\lambda}.
\end{equation}
Therefore, since $L^{\times}$ is a $\mathbbm{C}^{\times}$-principal bundle over $X$, it follows that 
\begin{equation}
\label{ellipticprincipal}
M_{\Gamma_{\lambda}} ={\text{Tot}}(L^{\times})/\Gamma_{\lambda} = \big (L^{\times} \times \mathbbm{E}(\Lambda) \big) / \mathbbm{C}^{\times}.
\end{equation}
Thus, we have $M_{\Gamma_{\lambda}}$ realized as an $\mathbbm{E}(\Lambda)$-bundle over $X$, e.g. \cite{KN}. Notice that, since $\Gamma_{\lambda} \subset \mathbbm{C}^{\times}$ is a normal subgroup, and $\mathbbm{E}(\Lambda) = T_{\mathbbm{C}}^{1}$, we have that $M_{\Gamma_{\lambda}}$ is in fact a $T_{\mathbbm{C}}^{1}$-principal bundle over $X$, see for instance \cite[Proposition 2.2.20]{Gund}.

\subsection{Vaisman structure from principal connections}\label{mainremark}
In order to provide an explicit description for Vaisman structures on $M_{\Gamma_{\lambda}}$ in terms of its principal bundle features, we proceed as follows: Considering $M_{\Gamma_{\lambda}} $ as a principal elliptic bundle over $X$ (Eq. (\ref{ellipticprincipal})), we can construct a connection on $M_{\Gamma_{\lambda}}$ satisfying the conditions of Theorem \ref{vaismanbundle} from a $\mathbbm{C}^{\times}$-principal connection on $L^{\times}$. In order to see that, we define  $\widetilde{\Psi} \in \Omega^{1}(L^{\times};\mathbbm{C})$,
such that
\begin{equation}
\label{potentiallee}
\mathfrak{Re}(\widetilde{\Psi}) = \frac{1}{2}d^{c}\log({\rm{K}}_{H}), \ \ {\text{and}} \ \ \mathfrak{Im}(\widetilde{\Psi}) = -\frac{1}{2}d \log({\rm{K}}_{H}),
\end{equation}
where $d^{c} = \sqrt{-1}( \overline{\partial} - \partial)$. Notice that from the above differential $1$-forms we can recover the canonical globally conformally K\"{a}hler structure on ${\text{Tot}}(L^{\times})$ given in Remark \ref{canonicalvaisman}. Actually, if we consider ${\text{Tot}}(L^{\times}) = \mathscr{C}(Q(L))$, by setting $r^{2} = {\rm{K}}_{H}$, it follows that $\widetilde{\theta} = -2d\psi$, where $\psi = \log(r)$ (see Eq. (\ref{structure})). Moreover, we have
\begin{equation}
\label{contactlee}
\displaystyle d^{c}\log({\rm{K}}_{H}(z,w)) =  \sqrt{-1}( \overline{\partial} - \partial) \log(h_{i}(z)) - \frac{\sqrt{-1}}{|w|^{2}}\big (\overline{w}dw - wd\overline{w} \big),
\end{equation}
where $(z,w)$ are local coordinates in $L^{\times}|_{U_{i}}$, see Definition \ref{Griffithsnegative}. Since $w \in \mathbbm{C}^{\times}$, it can be represented by polar coordinates as in Eq. (\ref{polar}), which provides
\begin{center}
$\displaystyle \frac{\sqrt{-1}}{|w|^{2}}\big (\overline{w}dw - wd\overline{w} \big) = -2d\sigma_{i}$.
\end{center}
Therefore, by rearranging the expression given in Eq. (\ref{contactlee}), from Eq. (\ref{principalconnection}) we can write
\begin{equation}
\label{contactlocal}
\displaystyle d^{c}\log({\rm{K}}_{H}(z,w)) = 2 \bigg [ \frac{\sqrt{-1}}{2}(\partial -\overline{\partial})\log(\varrho_{i}^{2}) + d\sigma_{i} \bigg ] = 2\eta.
\end{equation}
Hence, we obtain from $\widetilde{\Psi} \in \Omega^{1}(L^{\times};\mathbbm{C})$ (Eq. (\ref{potentiallee})) that
\begin{equation}
\displaystyle \widetilde{\Omega} = \frac{1}{2} d \big (\mathfrak{Re}(\widetilde{\Psi})\big ) - \mathfrak{Im}(\widetilde{\Psi}) \wedge \mathfrak{Re}(\widetilde{\Psi}),
\end{equation}
see Eq. (\ref{globalconformal1}). From this, we see that the differential $1$-forms in Eq. (\ref{potentiallee}) define completely the globally conformally K\"{a}hler structure on ${\text{Tot}}(L^{\times}) = \mathscr{C}(Q(L))$ presented in Remark \ref{canonicalvaisman}. In order to verify that $\widetilde{\Psi}$ defined above is a $\mathbbm{C}^{\times}$-principal connection, we observe that:
\begin{enumerate}

\item $a^{\ast} \widetilde{\Psi} = \widetilde{\Psi}$, $\forall a \in \mathbbm{C}^{\times}$,

\item $v \mapsto \frac{v}{2}(\xi - \sqrt{-1}J_{\mathscr{C}}(\xi))$, $\forall v \in \mathbbm{C} = {\text{Lie}}(\mathbbm{C}^{\times})$.

\end{enumerate}
The item (1) above follows from the following. Considering Eq. (\ref{principalholomorphic}), since $a = ({\mathrm{e}}^{ \sqrt{-1}\sigma(a)}, \tau_{|a|})$, where ${\mathrm{e}}^{ \sqrt{-1}\sigma(a)}$ preserves the underlying Sasakian structure of $Q(L)$, and $\tau_{|a|} \colon \mathscr{C}(Q(L)) \to \mathscr{C}(Q(L))$ is a homothety of dilation factor $|a|$, see Eq. (\ref{dilatation}), we have
\begin{center}
$a^{\ast}\eta = \eta$, \ \ and \ \ $a^{\ast}(d\psi) = d(\psi + \log(|a|)).$
\end{center}
Thus, observing that $\widetilde{\Psi} = \eta - \sqrt{-1}d\psi$, we obtain item (1). For item (2), it is straightforward to check that the map in item (2) is the infinitesimal action of ${\text{Lie}}(\mathbbm{C}^{\times})$, recall that $\xi - \sqrt{-1}J_{\mathscr{C}}(\xi)$ is a holomorphic vector field on $\mathscr{C}(Q(L))$. Therefore, one can easily verify that $\widetilde{\Psi}(\widetilde{v}) = v$, where $\widetilde{v} = \frac{v}{2}(\xi - \sqrt{-1}J_{\mathscr{C}}(\xi))$ is the fundamental vector field corresponding to some $v \in {\text{Lie}}(\mathbbm{C}^{\times})$. Since $\eta = - \frac{dr}{r} \circ J_{\mathscr{C}}$, and writing $\widetilde{\Psi} = \eta - \sqrt{-1}\frac{dr}{r}$, it follows that
\begin{equation}
\label{vaismanconnection}
\displaystyle \widetilde{\Psi} \circ J_{\mathscr{C}}  = \eta \circ J_{\mathscr{C}} - \sqrt{-1} \frac{dr}{r} \circ J_{\mathscr{C}} = \sqrt{-1} \widetilde{\Psi}.
\end{equation}
From the theory of connections on principal bundles \cite[ Section 10]{Nomizu}, we have that the $\mathbbm{C}^{\times}$-principal connection $\widetilde{\Psi} \in \Omega^{1}(L^{\times};\mathbbm{C})$ described above induces a unique $T_{\mathbbm{C}}^{1}$-principal connection $\Psi \in \Omega^{1}(M_{\Gamma_{\lambda}};\mathbbm{C})$, such that 
\begin{equation}
\label{pullback}
\wp^{\ast}\Psi = \widetilde{\Psi},
\end{equation}
where $\wp \colon {\text{Tot}}(L^{\times}) \to {\text{Tot}}(L^{\times})/\Gamma_{\lambda}$ is the natural projection. Moreover, since the induced complex structure $J$ on ${\text{Tot}}(L^{\times})/\Gamma_{\lambda}$ satisfies $\wp_{\ast} \circ J_{\mathscr{C}} = J \circ \wp_{\ast}$, it follows that $\Psi \circ J = \sqrt{-1}\Psi$. If we denote by $p_{\Lambda} \colon M_{\Gamma_{\lambda}} \to X$ the projection map, it follows that
\begin{equation}
d\Psi = p_{\Lambda}^{\ast} \omega_{X},
\end{equation}
such that $\omega_{X} = -\sqrt{-1}\partial \overline{\partial}\log(\varrho_{i}^{2})$ is a K\"{a}hler form on $X$, see Eq. (\ref{contactlocal}), and Eq. (\ref{connection}). Notice also that, in this case, we have $d\eta = \pi^{\ast}\omega_{X}$. Therefore, by applying Theorem \ref{vaismanbundle}, we have a Vaisman structure $(g,J)$ on $M_{\Gamma_{\lambda}}$, such that $J$ is the complex structure induced from $J_{\mathscr{C}}$, and 
\begin{equation}
\displaystyle g =  \frac{1}{2} \Big (d\Psi({\text{id}} \otimes J) + \Psi \odot \overline{\Psi} \Big ),
\end{equation}
where $\Psi \in \Omega^{1}(M_{\Gamma_{\lambda}};\mathbbm{C})$ is a $T_{\mathbbm{C}}^{1}$-principal connection. Further, the Lee form $\theta \in \Omega^{1}(M_{\Gamma_{\lambda}})$ in this case is given by $\theta = 2\mathfrak{Im}(\Psi)$, see Eq. \ref{Leefrombundle}, and we have
\begin{equation}
\label{globalpotlee}
\wp^{\ast}\theta = 2\mathfrak{Im}(\widetilde{\Psi}) = - d \log({\rm{K}}_{H}), 
\end{equation}
thus we obtain $\widetilde{\theta} = -2d\psi = \wp^{\ast}\theta$, where $\psi = \log(r)$ (cf. Remark \ref{canonicalvaisman}). It is worth pointing out that from Eq. (\ref{pullback}) we have
\begin{equation}
\wp^{\ast}\Omega = \displaystyle d\psi \wedge \eta + \frac{d\eta}{2}, \ \ {\text{and}} \ \ \wp^{\ast}g =  g_{Q(L)} + d\psi \otimes d\psi,
\end{equation}
where $g_{Q(L)}$ is the Sasaki metric given in Eq. (\ref{Sasakimetric}). Therefore, the Vaisman structure obtained from the  $\mathbbm{C}^{\times}$-principal connection
\begin{equation}
\label{connectionlee}
\widetilde{\Psi} = \frac{1}{2}\big (d^{c} - \sqrt{-1}d\big )\log({\rm{K}}_{H}) = -\sqrt{-1} \partial \log({\rm{K}}_{H}),
\end{equation}
is exactly the canonical Vaisman structure obtained from the K\"{a}hler cone $\mathscr{C}(Q(L))$, see Remark \ref{canonicalvaisman}.

\section{Generalities on Hermitian-Weyl geometry}

In this section, we provide a brief account on Weyl manifolds and Einstein-Weyl metrics. In what follows, by following \cite{Gauduchon2} and \cite{PedersenPoonSwann}, we shall provide a complete characterization of Hermitian-Einstein-Weyl manifolds with strongly regular underlying Vaisman structure in terms of the geometry of certain Calabi-Yau cones associated to Sasaki-Einstein manifolds. As we shall see, this last characterization fits in the context of principal elliptic bundles described in the previous section. 

\subsection{Hermitian-Weyl structures and l.c.K. manifolds} Recall that a conformal structure on a manifold $M$ is an equivalence class \footnote{Here, we say that two Riemannian metrics $g$ and $g'$ are equivalent if and only if $g' = {\rm{e}}^{f}g$, where $f$ is a smooth function on $M$.} $[g ]$ of Riemannian metrics on $M$, such that 
\begin{equation}
[g ] = \Big \{ {\rm{e}}^{f}g  \ \ \Big | \ \  f \in C^{\infty}(M) \Big \}.
\end{equation}
A manifold with a conformal structure is called a conformal manifold. Given a conformal manifold $(M,[g])$, we have the following notion of compatible connection with the conformal class $ [g]$.

\begin{definition}
\label{weylconnectiondef}
A Weyl connection $D$ on a conformal manifold $(M,[g])$ is a torsion-free connection which preserves the conformal class $[g]$. In this last setting, we say that $D$ defines a Weyl structure on $(M,[g])$ and that $(M,[g],D)$ is a Weyl manifold.
\end{definition}

In the above definition by preserving the conformal class it means that for each $g' \in [g]$, we have a $1$-form $\theta_{g'}$ (Higgs field), such that 
\begin{equation}
\label{higgs}
Dg' = \theta_{g'} \otimes g'.
\end{equation}
Note that Eq. (\ref{higgs}) is conformally invariant in the sense that, if $h = {\mathrm{e}}^{f}g'$, for some $f \in C^{\infty}(M)$, then
\begin{equation}
\label{conformalchange}
\theta_{h} = \theta_{g'} + df.
\end{equation}
From this, under a conformal change $g \mapsto {\mathrm{e}}^{f}g$, we have $\theta_{g} \mapsto \theta_{g} + df$, $\forall f \in C^{\infty}(M)$. Conversely, if we start with a Riemannian manifold $(M,g)$ and a fixed $1$-form $\theta \in \Omega^{1}(M)$, denoting $A := \theta^{\sharp}$, the connection
\begin{equation}
\label{WeylLevi}
D = \nabla^{g} - \frac{1}{2} \Big (\theta \odot {\text{id}} - g \otimes A \Big ),
\end{equation}
is a Weyl connection which preserves the conformal class of g. For the sake of simplicity, in the above setting we shall denote by $\nabla = \nabla^{g}$ when the choice of the Riemannian metric is implicit.

Let $(M,[g],D)$ be a Weyl manifold, in what follows we shall fix a representative $g$ for the underlying conformal class and consider the $1$-form $\theta_{g}$ which defines its Higgs field.
\begin{definition}
We say that a Weyl manifold $(M,[g],D)$ is a Hermitian-Weyl manifold if it admits an almost complex structure $J \in {\text{End}}(TM)$, which satisfies:

\begin{enumerate}

\item $g(JX,JY) = g(X,Y)$, $\forall X,Y \in \Gamma(TM);$

\item $DJ = 0.$

\end{enumerate}

\end{definition}

An important result to be considered in the setting of Hermitian-Weyl manifolds is the following proposition.

\begin{proposition}[Vaisman]
\label{HweylVaisman}
Any Hermitian-Weyl manifold of (real) dimension at least $6$ is l.c.K.. Conversely, a l.c.K. manifold of (real) dimension at least $4$ admits a compatible Hermitian-Weyl structure.
\end{proposition}

For a compact Hermitian-Weyl manifold $(M,[g],D,J)$ Gauduchon showed that, up to homothety, there is a unique choice of metric $g_{0}$ in the conformal class such that the corresponding $1$-form $\theta_{g_{0}}$ is co-closed, e.g. \cite{Gauduchon}. 

\begin{definition}
The unique (up to homothety) l.c.K. metric in the conformal class $[g]$ of $(M,[g],D,J)$ with harmonic associated Lee form is called the Gauduchon metric.
\end{definition}

It follows from Proposition \ref{HweylVaisman} that any compact Vaisman manifold admits a Hermitian-Weyl structure uniquely determined by the Gauduchon metric.

\subsection{Hermitian-Einstein-Weyl manifolds} The Einstein-Weyl equations are a conformally invariant generalization of the Einstein equations. Given a Weyl manifold $(M,[g],D)$, if one considers the curvature tensor $R^{D}$ of the affine connection $D$, one can define 
\begin{equation}
{\text{Ric}}^{D}(X,Y) = {\rm{Tr}} \big \{ Z \mapsto R^{D}(X,Z)Y \big \}.
\end{equation}
Since $D$ is not a metric connection, it follows that its Ricci curvature is not necessarily symmetric, thus we say that $(M,[g],D)$ is Einstein-Weyl if 
\begin{equation}
\label{EinsteinWeyl}
{\rm{Sym}}({\text{Ric}}^{D}) = \lambda g, 
\end{equation}
for some representative $g$, where $\lambda$ is a smooth function, and
\begin{equation}
{\rm{Sym}}({\text{Ric}}^{D})(X,Y) = \frac{1}{2} \big ( {\text{Ric}}^{D}(X,Y) + {\text{Ric}}^{D}(Y,X) \big),
\end{equation}
is the symmetric part of ${\text{Ric}}^{D}$. It is worthwhile to observe that, unlike the Einstein equations in manifolds with dimension at least three \cite{Besse}, the function factor $\lambda$ in Eq. (\ref{EinsteinWeyl}) is not constant.

\begin{remark}
\label{einsteinweyl}
The relation between the curvatures associated to $D$ and $\nabla$ (see Eq. (\ref{WeylLevi})) can be described as follows. At first, we consider the scalar curvatures of $D$ with respect to $g$, i.e.
\begin{equation}
\label{Scalweyl}
{\text{Scal}}_{g}^{D} = {\rm{Tr}}_{g}({\text{Ric}}^{D}).
\end{equation}
We can verify that the scalar curvature ${\text{Scal}}_{g} = {\rm{Tr}}_{g}({\text{Ric}}^{\nabla})$ and the Ricci curvature ${\text{Ric}}^{\nabla}$ associated to the Levi-Civita connection $\nabla$ satisfy
\begin{equation}
\label{Scals}
{\text{Scal}}_{g}^{D} = {\text{Scal}}_{g} + 2(n-1)\delta^{g}(\theta_{g}) - (n-1)(n-2)||\theta_{g}||^{2},
\end{equation}
\begin{equation}
\label{ricciweyl}
{\text{Ric}}^{D} = {\text{Ric}}^{\nabla} + \delta^{g}(\theta_{g}) \cdot g - (n-2) \Big ( \nabla^{g}(\theta_{g}) + ||\theta_{g}||^{2}g - \theta_{g} \otimes \theta_{g} \Big ).
\end{equation}
Moreover, for an Einstein-Weyl structure one has 
\begin{equation}
\label{einsteinweyleq}
{\text{Ric}}^{D} = \frac{{\text{Scal}}_{g}^{D}}{n}g - \frac{n-2}{2}d\theta_{g}.
\end{equation}
Thus, Eq. (\ref{EinsteinWeyl}) can be written as ${\rm{Sym}}({\text{Ric}}^{D})  = \frac{{\text{Scal}}_{g}^{D}}{n}g$, see for instance \cite{Gauduchon2}. It is worth observing that, from Eq. (\ref{einsteinweyleq}), it follows that ${\text{Ric}}^{D}$ is symmetric if and only if $D$ is a closed Weyl structure, i.e. $d\theta_{g} = 0$.
\end{remark}

Given a compact Hermitian-Weyl manifold $(M,[g],D,J)$ of (real) dimension at least $6$, by fixing the Gauduchon gauge $g_{0} \in [g]$, it follows that $\theta_{g_{0}}$ is co-closed, see for instance \cite{Besse}. Thus, it follows that $A = \theta_{g_{0}}^{\sharp}$ is a Killing vector field, cf. \cite{Tod}. As we have seen, from Proposition \ref{HweylVaisman} it follows that $(M,g_{0},J)$ is a l.c.K. manifold, which implies that $d\theta_{g_{0}} = 0$. From this, under the assumption that the associated Weyl structure is also Einstein-Weyl, it follows that ${\text{Scal}}_{g_{0}}^{D} = 0$ and ${\text{Ric}}^{D} \equiv 0$. In fact, since $\theta_{g_{0}}$ is parallel, it follows that $||\theta_{g_{0}}||$ is constant, so from {\it{Bochner's}} formula we obtain 
\begin{equation}
0 = -\frac{1}{2}\triangle ||\theta_{g_{0}}||^{2} = -g_{0}(\triangle \theta_{g_{0}},\theta_{g_{0}}) + ||\nabla \theta_{g_{0}}||^{2} + {\text{Ric}}^{\nabla}(A,A),
\end{equation}
which implies from Eq. (\ref{ricciweyl}), and Eq. (\ref{einsteinweyleq}), that
\begin{equation}
0 = {\text{Ric}}^{\nabla}(A,A) = {\text{Ric}}^{D}(A,A) = \frac{{\text{Scal}}_{g}^{D}}{n}||A||^{2},
\end{equation}
Hence, from Eq. (\ref{ricciweyl}), and Eq. (\ref{einsteinweyleq}), the Einstein-Weyl condition turns out to be 
\begin{equation}
\label{Ric}
{\text{Ric}}^{\nabla} = (n-2) \Big (||\theta_{g_{0}}||^{2}g_{0} - \theta_{g_{0}} \otimes \theta_{g_{0}}\Big ).
\end{equation}
In a more general setting we have the following theorem.
\begin{theorem}[\cite{Gauduchon2}]
\label{EinsteinWeylGauduchon}
Let $D$ be an Einstein-Weyl structure on a compact, oriented conformal manifold $(M,[g])$ of dimension $>2$. Let $g_{0} \in [g]$ be the Gauduchon metric associated to $D$ and $\theta_{g_{0}}$ the corresponding Higgs field. If $\theta_{g_{0}}$ is closed, but not exact, then
\begin{enumerate}
\item $\theta_{g_{0}}$ is $\nabla^{g_{0}}$-parallel, i.e. $\nabla^{g_{0}}\theta_{g_{0}} \equiv 0$ (in particular, $\theta_{g_{0}}$ is also $g_{0}$-harmonic).
\item ${\text{Ric}}^{D} \equiv 0$.
\end{enumerate}
\end{theorem}
From the above discussion the next result provides a unified setting which takes into account the aspects of the structure of principal elliptic fibration which underlies compact Hermitian-Einstein-Weyl manifolds of real dimension $\geq 6$. More precisely, we have the following result.

\begin{theorem}[\cite{PedersenPoonSwann}] 
\label{HEW}
Each compact Hermitian-Einstein-Weyl manifold $(M,[g],D,J)$ ($\dim_{\mathbbm{R}}(M) \geq 6$) which is strongly regular Vaisman manifold, arises as a fibration over a compact K\"{a}hler-Einstein manifold $X$ ($\dim_{\mathbbm{R}}(X) = \dim_{\mathbbm{R}}(M) - 2$) of positive scalar curvature. Moreover, $M$ is obtained as a discrete quotient of the Ricci-flat K\"{a}hler structure on a $\mathbbm{C}^{\times}$-principal bundle associated to a maximal root of the canonical bundle of $X$.
\end{theorem}

\begin{remark}
\label{canonicalroot}
In analogy to the tautogical line bundle over the projective space, given a compact connected K\"{a}hler-Einstein manifold $X$ of positive scalar curvature (Fano manifold), we shall denote the maximal root of its canonical bundle by
\begin{equation}
\mathscr{O}_{X}(-1) := K_{X}^{\otimes \frac{1}{I(X)}},
\end{equation}
here $I(X) \in \mathbbm{Z}$ denotes the Fano index of $X$. We also shall denote $\mathscr{O}_{X}(-\ell) := \mathscr{O}_{X}(-1)^{\otimes \ell}$,  for every $\ell \in \mathbbm{Z}$.
\end{remark}

\section{Compact homogeneous Vaisman manifolds}

In this section, we provide a description of compact homogeneous l.c.K. manifolds and explain their relation with complex flag manifolds. 

\subsection{Principal elliptic bundles over flag manifolds} 
\label{compacthomogeneoussetting}
Before we move on to the homogenous setting, let us recall some basic facts which summarize some ideas developed in the previous sections. As we have seen, given a compact Vaisman manifold $(M,J,\theta)$ with strongly regular canonical foliation $\mathcal{F} =  \mathcal{F}_{A} \oplus \mathcal{F}_{B}$, we obtain the following principal bundles:
\begin{center}
$ S^{1} \hookrightarrow (M,J,\theta) \to M/\mathcal{F}_{A} = Q$ \ \ and \ \ $T_{\mathbbm{C}}^{1} \hookrightarrow (M,J,\theta) \to M/\mathcal{F} = X$,
\end{center}
where the first fibration on the left-hand side above is a flat $S^{1}$-principal bundle over a Sasaki manifold $Q$, and the second fibration on the right-hand side above is a $T_{\mathbbm{C}}^{1}$-principal bundle over a compact Hodge manifold $X$. Also, we have a $S^{1}$-principal bundle defined by $S^{1} \hookrightarrow Q \to X$, where $Q$ is a compact Sasaki manifold. Under the assumption that $X$ is simply connected, we have that the Chern class of $M$ as a $T_{\mathbbm{C}}^{1}$-principal bundle over $X$ coincides with the Euler class of $Q$ as a $S^{1}$-principal bundle over $X$, see for instance \cite{Chen}. Thus, in this last case, we can realize $(M,J,\theta)$ as a quotient space of the form
\begin{equation}
\label{withouttorsion}
M = {\text{Tot}}(L^{\times})/\Gamma,
\end{equation}
where $L \in \text{Pic}(X)$ is a negative line bundle, such that $L = Q \times_{{\rm{U}}(1)} \mathbbm{C}$, and $\Gamma = \Gamma_{\lambda} \subset \mathbbm{C}^{\times}$ is a cyclic group obtained from $\lambda \in \mathbbm{C}^{\times}$, such that $|\lambda| < 1$, see for instance Subsection \ref{negativeelliptic} and \cite[Theorem 4.2]{Tsukada}.

In what follows we shall explore the basic ideas described above in the setting of homogeneous spaces. Our exposition is essentially based on \cite{Gauduchon1}, for more details about the background in Lie theory, we suggest \cite{Knapp}, \cite{Humphreys}, see also \cite{PARABOLICTHEORY}, \cite{Akhiezer}.

\begin{definition}
Let $(M,J,g)$ be a compact l.c.K. manifold, we say that $(M,J,g)$ is a homogeneous l.c.K. manifold if it admits an effective and transitive smooth (left) action of a compact connected Lie group $K$, which preserves the metric $g$ and the complex structure $J$.
\end{definition}

\begin{definition}
A compact Hermitian-Weyl manifold $(M,[g],D,J)$ is said to be homogeneous if it admits an effective and transitive smooth (left) action of a compact connected Lie group $K$, which preserves the metric $g$ and the complex structure $J$.
\end{definition}

\begin{theorem}[\cite{Gauduchon1}]
\label{LCKisVaisman}
Any compact homogeneous l.c.K. manifold $(M,g,J)$ is Vaisman.
\end{theorem}

From Theorem \ref{LCKisVaisman}, in the setting of homogeneous l.c.K. manifolds we have that the Harmonic representative and the $K$-invariant representative\footnote{Observe that in the setting of Theorem \ref{LCKisVaisman}, it follows that $H^{1}(M,\mathbbm{R}) \cong H_{K}^{1}(M,\mathbbm{R})$, where $H_{K}^{\bullet}(M,\mathbbm{R})$ denotes the invariant cohomology group of $(M,g,J)$ as a left $K$-manifold, cf. \cite[Theorem 1.28]{Algmodels}.} in $H^{1}(M,\mathbbm{R})$ for the associated Lee form are the same, i.e. a $K$-invariant l.c.K. metric is the Gauduchon metric in its conformal class. Just like in the case of compact homogeneous contact manifolds \cite{BW}, we have the following result related to the regularity of the canonical foliation of compact homogeneous Vaisman manifolds.

\begin{theorem}[Vaisman \cite{VaismanII}, \cite{Vaisman}] 
\label{Vaismanregular}
Any compact homogeneous Vaisman manifold $(M,g,J)$ is strongly regular. Moreover, the following facts hold:
\begin{enumerate}

\item Associated to $M$ we have two principal fiber bundles given by
\begin{center}
$ S^{1} \hookrightarrow (M,g,J) \to M/\mathcal{F}_{A} = Q$ \ \ and \ \ $T_{\mathbbm{C}}^{1} \hookrightarrow (M,g,J) \to M/\mathcal{F} = X$,
\end{center}
where $Q$ is a compact homogeneous Sasaki manifold, and $X$ is a compact simply connected homogeneous Hodge manifold.

\item The manifold $Q$ above defines a Boothby-Wang fibration over $X$, i.e. $Q$ is a $S^{1}$-principal bundle over $X$ whose Euler class defines an invariant Hodge metric on $X$.
\end{enumerate}
\end{theorem}
In the setting of the above theorem, since the base manifold $X$ is simply connected, it follows from Theorem \ref{vaismanclass} that the principal fiber bundles 
\begin{center}
$T_{\mathbbm{C}}^{1} \hookrightarrow M \to X$ \ \ and \ \ $ S^{1} \hookrightarrow Q \to X$,
\end{center}
are defined by the same characteristic class in $H^{2}(X,\mathbbm{Z})$. Therefore, we can realize $M$ as a quotient space
\begin{equation}
M = {\text{Tot}}(L^{\times})/\Gamma,
\end{equation}
where $L$ is a negative line bundle over $X$ which satisfies $Q(L) = Q$ (see Eq. (\ref{spherebundle})) and $\Gamma \subset \mathbbm{C}^{\times}$ is a cyclic group. Another important fact concerned with the setting above is that, since the base Hodge manifold $X$ is simply connected, and its odd Betti numbers vanish \cite{BorelH}, \cite[Lemma 26.1]{Borel}, it can be shown from the Gysin sequence for the fibrations described in Theorem \ref{Vaismanregular} that $b_{1}(M)  = 1$, see for instance \cite[Remark 6.2]{Dragomir}, \cite{Vaisman}.

Given a compact homogeneous l.c.K. manifold $(M,g,J)$, if we consider the decomposition of $\mathfrak{k} = {\text{Lie}}(K)$ given by $\mathfrak{k} = [\mathfrak{k},\mathfrak{k}] \oplus \mathfrak{z}(\mathfrak{k})$, it follows that the compact reductive Lie group $K$ can be written as
\begin{equation}
K = K_{{\text{ss}}}Z(K)_{0},
\end{equation}
such $K_{{\text{ss}}}$ is a closed, connected and semisimple Lie subgroup with ${\text{Lie}}(K_{{\text{ss}}}) = [\mathfrak{k},\mathfrak{k}]$, and $Z(K)_{0}$ is the closed connected subgroup defined by the connected component of the identity of the center of $K$, see for instance \cite{Knapp}. It follows from \cite{Gauduchon1} that the compact simply connected homogeneous Hodge manifold $X$ given in Theorem \ref{Vaismanregular} (see item (2)) can be realized as a quotient space
\begin{equation}
\label{HomogeneousHd}
X = K_{{\text{ss}}}/H,
\end{equation}
where $H$ is compact, connected and equal to the centralizer of a torus of $K_{{\text{ss}}}$, see for instance \cite{BorelK}. Considering the description provided in Eq. (\ref{HomogeneousHd}), if we take the decomposition of $K_{{\text{ss}}}$ into simple Lie groups, i.e.
\begin{equation}
K_{{\text{ss}}} = G_{1} \times \cdots \times G_{r},
\end{equation}
then we have $H = H_{1} \times \cdots \times H_{r}$, such that $H_{j} = H \cap G_{j}$ is a centralizer of a torus in $G_{j}$, for every $j = 1,\ldots, r$, and
\begin{equation}
\label{decompositionflag}
X = \big (G_{1}/H_{1} \big)\times \cdots \times \big (G_{r}/H_{r} \big),
\end{equation}
see for instance \cite{BorelK}. 
If we consider the complexification $\mathfrak{k}^{\mathbbm{C}} = \mathfrak{k} \otimes \mathbbm{C}$, by choosing a {\textit{Cartan subalgebra}} $\mathfrak{h} \subset \mathfrak{k}^{\mathbbm{C}}$, and fixing a root system $\Pi = \Pi^{+} \cup \Pi^{-}$, we obtain a triangular decomposition 
\begin{equation}
\mathfrak{k}^{\mathbbm{C}} = \bigg ( \sum_{\alpha \in \Pi^{+}}\mathfrak{k}_{\alpha}\bigg) \oplus \mathfrak{h} \oplus \bigg ( \sum_{\alpha \in \Pi^{-}}\mathfrak{k}_{\alpha}\bigg).
\end{equation}
Then, if we take the (canonical) associated {\textit{Borel subalgebra}}
\begin{equation}
\mathfrak{b} = \bigg ( \sum_{\alpha \in \Pi^{+}}\mathfrak{k}_{\alpha}\bigg) \oplus \mathfrak{h},
\end{equation}
it can be shown that there exists a parabolic Lie subalgebra $\mathfrak{p} \subset \mathfrak{k}^{\mathbbm{C}}$, which contains $\mathfrak{b}$, such that its normalizer $P = N_{K_{ss}^{\mathbbm{C}}}(\mathfrak{p})$ (parabolic Lie subgroup) satisfies
\begin{equation}
H = P \cap K_{ss}.
\end{equation}
Moreover, it follows from the Iwasawa decomposition of $K_{ss}^{\mathbbm{C}}$ that
\begin{equation}
\label{cartesian}
X = K_{ss}^{\mathbbm{C}}/P.
\end{equation}
From this, we have $K_{ss}^{\mathbbm{C}} = G_{1}^{\mathbbm{C}}\times \cdots \times G_{r}^{\mathbbm{C}}$, and
\begin{equation}
\label{flagdecompgeral}
X = \big (G_{1}^{\mathbbm{C}}/P_{1} \big) \times \cdots \times \big (G_{r}^{\mathbbm{C}}/P_{r} \big ),
\end{equation}
such that $P_{j} = P\cap G_{j}^{\mathbbm{C}}$ is a parabolic Lie subgroup of $G_{j}^{\mathbbm{C}}$, for every $j = 1,\ldots,r$. Further, we also can write
\begin{equation}
\label{parabolicdecomp}
P = P_{1} \times \cdots \times P_{r},
\end{equation}
notice that we have $H_{j} = P_{j} \cap G_{j}$, for every $j = 1,\ldots,r$. In the above setting, if we denote by $\Sigma \subset \mathfrak{h}^{\ast}$ the associated simple root system defined by the pair $(\mathfrak{k}^{\mathbbm{C}},\mathfrak{h})$, it follows that 
\begin{equation}
\mathfrak{h} = \mathfrak{h}_{1} \times \cdots \times \mathfrak{h}_{r} \ \ {\text{and}} \ \ \Sigma = \Sigma_{1}\cup \cdots \cup \Sigma_{r},
\end{equation}
such that $\mathfrak{h}_{j} \subset \mathfrak{g}_{j}^{\mathbbm{C}}$ is a Cartan subalgebra, and $\Sigma_{j} \subset \mathfrak{h}_{j}^{\ast}$, $j = 1,\ldots,r$, is a simple root system. We notice that the parabolic Lie subgroup $P \subset K_{ss}^{\mathbbm{C}}$ is uniquely determined by $\Theta_{P} \subset \Sigma$. Actually, we have $\mathfrak{p} = \mathfrak{p}_{\Theta_{P}}$, such that
\begin{equation}
\mathfrak{p}_{\Theta_{P}} = \bigg ( \sum_{\alpha \in \Pi^{+}}\mathfrak{k}_{\alpha}\bigg) \oplus \mathfrak{h} \oplus \bigg ( \sum_{\alpha \in \langle \Theta_{P} \rangle^{-}}\mathfrak{k}_{\alpha}\bigg),
\end{equation}
where $\langle \Theta_{P} \rangle^{-} = \langle \Theta_{P} \rangle \cap \Pi^{-}$. Furthermore, from Eq. (\ref{parabolicdecomp}) we obtain
\begin{equation}
\Theta_{P} = \Theta_{P_{1}} \cup \cdots \cup \Theta_{P_{r}},
\end{equation}
such that $\mathfrak{p}_{\Theta_{P_{j}}} = \mathfrak{p}_{j} \subset \mathfrak{g}_{j}^{\mathbbm{C}}$ is a parabolic Lie subalgebra defined by
\begin{equation}
\mathfrak{p}_{j} = \bigg ( \sum_{\alpha \in \Pi_{j}^{+}}\mathfrak{g}_{\alpha}^{(j)}\bigg) \oplus \mathfrak{h}_{j} \oplus \bigg ( \sum_{\alpha \in \langle \Theta_{P_{j}} \rangle^{-}}\mathfrak{g}_{\alpha}^{(j)}\bigg),
\end{equation}
for every $j = 1,\ldots,r$. The above ideas show that each factor in the decomposition given in Eq. (\ref{flagdecompgeral}) bears the same Lie-theoretical features as the Cartesian product $X$ described in Eq. (\ref{cartesian}). Moreover, given a negative line bundle $L \to X$, it follows that 
\begin{equation}
L = \text{pr}_{1}^{\ast} L_{1} \otimes \cdots \otimes \text{pr}_{r}^{\ast} L_{r},
\end{equation}
where $\text{pr}_{j} \colon X \to G_{j}^{\mathbbm{C}}/P_{j}$ is the canonical projection and $L_{j} \in {\text{Pic}}(G_{j}^{\mathbbm{C}}/P_{j})$, $j = 1,\ldots,r$. Therefore, from Eq. (\ref{withouttorsion}), and the fact that ${\text{Pic}}(G_{j}^{\mathbbm{C}}/P_{j})$ can be described purely in terms of Lie-theoretical elements (e.g. \cite{Bott}, \cite{BorelH}), the study of compact homogeneous l.c.K. manifolds reduces to the study of principal elliptic bundles defined by negative line bundles over a compact simply connected homogeneous Hodge manifold defined as
\begin{equation}
X = G^{\mathbbm{C}}/P = G/G \cap P,
\end{equation}
where $G^{\mathbbm{C}}$ is a connected simply connected complex simple Lie group with compact real form given by $G$, and $P \subset G^{\mathbbm{C}}$ is a parabolic Lie subgroup. In order to emphasize the parabolic Lie subgroup which defines a complex flag manifold, as well as the parabolic geometry defined by the pair $(G^{\mathbbm{C}},P)$, we shall denote a flag manifold by $X_{P} = G^{\mathbbm{C}}/P$.

\section{Proof of main results and examples} 

In this section, we develop a study of principal elliptic fibrations over complex flag manifolds using elements of representation theory of Lie algebras and Lie groups. The aim is to prove Theorem \ref{Theorem1}, Theorem \ref{Theorem2} and Theorem \ref{Theorem3}. In what follows, we will keep the notation introduced in the previous section.

\subsection{Vaisman structures on principal elliptic bundles over flag manifolds}
In order to study principal elliptic fibrations defined by negative line bundles over complex flag manifolds, let us collect some general facts about line bundles in the homogeneous setting through Lie-theoretical elements. 

\subsubsection{Line bundles over complex flag manifolds}
\label{subsec3.1}
Let $\mathfrak{g}^{\mathbbm{C}}$ be a complex simple Lie algebra, by fixing a  Cartan subalgebra $\mathfrak{h}$ and a simple root system $\Sigma \subset \mathfrak{h}^{\ast}$, we have a (triangular) decomposition of $\mathfrak{g}^{\mathbbm{C}}$ given by
\begin{center}
$\mathfrak{g}^{\mathbbm{C}} = \mathfrak{n}^{-} \oplus \mathfrak{h} \oplus \mathfrak{n}^{+}$, 
\end{center}
where $\mathfrak{n}^{-} = \sum_{\alpha \in \Pi^{-}}\mathfrak{g}_{\alpha}$ and $\mathfrak{n}^{+} = \sum_{\alpha \in \Pi^{+}}\mathfrak{g}_{\alpha}$, here we denote by $\Pi = \Pi^{+} \cup \Pi^{-}$ the root system associated to the simple root system $\Sigma = \{\alpha_{1},\ldots,\alpha_{l}\} \subset \mathfrak{h}^{\ast}$. Let us denote by $\kappa$ the Cartan-Killing form of $\mathfrak{g}^{\mathbbm{C}}$. From this, for every  $\alpha \in \Pi^{+}$ we have $h_{\alpha} \in \mathfrak{h}$, such  that $\alpha = \kappa(\cdot,h_{\alpha})$, and we can choose $x_{\alpha} \in \mathfrak{g}_{\alpha}$ and $y_{\alpha} \in \mathfrak{g}_{-\alpha}$, such that $[x_{\alpha},y_{\alpha}] = h_{\alpha}$. Moreover, for every $\alpha \in \Sigma$, we can set 
$$h_{\alpha}^{\vee} = \frac{2}{\kappa(h_{\alpha},h_{\alpha})}h_{\alpha}.$$ 
The fundamental weights $\{\omega_{\alpha} \ | \ \alpha \in \Sigma\} \subset \mathfrak{h}^{\ast}$ are defined by requiring that $\omega_{\alpha}(h_{\beta}^{\vee}) = \delta_{\alpha \beta}$, $\forall \alpha, \beta \in \Sigma$. We denote by $$\Lambda_{\mathbbm{Z}_{\geq 0}}^{\ast} = \bigoplus_{\alpha \in \Sigma}\mathbbm{Z}_{\geq 0}\omega_{\alpha},$$ the set of integral dominant weights of $\mathfrak{g}^{\mathbbm{C}}$. From the Lie algebra representation theory, for every $\mu \in \Lambda_{\mathbbm{Z}_{\geq 0}}^{\ast}$ we have an irreducible $\mathfrak{g}^{\mathbbm{C}}$-module $V(\mu)$ with highest weight $\mu$. We denote by $v_{\mu}^{+} \in V(\mu)$ the highest weight vector associated to $\mu \in  \Lambda_{\mathbbm{Z}_{\geq 0}}^{\ast}$. Let us denote by $G^{\mathbbm{C}}$ the connected, simply connected, and complex Lie group with simple Lie algebra $\mathfrak{g}^{\mathbbm{C}}$, and consider $G \subset G^{\mathbbm{C}}$ as being a compact real form for $G^{\mathbbm{C}}$. In this setting, given a parabolic Lie subgroup $P \subset G^{\mathbbm{C}}$, without loss of generality, we can suppose that
\begin{center}
$P  = P_{\Theta}$, \ for some \ $\Theta \subseteq \Sigma$.
\end{center}
Recall that, by definition, we have $P_{\Theta} = N_{G^{\mathbbm{C}}}(\mathfrak{p}_{\Theta})$, where ${\text{Lie}}(P_{\Theta}) = \mathfrak{p}_{\Theta} \subset \mathfrak{g}^{\mathbbm{C}}$ is a parabolic Lie subalgebra given by
\begin{center}

$\mathfrak{p}_{\Theta} = \mathfrak{n}^{+} \oplus \mathfrak{h} \oplus \mathfrak{n}(\Theta)^{-}$, \ with \ $\mathfrak{n}(\Theta)^{-} = \displaystyle \sum_{\alpha \in \langle \Theta \rangle^{-}} \mathfrak{g}_{\alpha}$, 

\end{center}
and $N_{G^{\mathbbm{C}}}(\mathfrak{p}_{\Theta})$ is its normalizer in  $G^{\mathbbm{C}}$. In what follows, it will be useful for us to consider the following basic chain of Lie subgroups

\begin{center}

$T^{\mathbbm{C}} \subset B \subset P \subset G^{\mathbbm{C}}$.

\end{center}
For each element in the aforementioned chain of Lie subgroups we have the following characterization: 

\begin{itemize}

\item $T^{\mathbbm{C}} = \exp(\mathfrak{h})$;  \ \ (complex torus)

\item $B = N^{+}T^{\mathbbm{C}}$, where $N^{+} = \exp(\mathfrak{n}^{+})$; \ \ (Borel subgroup)

\item $P = P_{\Theta} = N_{G^{\mathbbm{C}}}(\mathfrak{p}_{\Theta})$, for some $\Theta \subset \Sigma \subset \mathfrak{h}^{\ast}$. \ \ (parabolic subgroup)

\end{itemize}
Given a complex flag manifold $X_{P} = G^{\mathbbm{C}}/P$, the following theorem allows us to describe all $G$-invariant K\"{a}hler structures on $X_{P}$.
\begin{theorem}[Azad-Biswas, \cite{AZAD}]
\label{AZADBISWAS}
Let $\omega \in \Omega^{1,1}(X_{P})^{G}$ be a closed invariant real $(1,1)$-form, then we have

\begin{center}

$\pi^{\ast}\omega = \sqrt{-1}\partial \overline{\partial}\varphi$,

\end{center}
where $\pi \colon G^{\mathbbm{C}} \to X_{P}$ is the projection, and $\varphi \colon G^{\mathbbm{C}} \to \mathbbm{R}$ is given by 
\begin{center}
$\varphi(g) = \displaystyle \sum_{\alpha \in \Sigma \backslash \Theta}c_{\alpha}\log||gv_{\omega_{\alpha}}^{+}||$, \ \ \ \ $(\forall g \in G^\mathbbm{C})$
\end{center}
with $c_{\alpha} \in \mathbbm{R}$, $\forall \alpha \in \Sigma \backslash \Theta$. Conversely, every function $\varphi$ as above defines a closed invariant real $(1,1)$-form $\omega_{\varphi} \in \Omega^{1,1}(X_{P})^{G}$. Moreover, $\omega_{\varphi}$ defines a K\"{a}hler form on $X_{P}$ if and only if $c_{\alpha} > 0$,  $\forall \alpha \in \Sigma \backslash \Theta$.
\end{theorem}

\begin{remark}
\label{innerproduct}
It is worth pointing out that the norm $|| \cdot ||$ in the last theorem is a norm induced from some fixed $G$-invariant inner product $\langle \cdot, \cdot \rangle_{\alpha}$ on $V(\omega_{\alpha})$, for every $\alpha \in \Sigma \backslash \Theta$. 
\end{remark}

By means of the above theorem we can describe the unique $G$-invariant element in each integral class in $H^{2}(X_{P},\mathbbm{Z})$. In fact, consider the associated $P$-principal bundle $P \hookrightarrow G^{\mathbbm{C}} \to X_{P}$. By choosing a trivializing open covering $X_{P} = \bigcup_{i \in I}U_{i}$, in terms of $\check{C}$ech cocycles we can write 
\begin{center}
$G^{\mathbbm{C}} = \Big \{(U_{i})_{i \in I}, \psi_{ij} \colon U_{i} \cap U_{j} \to P \Big \}$.
\end{center}
From above, given a fundamental weight $\omega_{\alpha} \in \Lambda_{\mathbbm{Z}_{\geq 0}}^{\ast}$, by considering $\chi_{\omega_{\alpha}} \in {\text{Hom}}(T^{\mathbbm{C}},\mathbbm{C}^{\times})$, such that $d(\chi_{\omega_{\alpha}})_{e} = \omega_{\alpha}$, we can take the homomorphism $\chi_{\omega_{\alpha}} \colon P \to \mathbbm{C}^{\times}$, defined via holomorphic extension of $\chi_{\omega_{\alpha}}$. From this homomorphism we consider $\mathbbm{C}_{-\omega_{\alpha}}$ as being a $P$-space, such that $pz = \chi_{\omega_{\alpha}}(p)^{-1}z$, $\forall p \in P$, and $\forall z \in \mathbbm{C}$. By considering $\mathbbm{C}_{-\omega_{\alpha}}$ we can form an associated holomorphic line bundle $\mathscr{O}_{\alpha}(1) = G^{\mathbbm{C}} \times_{P}\mathbbm{C}_{-\omega_{\alpha}}$, which can be described in terms of $\check{C}$ech cocycles by
\begin{equation}
\label{linecocycle}
\mathscr{O}_{\alpha}(1) = \Big \{(U_{i})_{i \in I},\chi_{\omega_{\alpha}}^{-1} \circ \psi_{i j} \colon U_{i} \cap U_{j} \to \mathbbm{C}^{\times} \Big \},
\end{equation}
that is, $\mathscr{O}_{\alpha}(1) = \{g_{ij}\} \in \check{H}^{1}(X_{P},\mathscr{O}_{X_{P}}^{\ast})$, such that $g_{ij} = \chi_{\omega_{\alpha}}^{-1} \circ \psi_{i j}$, for every $i,j \in I$. Notice that we can also realize $\mathscr{O}_{\alpha}(1)$ as a quotient space of $G^{\mathbbm{C}} \times \mathbbm{C}_{-\omega_{\alpha}}$ by the equivalence relation $``\sim"$ defined by
\begin{equation}
\label{eqrelationbundle}
(g,z) \sim (h,w) \Longleftrightarrow \exists p \in P, \ {\text{such that}} \ h = gp \ {\text{and}} \ w = p^{-1}z = \chi_{\omega_{\alpha}}(p)z.
\end{equation}
Therefore, for a typical element $u \in \mathscr{O}_{\alpha}(1)$ we have $u = [g,z]$, for some $(g,z) \in G^{\mathbbm{C}} \times \mathbbm{C}_{-\omega_{\alpha}}$.
\begin{remark}
\label{parabolicdec}
We observe that, if we have a parabolic Lie subgroup $P \subset G^{\mathbbm{C}}$, such that $P = P_{\Theta}$, the decomposition 
\begin{equation}
P_{\Theta} = \big[P_{\Theta},P_{\Theta} \big]T(\Sigma \backslash \Theta)^{\mathbbm{C}},
\end{equation}
see for instance \cite[Proposition 8]{Akhiezer}, shows us that ${\text{Hom}}(P,\mathbbm{C}^{\times}) = {\text{Hom}}(T(\Sigma \backslash \Theta)^{\mathbbm{C}},\mathbbm{C}^{\times})$. Therefore, if we take $\omega_{\alpha} \in \Lambda_{\mathbbm{Z}_{\geq 0}}^{\ast}$, such that $\alpha \in \Theta$, it follows that $\mathscr{O}_{\alpha}(1) = X_{P} \times \mathbbm{C}$, i.e., the associated holomorphic line bundle $\mathscr{O}_{\alpha}(1)$ is trivial.
\end{remark}

Given $\mathscr{O}_{\alpha}(1) \in {\text{Pic}}(X_{P})$, such that $\alpha \in \Sigma \backslash \Theta$, as described previously, if we consider an open covering $X_{P} = \bigcup_{i \in I} U_{i}$ which trivializes both $P \hookrightarrow G^{\mathbbm{C}} \to X_{P}$ and $ \mathscr{O}_{\alpha}(1) \to X_{P}$, such that $\alpha \in \Sigma \backslash \Theta$, by taking a collection of local sections $(s_{i})_{i \in I}$, such that $s_{i} \colon U_{i} \to G^{\mathbbm{C}}$, we can define $h_{i} \colon U_{i} \to \mathbbm{R}^{+}$, such that 
\begin{equation}
\label{functionshermitian}
h_{i} =  {\mathrm{e}}^{-2\pi \varphi_{\omega_{\alpha}} \circ s_{i}} = \frac{1}{||s_{i}v_{\omega_{\alpha}}^{+}||^{2}},
\end{equation}
for every $i \in I$. The collection of functions $(h_{i})_{i \in I}$ satisfy $h_{j} = |\chi_{\omega_{\alpha}}^{-1} \circ \psi_{ij}|^{2}h_{i}$ on $U_{i} \cap U_{j} \neq \emptyset$, here we have used that $s_{j} = s_{i}\psi_{ij}$ on $U_{i} \cap U_{j} \neq \emptyset$, and $pv_{\omega_{\alpha}}^{+} = \chi_{\omega_{\alpha}}(p)v_{\omega_{\alpha}}^{+}$, for every $p \in P$, such that $\alpha \in \Sigma \backslash \Theta$. Hence, we have a collection of functions $(h_{i})_{i \in I}$ which satisfies on $U_{i} \cap U_{j} \neq \emptyset$ the following relation
\begin{equation}
\label{collectionofequ}
h_{j} = |g_{ij}|^{2}h_{i},
\end{equation}
such that $g_{ij} = \chi_{\omega_{\alpha}}^{-1} \circ \psi_{i j}$, where $i,j \in I$. From this, we can define a Hermitian structure $H$ on $\mathscr{O}_{\alpha}(1)$ by taking on each trivialization $f_{i} \colon L_{\chi_{\omega_{\alpha}}} \to U_{i} \times \mathbbm{C}$ a metric defined by
\begin{equation}
\label{hermitian}
H(f_{i}^{-1}(x,v),f_{i}^{-1}(x,w)) = h_{i}(x) v\overline{w},
\end{equation}
for all $(x,v),(x,w) \in U_{i} \times \mathbbm{C}$. The Hermitian metric above induces a Chern connection $\nabla = d + \partial \log H$ with curvature $F_{\nabla}$ satisfying 
\begin{equation}
\displaystyle \frac{\sqrt{-1}}{2\pi}F_{\nabla} \Big |_{U_{i}} = \frac{\sqrt{-1}}{2\pi} \partial \overline{\partial}\log \Big ( \big | \big | s_{i}v_{\omega_{\alpha}}^{+}\big | \big |^{2} \Big).
\end{equation}
Hence, by considering the $G$-invariant $(1,1)$-form $\Omega_{\alpha} \in \Omega^{1,1}(X_{P})^{G}$, which satisfies $\pi^{\ast}\Omega_{\alpha} = \sqrt{-1}\partial \overline{\partial} \varphi_{\omega_{\alpha}}$, where $\pi \colon G^{\mathbbm{C}} \to G^{\mathbbm{C}} / P = X_{P}$, and $\varphi_{\omega_{\alpha}}(g) = \frac{1}{2\pi}\log||gv_{\omega_{\alpha}}^{+}||^{2}$, $\forall g \in G^{\mathbbm{C}}$, we have $c_{1}(\mathscr{O}_{\alpha}(1)) = \big [ \Omega_{\alpha}\big]$. From the ideas described above we have the following result.
\begin{proposition}
\label{C8S8.2Sub8.2.3P8.2.6}
Let $X_{P}$ be a complex flag manifold associated to some parabolic Lie subgroup $P = P_{\Theta}\subset G^{\mathbbm{C}}$. Then, we have
\begin{equation}
\label{picardeq}
{\text{Pic}}(X_{P}) = H^{1,1}(X_{P},\mathbbm{Z}) = H^{2}(X_{P},\mathbbm{Z}) = \displaystyle \bigoplus_{\alpha \in \Sigma \backslash \Theta}\mathbbm{Z}\big [\Omega_{\alpha} \big ].
\end{equation}
\end{proposition}
\begin{proof}

Let us sketch the proof. The last equality on the right-hand side of Eq. ($\ref{picardeq}$) follows from the following facts:

\begin{itemize}

\item $\pi_{2}(X_{P}) \cong \pi_{1}(T(\Sigma \backslash \Theta)^{\mathbbm{C}}) = \mathbbm{Z}^{|\Sigma \backslash \Theta|}$, where

$$T(\Sigma \backslash \Theta)^{\mathbbm{C}} = \exp \Big \{ \displaystyle \sum_{\alpha \in  \Sigma \backslash \Theta}a_{\alpha}h_{\alpha} \ \Big | \ a_{\alpha} \in \mathbbm{C} \Big \};$$

\item Since $X_{P}$ is simply connected, it follows that $H_{2}(X_{P},\mathbbm{Z}) \cong \pi_{2}(X_{P})$ (Hurewicz's theorem);

\item By taking $\mathbbm{P}_{\alpha}^{1} \hookrightarrow X_{P}$, such that 

$$\mathbbm{P}_{\alpha}^{1} = \overline{\exp(\mathfrak{g}_{-\alpha})x_{0}} \subset X_{P},$$ 
for all $\alpha \in \Sigma \backslash \Theta$, where $x_{0} = eP \in X_{P}$, it follows that 

\begin{center}

$\Big \langle c_{1}(\mathscr{O}_{\alpha}(1)), \big [ \mathbbm{P}_{\beta}^{1}\big] \Big \rangle = \displaystyle \int_{\mathbbm{P}_{\beta}^{1}} c_{1}(\mathscr{O}_{\alpha}(1)) = \delta_{\alpha \beta},$

\end{center}
for every $\alpha,\beta \in \Sigma \backslash \Theta$. Hence, we obtain
\begin{center}

$\pi_{2}(X_{P}) = \displaystyle \bigoplus_{\alpha \in \Sigma \backslash \Theta} \mathbbm{Z}\big [ \mathbbm{P}_{\alpha}^{1}\big],$ \ \ and \ \ $H^{2}(X_{P},\mathbbm{Z}) = \displaystyle \bigoplus_{\alpha \in \Sigma \backslash \Theta}  \mathbbm{Z} c_{1}(\mathscr{O}_{\alpha}(1))$.

\end{center}
\end{itemize}
Moreover, form above we also have $H^{1,1}(X_{P},\mathbbm{Z}) = H^{2}(X_{P},\mathbbm{Z})$. In order to conclude the proof, from the Lefschetz theorem on (1,1)-classes \cite{DANIEL}, and from the fact that ${\text{rk}}({\text{Pic}}^{0}(X_{P})) = 0$, we obtain the first equality in Eq. (\ref{picardeq}).
\end{proof}
From Theorem \ref{AZADBISWAS} and Proposition \ref{C8S8.2Sub8.2.3P8.2.6}, given a negative line bundle $L \in {\text{Pic}}(X_{P})$, we have
\begin{equation}
\label{negativedecomp}
L = \bigotimes_{\alpha \in \Sigma \backslash \Theta}\mathscr{O}_{\alpha}(1)^{\otimes \langle c_{1}(L),[ \mathbbm{P}_{\alpha}^{1}] \rangle},
\end{equation}
where $\big \langle c_{1}(L),\big [ \mathbbm{P}_{\alpha}^{1} \big] \big \rangle < 0$, $\forall \alpha \in \Sigma \backslash \Theta$. For the sake of simplicity, we denote $\ell_{\alpha} = - \big \langle c_{1}(L),\big [ \mathbbm{P}_{\alpha}^{1} \big] \big \rangle$, and $\mathscr{O}_{\alpha}(1)^{\otimes k} = \mathscr{O}_{\alpha}(k)$, for every $k \in \mathbbm{Z}$, and every $\alpha \in \Sigma \backslash \Theta$. Thus, we can rewrite Eq. (\ref{negativedecomp}) as follows
\begin{equation}
L = \bigotimes_{\alpha \in \Sigma \backslash \Theta}\mathscr{O}_{\alpha}(-\ell_{\alpha}).
\end{equation}
From the characterization above, if we consider the weight $\mu(L) \in  \Lambda_{\mathbbm{Z}_{\geq 0}}^{\ast}$ defined by
\begin{equation}
\label{negativeweight}
\mu(L) := - \sum_{\alpha \in \Sigma \backslash \Theta} \big \langle c_{1}(L),\big [ \mathbbm{P}_{\alpha}^{1} \big] \big \rangle \omega_{\alpha} = \sum_{\alpha \in \Sigma \backslash \Theta} \ell_{\alpha} \omega_{\alpha},
\end{equation}
we can assign to each negative line bundle $L \in {\text{Pic}}(X_{P})$ an irreducible $\mathfrak{g}^{\mathbbm{C}}$-module $V(\mu(L))$ with highest weight vector $v_{\mu(L)}^{+} \in V(\mu(L))$, such that
\begin{equation}
\displaystyle V(\mu(L)) \subset \bigotimes_{\alpha \in \Sigma \backslash \Theta}V(\omega_{\alpha})^{\otimes \ell_{\alpha}}, \ \ {\text{and}} \ \ \displaystyle v_{\mu(L)}^{+} = \bigotimes_{\alpha \in \Sigma \backslash \Theta}(v_{\omega_{\alpha}}^{+})^{\otimes \ell_{\alpha}},
\end{equation}
see for instance \cite[p. 186]{PARABOLICTHEORY}. Also, notice that, associated to the weight defined in Eq. (\ref{negativeweight}) we have a character $\chi^{(L)} \colon P \to \mathbbm{C}^{\times}$, such that
\begin{equation}
\label{negativecharacter}
\chi^{(L)} = \prod_{\alpha \in \Sigma \backslash \Theta}\chi_{\omega_{\alpha}}^{\ell_{\alpha}}.
\end{equation}
Thus, in terms of $\check{C}$ech cocycles we have $L = \{g_{ij}\} \in \check{H}^{1}(X_{P},\mathscr{O}_{X_{P}}^{\ast})$, with $g_{ij} = \chi^{(L)} \circ \psi_{i j}$, $\forall i,j \in I$, cf. Eq. (\ref{linecocycle}).
\begin{remark}[Projective embedding]
\label{proj}
In the setting above, the ample line bundle $L^{-1} \in {\text{Pic}}(X_{P})$ is in fact very ample, i.e., we have a projective embedding 
\begin{center}
$\iota \colon X_{P} \hookrightarrow \mathbbm{P}(V(\mu(L))) = {\text{Proj}}\big (H^{0}(X_{P},L^{-1})^{\ast} \big),$
\end{center}
see for instance \cite[Page 193]{Flaginterplay}, \cite[Theorem 3.2.8]{PARABOLICTHEORY}, \cite{TAYLOR}. Thus, we have the identification
\begin{equation}
\label{tautological}
L \cong \iota^{\ast}\mathscr{O}_{\mathbbm{P}(V(\mu(L))}(-1) = \Big \{ \big ([x],v \big ) \in X_{P} \times V(\mu(L)) \ \ \Big | \ \ v \in \langle x \rangle_{\mathbbm{C}}\Big \},
\end{equation}
where $\mathscr{O}_{\mathbbm{P}(V(\mu(L))}(-1)$ is the tautological line bundle over $\mathbbm{P}(V(\mu(L))$, here we consider the geometric realization $V(\mu(L)) = H^{0}(X_{P},L^{-1})^{\ast}$ via Borel-Weil theorem.
\end{remark}

\subsection{Proof of Theorem 1} In what follows, we restate and prove our first result.

\begin{theorem}
\label{Maintheo1}
Let $X_{P}$ be a complex flag manifold, associated to some parabolic Lie subgroup $P = P_{\Theta} \subset G^{\mathbbm{C}}$, and let $L \in {\text{Pic}}(X_{P})$ be a negative line bundle. Then, for every $\lambda \in \mathbbm{C}^{\times}$, such that $|\lambda| <1$, we have that the $T_{\mathbbm{C}}^{1}$-principal bundle
\begin{equation}
\label{quotient}
M = {\rm{Tot}}(L^{\times})/\Gamma_{\lambda}, \ \ {\text{such that}}  \ \ \Gamma_{\lambda} = \big \{ \lambda^{n} \in \mathbbm{C}^{\times} \ \big | \ n \in \mathbbm{Z}  \big \},
\end{equation}
admits a Vaisman structure completely determined by the $T_{\mathbbm{C}}^{1}$-principal connection $\Psi \in \Omega^{1}(M;\mathbbm{C})$ locally described by
\begin{equation}
\label{T1connection}
\displaystyle \Psi = \frac{1}{\sqrt{-1}}\bigg [\partial \log \Big ( \big | \big |s_{U}v_{\mu(L)}^{+} \big | \big |^{2}\Big ) + \frac{dw}{w}\bigg],
\end{equation}
for some local section $s_{U} \colon U \subset X_{P} \to G^{\mathbbm{C}}$, where $v_{\mu(L)}^{+}$ is the highest weight vector of weight $\mu(L)$ associated to the irreducible $\mathfrak{g}^{\mathbbm{C}}$-module $V(\mu(L)) = H^{0}(X_{P},L^{-1})^{\ast}$.
\end{theorem}

\begin{proof}
At first, by fixing a suitable trivializing open covering $X_{P} = \bigcup_{i \in I}U_{i}$, and considering the characterization 
\begin{center}
$L = \Big \{(U_{i})_{i \in I},g_{ij}:=\chi^{(L)} \circ \psi_{i j} \colon U_{i} \cap U_{j} \to \mathbbm{C}^{\times} \Big \},$
\end{center}
we can define a Hermitian structure on $L \to X_{P}$ by setting
\begin{equation}
\label{hermiotianinvariant}
\displaystyle H \big (f_{i}^{-1}(z,v),f_{i}^{-1}(z,w) \big ) = \Big ( \prod_{\alpha \in \Sigma \backslash \Theta}\big | \big |s_{i}(z)v_{\omega_{\alpha}}^{+} \big | \big |^{2\ell_{\alpha}}\Big)v\overline{w},
\end{equation}
for any local trivialization $(L|_{U_{i}},f_{i})$ and any local section $s_{i} \colon U_{i} \to G^{\mathbbm{C}}$, $\forall i \in I$. Since we have $s_{j} = s_{i}\psi_{ij}$ on $U_{i} \cap U_{j} \neq \emptyset$, $(f_{i} \circ f_{j}^{-1})(z,w) = (z,g_{ij}(z)w)$ on $U_{i} \cap U_{j} \neq \emptyset$, and $pv_{\omega_{\alpha}}^{+} = \chi_{\omega_{\alpha}}(p)v_{\omega_{\alpha}}^{+}$, for every $p \in P$, such that $\alpha \in \Sigma \backslash \Theta$, it follows from Eq. (\ref{negativecharacter}) that 
\begin{center}
$\displaystyle H \big (f_{j}^{-1}(z,v),f_{j}^{-1}(z,w) \big ) = \big |\chi^{(L)}(\psi_{ij}(z)) \big |^{2}\Big ( \prod_{\alpha \in \Sigma \backslash \Theta}\big | \big |s_{i}(z)v_{\omega_{\alpha}}^{+} \big | \big |^{2\ell_{\alpha}}\Big)v\overline{w} = H \big (f_{i}^{-1}(z,g_{ij}(z)v),f_{i}^{-1}(z,g_{ij}(z)w) \big ),$
\end{center}
$\forall (z,w),(z,v) \in U_{i}\cap U_{j} \times \mathbbm{C}$. Hence, from the Hermitian structure above we can define a K\"{a}hler potential ${\rm{K}}_{H} \colon {\text{Tot}}(L^{\times}) \to \mathbbm{R}^{+}$, by setting ${\rm{K}}_{H}(u) = H(u,u)$. From this, locally on $L^{\times}|_{U_{i}}$ we have
\begin{equation}
\displaystyle {\rm{K}}_{H}(z,w) = {\rm{K}}_{H}(f_{i}^{-1}(z,w)) = \Big ( \prod_{\alpha \in \Sigma \backslash \Theta}\big | \big |s_{i}(z)v_{\omega_{\alpha}}^{+} \big | \big |^{2\ell_{\alpha}}\Big)w\overline{w}. 
\end{equation}
In order to obtain Eq. (\ref{T1connection}), we proceed as follows: Since we have
\begin{equation}
\displaystyle V(\mu(L)) \subset \bigotimes_{\alpha \in \Sigma \backslash \Theta}V(\omega_{\alpha})^{\otimes \ell_{\alpha}}, \ \ {\text{and}} \ \ \displaystyle v_{\mu(L)}^{+} = \bigotimes_{\alpha \in \Sigma \backslash \Theta}(v_{\omega_{\alpha}}^{+})^{\otimes \ell_{\alpha}},
\end{equation}
we can take a $G$-invariant inner product on $V(\mu(L))$ induced from a $G$-invariant inner product $\langle \cdot, \cdot \rangle_{\alpha}$ on each factor $V(\omega_{\alpha})$, $\forall \alpha \in \Sigma \backslash \Theta$, such that\footnote{See for instance \cite{Greub}.} 
\begin{equation}
\label{innerinv}
\Big \langle  \bigotimes_{\alpha \in \Sigma \backslash \Theta} (v_{1}^{(\alpha)} \otimes \cdots \otimes v_{\ell_{\alpha}}^{(\alpha)}),  \bigotimes_{\alpha \in \Sigma \backslash \Theta} (w_{1}^{(\alpha)} \otimes \cdots \otimes w_{\ell_{\alpha}}^{(\alpha)}) \Big \rangle = \prod_{ \alpha \in \Sigma \backslash \Theta}\big \langle v_{1}^{(\alpha)},w_{1}^{(\alpha)} \big \rangle_{\alpha}\cdots \big \langle v_{\ell_{\alpha}}^{(\alpha)},w_{\ell_{\alpha}}^{(\alpha)} \big \rangle_{\alpha},
\end{equation}
see Remark \ref{innerproduct}. Considering the norm on $V(\mu(L))$ induced by the inner product above, we can rewrite
\begin{equation}
\label{Kpotentialcone}
\displaystyle {\rm{K}}_{H}\big (z,w \big ) = \Big ( \prod_{\alpha \in \Sigma \backslash \Theta}\big | \big |s_{i}(z)v_{\omega_{\alpha}}^{+} \big | \big |^{2\ell_{\alpha}}\Big)w\overline{w} = \Big (\big | \big |s_{i}(z)v_{\mu(L)}^{+} \big | \big |^{2} \Big)  w\overline{w}.
\end{equation}
From the K\"{a}hler potential above we have a $\mathbbm{C}^{\times}$-principal connection $\widetilde{\Psi} \in \Omega^{1}(L^{\times};\mathbbm{C})$, given by
\begin{center}
$\widetilde{\Psi} = \frac{1}{2}\big (d^{c} - \sqrt{-1}d\big )\log({\rm{K}}_{H}) = -\sqrt{-1} \partial \log({\rm{K}}_{H})$,
\end{center}
Hence, regarding $M = \big (L^{\times} \times \mathbbm{E}(\Lambda) \big) / \mathbbm{C}^{\times}$ as a principal elliptic bundle over $X_{P}$, such that $\mathbbm{E}(\Lambda) = \mathbbm{C}^{\times}/\Gamma_{\lambda}$, we have from $\widetilde{\Psi}$ a unique induced $T_{\mathbbm{C}}^{1}$-principal connection $\Psi \in \Omega^{1}(M;\mathbbm{C})$, locally described as in Eq. (\ref{T1connection}). From this, the result follows from Theorem \ref{vaismanbundle} and the computations presented in Section \ref{mainremark}.

\end{proof}
\begin{remark}
\label{metricdescription}
In the setting of Theorem \ref{Maintheo1}, we obtain an explicit description for the Lee form $\theta \in \Omega^{1}(M)$ associated to the Vaismans structure $(M,J,g)$ induced by $\Psi \in \Omega^{1}(M;\mathbbm{C})$ in terms of representation theory. In fact, from Eq. (\ref{T1connection}) a straightforward computation shows us that 
\begin{equation}
\label{leelocal}
\theta = -\bigg [d\log \Big ( \big | \big |s_{i}v_{\mu(L)}^{+} \big | \big |^{2}\Big ) + \frac{\overline{w}dw + wd\overline{w}}{|w|^{2}}\bigg],
\end{equation}
see for instance Eq. (\ref{globalpotlee}). Notice that we can also describe the anti-Lee form $\vartheta = -\theta \circ J \in \Omega^{1}(M)$ by
\begin{equation}
 \vartheta = -\bigg [d^{c}\log \Big ( \big | \big |s_{i}v_{\mu(L)}^{+} \big | \big |^{2}\Big )- \frac{\sqrt{-1}}{|w|^{2}}\big (\overline{w}dw - wd\overline{w} \big) \bigg ],
\end{equation}
see for instance Eq. (\ref{connectiontorus}), and Eq. (\ref{connectionlee}). In conclusion, Theorem \ref{Maintheo1} provides a systematic and constructive way to obtain explicit examples of Vaisman structures using elements of representation theory of semisimple Lie algebras.
\end{remark}

\subsection{Proof of Theorem 2}

Combining the previous theorem with the result provided in \cite{Gauduchon1}, we obtain the following theorem.

\begin{theorem}
\label{Theo1}
Let $(M,[g],J)$ be a compact homogeneous l.c.K. manifold, and let $K$ be the compact connected Lie group which acts on $M$ by preserving the l.c.K. structure. Suppose that the semisimple Lie subgroup $K_{{\text{ss}}}$ is simply connected and has a unique simple component. Then, the Lee form associated to the l.c.K. structure $([g],J)$ is (up to scale) completely determined by the $1$-form
\begin{equation}
\displaystyle \theta = -\bigg [d\log \Big ( \big | \big |s_{U}v_{\mu(L)}^{+} \big | \big |^{2}\Big ) + \frac{\overline{w}dw + wd\overline{w}}{|w|^{2}}\bigg],
\end{equation}
such that $s_{U} \colon U \subset X_{P} \to K_{\text{ss}}^{\mathbbm{C}}$ is some local section, $v_{\mu(L)}^{+}$ is the highest weight vector of weight $\mu(L)$ associated to the irreducible $\mathfrak{k}_{\text{ss}}^{\mathbbm{C}}$-module $V(\mu(L)) = H^{0}(X_{P},L^{-1})^{\ast}$, and $L \in {\text{Pic}}(X_{P})$ is a negative line bundle.
\end{theorem}

\begin{proof}
The result above can be obtained as follows. At first, we notice that from Theorem \ref{LCKisVaisman}, Theorem \ref{Vaismanregular}, and Theorem \ref{vaismanclass}, it follows that $M$ is a $T_{\mathbbm{C}}^{1}$-principal bundle over a complex flag manifold $X_{P} = K_{{\text{ss}}}^{\mathbbm{C}}/P$. Since $X_{P}$ has no torsion elements in $H^{2}(X_{P},\mathbbm{Z})$, it follows that there exists a negative line bundle $L \in {\text{Pic}}(X_{P})$ such that
\begin{center}
$M = {\rm{Tot}}(L^{\times})/\Gamma, \ \ {\text{where}}  \ \ \Gamma = \big \{ \lambda^{n} \in \mathbbm{C}^{\times} \ \big | \ n \in \mathbbm{Z}  \big \},$
\end{center}
for some $\lambda \in \mathbbm{C}^{\times}$, such that $|\lambda| <1$, see for instance Theorem \ref{vaismanclass}, and \cite[Theorem 4.2]{Tsukada}. Since  
\begin{equation}
\label{bettiharmonic}
\dim\mathscr{H}^{1}(M) = b_{1}(M) = 1,
\end{equation}
where $\mathscr{H}^{1}(M)$ is the space of harmonic 1-forms, e.g. \cite[Remark 6.2]{Dragomir}, it follows that $H_{DR}^{1}(M) = \mathbbm{R}[\theta]$, such that $\theta \in \Omega^{1}(M)$ is given by Eq. (\ref{leelocal}). Hence, we have that $\theta_{g} = c_{0}\theta + df$, for some $f \in C^{\infty}(M)$ and some real constant $c_{0} \in \mathbbm{R}$. After a conformal change $g \mapsto h = {\rm{e}}^{-f}g$, we obtain $\theta_{g} \mapsto \theta_{h} = \theta_{g} -df$. Therefore, in order to conclude the proof we just need to show that $df \equiv 0$. This last fact is a consequence of the following: since $\theta_{h} = c_{0}\theta$, we can define a $T_{\mathbbm{C}}^{1}$-principal connection $\Psi_{h} \in \Omega^{1}(M;\mathbbm{C})$, such that 
\begin{equation}
\displaystyle \Psi_{h} = \frac{c_{0}}{\sqrt{-1}}\bigg [\partial \log \Big ( \big | \big |s_{i}v_{\mu(L)}^{+} \big | \big |^{2}\Big ) + \frac{dw}{w}\bigg],
\end{equation}
notice that $\Psi_{h} = -\vartheta_{h} + \sqrt{-1}\theta_{h}$, where $\vartheta_{h} = -\theta_{h} \circ J$. On the other hand, we have a connection on $M$ induced by the l.c.K. structure $(\theta_{g},[g],J)$, such that $\Psi_{g} = -\vartheta_{g} + \sqrt{-1}\theta_{g}$. Hence, we obtain
\begin{center}
$\Psi_{g} = \Psi_{h} + d^{c}f + \sqrt{-1}df.$
\end{center}
Since both $d\Psi_{g}$ and $d\Psi_{h}$ descend to a real $K_{{\text{ss}}}$-invariant $(1,1)$-form on $X_{P}$, from the uniqueness of $K_{{\text{ss}}}$-invariant representative for the Chern class of the $T_{\mathbbm{C}}^{1}$-principal bundle $M$ over $X_{P}$, we conclude that $d\Psi_{g} = d\Psi_{h}$, i.e., $\sqrt{-1} \partial \overline{\partial} f = \frac{1}{2}dd^{c}f = 0$. Since $M$ is a connected compact manifold, it follows that $f$ is constant, so $df \equiv 0$.
\end{proof}

\subsection{Proof of Theorem 3}
Let us start by recalling some basic facts about K\"{a}hler-Einstein metrics on flag manifolds.

\begin{remark}[K\"{a}hler-Einstein metrics on flag manifolds]In the context of complex flag manifolds, the anticanonical line bundle can be described as follows. Let $X_{P}$ be a complex flag manifold associate to some parabolic Lie subgroup $P = P_{\Theta} \subset G^{\mathbbm{C}}$. By considering the identification

\begin{center}

$\mathfrak{m} = \displaystyle \sum_{\alpha \in \Pi^{+} \backslash \langle \Theta \rangle^{+}} \mathfrak{g}_{-\alpha} = T_{x_{0}}^{1,0}X_{P}$,
\end{center}
where $x_{0} = eP \in X_{P}$. We have the following characterization for $T^{1,0}X_{P}$ as an associated holomoprphic vector bundle to the $P$-principal bundle $P \hookrightarrow G^{\mathbbm{C}} \to X_{P}$

\begin{center}

$T^{1,0}X_{P} = G^{\mathbbm{C}} \times_{P} \mathfrak{m}$,

\end{center}
such that the twisted product on the right-hand side above is obtained from the isotropy representation ${\rm{Ad}} \colon P \to {\rm{GL}}(\mathfrak{m})$. From this, we observe the following facts:
\begin{enumerate}
\item Since $P = [P,P]T(\Sigma \backslash \Theta)^{\mathbbm{C}}$, $\forall g \in P$ we have $g = g'\exp(h)$, for some $g' \in [P,P]$ and some $h \in {\text{Lie}}(T(\Sigma \backslash \Theta)^{\mathbbm{C}})$;

\item From the decomposition above, we obtain 
\begin{center}
$\det({\rm{Ad}}(g)) = \det({\rm{Ad}}(g'))\det({\rm{Ad}}(\exp(h))) = {\rm{e}}^{{\rm{Tr}}({\rm{ad}}(h))}$,
\end{center}
here we have ${\rm{ad}}(h) \colon \mathfrak{m} \to  \mathfrak{m}$, such that ${\rm{ad}}(h)(v + \mathfrak{p}_{\Theta}) := {\rm{ad}}(h)(v) + \mathfrak{p}_{\Theta}$, for all $v + \mathfrak{p}_{\Theta} \in \mathfrak{m} = \mathfrak{g}^{\mathbbm{C}}/\mathfrak{p}_{\Theta}$;

\item By definition of $\mathfrak{m}$, it follows that 
\begin{center}
$\displaystyle {\rm{Tr}}({\rm{ad}}(h)) = - \sum_{\alpha \in \Pi^{+} \backslash \langle \Theta \rangle^{+}} \alpha(h) = - \delta_{P}(h)$.
\end{center}
Thus, we conclude that $\det({\rm{Ad}}(g)) = \chi_{\delta_{P}}^{-1}(g)$, $\forall g \in P$, so $\det \circ {\rm{Ad}} = \chi_{\delta_{P}}^{-1}$.
\end{enumerate}
From the facts listed above, it follows that 
\begin{equation}
\label{canonicalbundleflag}
K_{X_{P}}^{-1} = \det \big(T^{1,0}X_{P} \big) =\det \big ( G^{\mathbbm{C}} \times_{P} \mathfrak{m} \big )= L_{\chi_{\delta_{P}}}.
\end{equation}
Also, we have 
\begin{equation}
\chi_{\delta_{P}} = \displaystyle \prod_{\alpha \in \Sigma \backslash \Theta} \chi_{\omega_{\alpha}}^{\langle \delta_{P},h_{\alpha}^{\vee} \rangle} \Longrightarrow K_{X_{P}} = \bigotimes_{\alpha \in \Sigma \backslash \Theta}\mathscr{O}_{\alpha}(-\ell_{\alpha}),
\end{equation}
such that $\ell_{\alpha} = \langle \delta_{P}, h_{\alpha}^{\vee} \rangle, \forall \alpha \in \Sigma \backslash \Theta$. Notice that, from the description above, the Fano index of a complex flag manifold $X_{P}$ is given explicitly by
\begin{equation}   
I(X_{P}) = {\text{gcd}} \Big (  \langle \delta_{P}, h_{\alpha}^{\vee} \rangle \ \Big | \ \alpha \in \Sigma \backslash \Theta \Big ),    
\end{equation}
Thus, $I(X_{P})$ can be completely determined from the Cartan matrix of $\mathfrak{g}^{\mathbbm{C}}$.

Given a complex flag manifold $X_{P}$, associated to some parabolic Lie subgroup $P \subset G^{\mathbbm{C}}$, we can consider the invariant K\"{a}hler metric $\rho_{0} \in \Omega^{1,1}(X_{P})^{G}$, locally describe by
\begin{equation}
\label{riccinorm}
\rho_{0}|_{U} = \sqrt{-1}\partial \overline{\partial} \log \Big (\big | \big |s_{U}v_{\delta_{P}}^{+} \big| \big |^{2} \Big ),
\end{equation}
for some local section $s_{U} \colon U \subset X_{P} \to G^{\mathbbm{C}}$. In the description above we have used that $\mu(K_{X_{P}}) = \delta_{P}$ (cf. Eq. (\ref{negativeweight})). It is straightforward to see that 
\begin{center}
$\displaystyle c_{1}(X_{P}) = \Big [ \frac{\rho_{0}}{2\pi}\Big]$,
\end{center}
and by the uniqueness of $G$-invariant representative for $c_{1}(X_{P})$, see for instance \cite[page 13]{Algmodels}, it follows that 
\begin{center}
${\text{Ric}}^{\nabla}(\rho_{0}) = \rho_{0}$, 
\end{center}
i.e. $\rho_{0} \in \Omega^{1,1}(X_{P})^{G}$ defines a K\"{a}hler-Einstein metric (see \cite{MATSUSHIMA} for more details). 
\end{remark}

As we shall see, the K\"{a}hler-Einstein structure described in the previous remark (Eq. (\ref{riccinorm})) will play an important role in the proof of our next theorem. We also will need the following ideas.

\begin{remark}[Transversal K\"{a}hler-Einstein metrics] 
\label{homothetic}
Given a Sasaki manifold $(Q,g)$, such that $\dim_{\mathbbm{R}}(Q) = 2n+1$, with structure tensors $(\phi,\xi,\eta)$, since $Q$ is a ${\text{K}}$-contact manifold, we have that $\frac{d\eta}{2}$ defines a symplectic form on the distribution $\mathscr{D} = \ker(\eta)$, which satisfies $\mathscr{L}_{\xi}(\frac{d\eta}{2}) = 0$. Let us denote the induced transversal Riemannian metric on $\mathscr{D}$ by
\begin{equation}
g^{T} := \frac{1}{2}d\eta({\rm{id}}\otimes \phi),
\end{equation}
notice that $J = \phi|_{\mathscr{D}}$ defines an almost complex structure on $\mathscr{D} = \ker(\eta)$ which is in fact integrable, since $(\phi,\xi,\eta)$ is normal, thus $\omega^{T} := \frac{d\eta}{2}$ is a K\"{a}hler form. The relationship between the Ricci curvature of $g^{T}$ and $g$ is given by the following identities:

\begin{enumerate}

\item ${\text{Ric}}_{g}(X,Y) = {\text{Ric}}_{g^{T}}(X,Y) - 2g(X,Y)$, $\forall X,Y \in \mathscr{D} = \ker{\eta}$;

\item ${\text{Ric}}_{g}(X,\xi) = 2n\eta(X)$, $\forall X \in TQ;$
\end{enumerate}
see for instance \cite[Theorem 7.3.12]{BoyerGalicki}. In the setting above, if we have ${\text{Ric}}(g^{T}) = \lambda_{0}g^{T}$, for some real constant $\lambda_{0} > 0$, we can take $a > 0$, such that
\begin{equation}
a = \frac{\lambda_{0}}{2(n+1)},
\end{equation}
and define
\begin{equation}
g_{a} = ag + (a^{2} - a)\eta \otimes \eta, \ \ \eta' = a \eta, \ \ \xi' = \frac{1}{a}\xi,  \ \ \phi' = \phi.
\end{equation}
The change above is called $\mathscr{D}$-{\textit{homothetic transformation}} \cite{Tanno}, and the resulting structure tensors $(g_{a},\phi',\xi',\eta')$ also define a Sasaki structure on $Q$. Furthermore, it is straightforward to see that $g_{a}^{T} = ag^{T}$. Thus, since ${\text{Ric}}(ag^{T}) = {\text{Ric}}(g^{T})$, from the previous identity (1), we have 
\begin{center}
${\text{Ric}}_{g_{a}}(X,Y) = {\text{Ric}}_{g_{a}^{T}}(X,Y) - 2g_{a}(X,Y) = 2(n+1)g_{a}^{T}(X,Y) - 2g_{a}(X,Y) = 2ng_{a}(X,Y),$
\end{center}
$\forall X,Y \in \mathscr{D} = \ker{\eta}$, here, in the last equality on the right-hand side above, we have used that $g_{a}|_{\mathscr{D}} = g_{a}^{T}$. Since the previous identity (2) remains essentially the same, namely, ${\text{Ric}}_{g_{a}}(X,\xi') = 2n\eta'(X)$, $\forall X \in TQ$, we conclude that ${\text{Ric}}(g_{a}) = 2ng_{a}$. Thus, we have that the structure tensors $(g_{a},\phi',\xi',\eta')$ define a Sasaki-Einstein structure on $Q$. For more details, see \cite{BOYER}.
\end{remark}

Now we are able to prove the following result.

\begin{theorem}
\label{Theo2}
Let $(M,[g],D,J)$ be a compact homogeneous Hermitian-Einstein-Weyl manifold, such that $\dim_{\mathbbm{R}}(M) \geq 6$, and let $K$ be the compact connected Lie group which acts on $M$ by preserving the Hermitian-Einstein-Weyl structure. Suppose also that $K_{{\text{ss}}}$ is simply connected and has a unique simple component. Then, the Hermitian-Einstein-Weyl metric $g$ is completely determined by the Lee form $\theta_{g} \in \Omega^{1}(M)$, locally described by
\begin{equation}
\label{higgsfield1}
\theta_{g} = -\bigg [ \frac{\ell}{I(X_{P})}d\log \Big ( \big | \big |s_{U}v_{\delta_{P}}^{+} \big | \big |^{2}\Big ) +  \frac{\overline{w}dw + wd\overline{w}}{|w|^{2}} \bigg],
\end{equation}
for some $\ell \in \mathbbm{Z}_{>0}$, such that $s_{U} \colon U \subset X_{P} \to K_{\text{ss}}^{\mathbbm{C}}$ is some local section, $v_{\delta_{P}}^{+}$ is the highest weight vector of weight $\delta_{P}$ associated to the irreducible $\mathfrak{k}_{\text{ss}}^{\mathbbm{C}}$-module $V(\delta_{P}) = H^{0}(X_{P},K_{X_{P}}^{-1})^{\ast}$, and $I(X_{P})$ is the Fano index of $X_{P}$. 

\end{theorem}

\begin{proof}
Since $(M,[g],D,J)$ is a compact homogeneous l.c.K. manifold, from Theorem \ref{LCKisVaisman}, and Theorem \ref{Vaismanregular}, it follows that $M$ is a $T_{\mathbbm{C}}^{1}$-principal bundle over a complex flag manifold $X_{P} = K_{{\text{ss}}}^{\mathbbm{C}}/P$, for some parabolic Lie subgroup $P = P_{\Theta} \subset K_{{\text{ss}}}^{\mathbbm{C}}$. Moreover, it follows also from Theorem \ref{Vaismanregular} that $M$ is a flat $S^{1}$-principal bundle over a compact homogeneous Sasaki manifold $Q$. From Theorem \ref{HEW}, we have that $M = {\rm{Tot}}(\mathscr{O}_{X_{P}}(-\ell)^{\times})/\Gamma$, for some $\ell \in \mathbbm{Z}_{>0}$, which implies that ${\rm{e}}(Q) = - \frac{\ell}{I(X_{P})} c_{1}(X_{P})$. Hence, we have that $Q$ is a compact homogeneous Sasaki-Einstein manifold, see for instance \cite{BOYER}, \cite{Sparks}. In particular, we have $\pi_{1}(Q) = \mathbbm{Z}_{\ell}$, and $Q = \widetilde{Q}/\mathbbm{Z}_{\ell}$, where $\widetilde{Q} = Q(\mathscr{O}_{X_{P}}(-1))$ is the associated universal covering space. From this, let us suppose at first that $\pi_{1}(Q)$ is trivial ($\ell = 1$). In this case, we have $M = \widetilde{Q} \times S^{1}$, and the associated presentation for $M$ is given by
\begin{equation}
M = \big ({\rm{Tot}}(\mathscr{O}_{X_{P}}(-1)^{\times}),\mathbbm{Z} \big).
\end{equation}
Under the considerations above, it follows from  Theorem \ref{HEW} that the Hermitian-Einstein-Weyl structure on $M$ is obtained from the K\"{a}hler Ricci-flat structure on the cone ${\rm{Tot}}(\mathscr{O}_{X_{P}}(-1)^{\times})$. In order to explicitly describe this  Hermitian-Einstein-Weyl structure, we observe the following. Since in this case we have
\begin{center}
$\displaystyle \mathscr{O}_{X_{P}}(-1) = \bigotimes_{\alpha \in \Sigma \backslash \Theta}\mathscr{O}_{\alpha}(-1)^{\otimes \frac{\langle \delta_{P}, h_{\alpha}^{\vee}\rangle}{I(X_{P})}},$
\end{center}
we can consider the K\"{a}hler potential ${\rm{K}}_{H} \colon \displaystyle \mathscr{O}_{X_{P}}(-1)^{\times} \to \mathbbm{R}^{+}$, described locally by 
\begin{equation}
\label{Kpotencialcanonical}
\displaystyle {\rm{K}}_{H}\big (z,w \big ) = \Big ( \prod_{\alpha \in \Sigma \backslash \Theta}\big | \big |s_{i}(z)v_{\omega_{\alpha}}^{+} \big | \big |^{\frac{2\langle \delta_{P}, h_{\alpha}^{\vee}\rangle}{I(X_{P})}}\Big)w\overline{w} = \Big (\big | \big |s_{i}(z)v_{\delta_{P}}^{+} \big | \big |^{\frac{2}{I(X_{P})}} \Big)  w\overline{w},
\end{equation}
cf. Eq. (\ref{Kpotentialcone}). From the K\"{a}hler potential $r^{2} = {\rm{K}}_{H}$ above we obtain a K\"{a}hler form
\begin{center}
$\displaystyle \omega_{\mathscr{C}} = \frac{\sqrt{-1}}{2} \partial \overline{\partial} {\rm{K}}_{H} = d\Big(\frac{r^{2}\eta}{2}\Big) = {\mathrm{e}}^{2\psi} \bigg ( d\psi \wedge \eta + \frac{d\eta}{2}\bigg)$,
\end{center}
such that $\psi = \log(r)$, and
\begin{center}
$\displaystyle \eta = \frac{1}{2} d^{c}\log({\rm{K}}_{H}(z,w)) = \frac{1}{2I(X_{P})}d^{c}\log \Big( \big | \big |s_{i}(z)v_{\delta_{P}}^{+} \big | \big |^{2}\Big) + d\sigma_{i}$.
\end{center}
Notice that $d\eta = \pi^{\ast}\big(\frac{\rho_{0}}{I(X_{P})}\big)$, where $\pi \colon Q(\mathscr{O}_{X_{P}}(-1) ) \to X_{P}$ denotes the associated bundle projection, see Eq. (\ref{riccinorm}). From above, consider the Sasaki structure $(g_{S},\phi,\xi,\eta)$ on $Q(\mathscr{O}_{X_{P}}(-1) )$, obtained as in Eq. (\ref{sasakistructure}), where the Sasaki metric $g_{S}$ on $Q(\mathscr{O}_{X_{P}}(-1) )$ is given by
\begin{center}
$g_{S} = \frac{1}{2}d\eta(\text{id}\otimes \phi) + \eta \otimes \eta$.
\end{center}
Denoting by $g_{S}^{T}$ the (transversal) K\"{a}hler metric induced by $\frac{\rho_{0}}{2I(X_{P})}$ on $X_{P}$, it follows from Remark \ref{homothetic} that 
\begin{enumerate}

\item ${\text{Ric}}_{g_{S}}(X,Y) = {\text{Ric}}_{g_{S}^{T}}(X,Y) - 2g_{S}(X,Y)$, $\forall X,Y \in \mathscr{D} \cong TX_{P}$,

\item ${\text{Ric}}_{g_{S}}(X,\xi) = 2\dim_{\mathbbm{C}}(X_{P})\eta(X)$, $\forall X \in TQ(\mathscr{O}_{X_{P}}(-1) ).$
\end{enumerate}
Therefore, since ${\text{Ric}}(g_{S}^{T}) = 2I(X_{P})g_{S}^{T}$, see Eq. (\ref{riccinorm}), by considering the $\mathscr{D}$-homothetic transformation defined by
\begin{equation}
\label{rescaleconst}
a = \frac{I(X_{P})}{\dim_{\mathbbm{C}}(X_{P}) + 1},
\end{equation}
we obtain a Sasaki structure $(g_{a},\phi,\frac{1}{a}\xi,a\eta)$ on $Q(\mathscr{O}_{X_{P}}(-1) )$, with rescaled Riemannian metric $g_{a}$ given by
\begin{equation}
\label{metricsasakieinstein}
g_{a} = \displaystyle \frac{I(X_{P})}{\dim_{\mathbbm{C}}(X_{P})+1} \bigg ( \frac{1}{2}d \eta ({\rm{id}} \otimes \phi)  + \frac{I(X_{P})}{\dim_{\mathbbm{C}}(X_{P})+1}\eta \otimes \eta \bigg ),
\end{equation}
cf. \cite[Theorem 2]{CONTACTCORREA}. The metric above satisfies ${\text{Ric}}(g_{a}) = 2\dim_{\mathbbm{C}}(X_{P})g_{a}$, see Remark \ref{homothetic}, so it defines the appropriate Sasaki-Einstein structure to be considered. From the Sasaki-Einstein structure $(g_{a},\phi,\frac{1}{a}\xi,a\eta)$ described above, we can consider the Ricci-flat metric on ${\rm{Tot}}(\mathscr{O}_{X_{P}}(-1)^{\times})$ defined by
\begin{equation}
g_{CY} = r^{2}g_{a} + dr\otimes dr,
\end{equation}
see \cite[Corollary 11.1.8]{BoyerGalicki}, such that $r^{2} = {\rm{K}}_{H}$, with $ {\rm{K}}_{H}$ defined as in Eq. (\ref{Kpotencialcanonical}). 
From the $\mathscr{D}$-homothetic transformation induced by the scalar defined in Eq. (\ref{rescaleconst}), the complex structure $\mathscr{J} \in {\text{End}}(T({\rm{Tot}}(\mathscr{O}_{X_{P}}(-1)^{\times})))$ becomes
\begin{equation}
\label{complexchange}
\mathscr{J}(Y)= \phi(Y) - a\eta(Y)r\partial_{r}, \ \ \ \ \ \mathscr{J}( r\partial_{r}) = \frac{1}{a}\xi.
\end{equation}
Thus, considering $\psi = \log(r)$, we obtain $g_{CY} = {\rm{e}}^{2\psi}\big(g_{a} + d\psi\otimes d\psi\big)$, and
\begin{equation}
\omega_{CY} = g_{CY}(\mathscr{J}\otimes {\rm{id}}) = a \Big ( rdr \wedge \eta + \frac{r^{2}}{2}d\eta \Big) =  a{\rm{e}}^{2\psi}\Big ( d\psi \wedge \eta + \frac{d\eta}{2}\Big).
\end{equation}
From above, we obtain a globally conformally K\"{a}hler structure $(\widetilde{\Omega}, \mathscr{J},\widetilde{\theta})$ on ${\rm{Tot}}(\mathscr{O}_{X_{P}}(-1))$, such that 
\begin{equation}
\displaystyle \widetilde{\Omega} = {\rm{e}}^{-2\psi}\omega_{CY} = a \Big (d\psi \wedge \eta + \frac{d\eta}{2} \Big),
\end{equation}
notice that $d\widetilde{\Omega} = (-2d\psi) \wedge \widetilde{\Omega}$, so we have $\widetilde{\theta}  = -2d\psi$. The globally conformally K\"{a}hler structure $(\widetilde{\Omega}, \mathscr{J},\widetilde{\theta})$ descends to a Hermitian-Einstein-Weyl structure $(\Omega,J,\theta)$ on $M = {\rm{Tot}}(\mathscr{O}_{X_{P}}(-1)^{\times})/\mathbbm{Z}$, satisfying
\begin{equation}
\label{VaismanRicciflat}
\displaystyle \wp^{\ast}\Omega =  a \Big (d\psi \wedge \eta + \frac{d\eta}{2} \Big), \ \ J \circ \wp_{\ast} = \wp_{\ast} \circ \mathscr{J}, \ \ {\text{and}} \ \ \wp^{\ast}\theta = -2d\psi,
\end{equation}
such that $\wp \colon {\rm{Tot}}(\mathscr{O}_{X_{P}}(-1)^{\times}) \to M$ is the associated projection map. By denoting $\theta = \theta_{g}$, where $g = \Omega({\rm{id}}\otimes J)$, we have locally
\begin{equation}
\label{Leericciflat}
\theta_{g} = -\bigg [ \frac{1}{I(X_{P})}d\log \Big ( \big | \big |s_{U}v_{\delta_{P}}^{+} \big | \big |^{2}\Big ) +  \frac{\overline{w}dw + wd\overline{w}}{|w|^{2}} \bigg],
\end{equation}
for some local section $s_{U} \colon U \subset X_{P} \to K_{\text{ss}}^{\mathbbm{C}}$. It is straightforward to see that $\theta$ completely defines $(\Omega,J,\theta)$, see for instance Section \ref{mainremark}.

In order to conclude, we consider the case that $\pi_{1}(Q) = \mathbbm{Z}_{\ell}$, i.e., $Q = \widetilde{Q}/\mathbbm{Z}_{\ell}$, and $M = \widetilde{Q} \times_{\mathbbm{Z}_{\ell}}S^{1}$. In this case, the associated (minimal) presentation is given by $M = \big ({\rm{Tot}}(\mathscr{O}_{X_{P}}(-\ell)^{\times}), \mathbbm{Z} \big)$. The description for the homogeneous Hermitian-Einstein-Weyl structure on $M$ can be obtained exactly as in the previous case. In fact, we just need to consider the K\"{a}hler potential ${\rm{K}}_{H} \colon {\rm{Tot}}(\mathscr{O}_{X_{P}}(-\ell)^{\times}) \to \mathbbm{R}$, locally described by
\begin{equation}
\displaystyle {\rm{K}}_{H}\big (z,w \big ) = \Big (\big | \big |s_{U}(z)v_{\delta_{P}}^{+} \big | \big |^{\frac{2\ell}{I(X_{P})}} \Big)w\overline{w},
\end{equation}
for some local section $s_{U} \colon U \subset X_{P} \to K_{\text{ss}}^{\mathbbm{C}}$, and proceed as in Section \ref{mainremark}. 
\end{proof}

\section{Examples of Vaisman structures and Hermitian-Einstein-Weyl metrics}
In this section we provide several examples which illustrate the results presented in the previous theorems. In what follows, we shall use the notations and conventions introduced in Section \ref{subsec3.1}. Also, in order to perform some local computations it will be useful to consider the analytic cellular decomposition of complex flag manifolds by means Schubert cells. Being more precise, we shall consider the open set defined by the ``opposite" big cell in $X_{P}$. This open set is a distinguished coordinate neighbourhood $U \subset X_{P}$ of $x_{0} = eP \in X_{P}$ defined by the maximal Schubert cell. A brief description for the opposite big cell can be done as follows: Let $\Pi = \Pi^{+} \cup \Pi^{-}$ be the root system associated to a simple root system $\Sigma \subset \mathfrak{h}^{\ast}$, from this we can define the opposite big cell $U \subset X_{P}$ by

\begin{center}

 $U =  B^{-}x_{0} = R_{u}(P_{\Theta})^{-}x_{0} \subset X_{P}$,  

\end{center}
 where $B^{-} = \exp(\mathfrak{h} \oplus \mathfrak{n}^{-})$, and
 
 \begin{center}
 
 $R_{u}(P_{\Theta})^{-} = \displaystyle \prod_{\alpha \in \Pi^{-} \backslash \langle \Theta \rangle^{-}}N_{\alpha}^{-}$, \ \ (opposite unipotent radical)
 
 \end{center}
with $N_{\alpha}^{-} = \exp(\mathfrak{g}_{\alpha})$, $\forall \alpha \in \Pi^{-} \backslash \langle \Theta \rangle^{-}$. It is worth mentioning that the opposite big cell defines a contractible open dense subset in $X_{P}$, thus the restriction of any vector bundle over this open set is trivial. For further results about Schubert cells and Schubert varieties, we suggest \cite{MONOMIAL}.

\begin{example}[Basic model] 
\label{basicmodel}
Consider $G^{\mathbbm{C}}$ as being a simple Lie group, and take $\Theta = \Sigma \backslash \{\alpha\}$, for some fixed $\alpha \in \Sigma$. Let us denote by $P_{\Theta} = P_{\omega_{\alpha}}$, such that $\omega_{\alpha} \in \Lambda_{\mathbbm{Z}_{\geq 0}}^{\ast}$, the parabolic Lie subgroup associated to $\Theta \subset \Sigma$. From this, we can consider the flag manifold 
\begin{equation}
X_{P_{\omega_{\alpha}}} = G^{\mathbbm{C}}/P_{\omega_{\alpha}}. 
\end{equation}
Applying Proposition \ref{C8S8.2Sub8.2.3P8.2.6}, it follows that ${\text{Pic}}(X_{P_{\omega_{\alpha}}}) = \mathbbm{Z}c_{1}(\mathscr{O}_{\alpha}(1))$, and a straightforward computation shows that 
\begin{equation} 
\label{maximalparabolic}
I(X_{P_{\omega_{\alpha}}}) = \langle \delta_{P_{\omega_{\alpha}}},h_{\alpha}^{\vee} \rangle, \ \ \  {\text{and}}  \  \ \ K_{X_{P_{\omega_{\alpha}}}}^{ \otimes \frac{1}{ \langle \delta_{P_{\omega_{\alpha}}},h_{\alpha}^{\vee} \rangle}} = \mathscr{O}_{\alpha}(-1).
\end{equation}
Hence, given a negative line bundle $L \in {\text{Pic}}(X_{P_{\omega_{\alpha}}})$, it follows that 
\begin{center}
$L = \mathscr{O}_{\alpha}(-\ell)$,
\end{center}
for some integer $\ell \in \mathbbm{Z}$, such that $\ell>0$. As we have seen, we can associate to $L$ an irreducible $\mathfrak{g}^{\mathbbm{C}}$-module $V(\mu(L))$ with highest weight vector $v_{\mu(L)}^{+} \in V(\mu(L))$, such that
\begin{center}
$\displaystyle V(\mu(L)) \subset V(\omega_{\alpha})^{\otimes \ell}$, \ \ \ {\text{and}} \ \ \ $\displaystyle v_{\mu(L)}^{+} = \underbrace{v_{\omega_{\alpha}}^{+} \otimes \cdots \otimes v_{\omega_{\alpha}}^{+}}_{\ell{\text{-times}}}$,
\end{center}
notice that, in this case,  $\mu(L) = \ell\omega_{\alpha}$, see for instance Eq. (\ref{negativeweight}).
From the above data we can consider the manifold
\begin{equation}
M = {\rm{Tot}}(\mathscr{O}_{\alpha}(-\ell)^{\times})/\Gamma, \ \ \ {\text{where}} \ \ \ \Gamma = \big \{ \lambda^{n} \in \mathbbm{C}^{\times} \ \big | \ n \in \mathbbm{Z}  \big \},
\end{equation}
for some $\lambda \in \mathbbm{C}^{\times}$, such that $|\lambda| <1$. Applying Theorem \ref{Theo1} we obtain a (Vaisman) l.c.K. structure $(\Omega, J, \theta)$ on $M$ completely determined by a K\"{a}hler potential ${\rm{K}}_{H} \colon {\rm{Tot}}(L^{\times}) \to \mathbbm{R}^{+}$, which can be described in coordinates $(z,w) \in  \mathscr{O}_{\alpha}(-\ell)^{\times}|_{U}$ as
\begin{equation}
\label{potentialmaxparabolic}
{\rm{K}}_{H}\big (z,w \big ) = \Big (\big | \big |s_{U}(z)v_{\omega_{\alpha}}^{+} \big | \big |^{2\ell} \Big)  w\overline{w},
\end{equation}
for some local section $s_{U} \colon U \subset X_{P_{\omega_{\alpha}}} \to G^{\mathbbm{C}}$. By means of the K\"{a}hler potential above we can describe the l.c.K. structure $(\Omega, J, \theta)$ on $M$ from the Lee form
\begin{equation}
 \displaystyle \theta = -\ell\bigg [d\log \Big ( \big | \big |s_{U}v_{\omega_{\alpha}}^{+} \big | \big |^{2}\Big ) + \frac{1}{\ell}\frac{\overline{w}dw + wd\overline{w}}{|w|^{2}}\bigg],
\end{equation}
recall that $\vartheta = -\theta \circ J$, and $\Omega = -\frac{1}{4}(d\vartheta - \theta \wedge \vartheta)$, see for instance Section \ref{mainremark}. In particular, if we consider $\ell = 1$, it follows that $L = \mathscr{O}_{\alpha}(-1)$, i.e., $L$ is the maximal root of the canonical bundle of $X_{P_{\omega_{\alpha}}}$, see Eq. (\ref{maximalparabolic}). In this particular case, by applying Theorem \ref{Theo2} on the compact Homogeneous Hermitian-Einstein-Weyl manifold of the form $M = {\rm{Tot}}(\mathscr{O}_{\alpha}(-1)^{\times})/\mathbbm{Z}$, such that $K_{\text{ss}} = G$, we have the unique Hermitian-Einstein-Weyl structure $(\Omega,J,\theta_{g})$ on $M$ completely determined by
\begin{equation}
\theta_{g} = -\frac{1}{\langle \delta_{P_{\omega_{\alpha}}},h_{\alpha}^{\vee} \rangle}\Bigg [d\log \Big ( \big | \big |s_{U}v_{ \delta_{P_{\omega_{\alpha}}}}^{+} \big | \big |^{2}\Big ) + \langle \delta_{P_{\omega_{\alpha}}},h_{\alpha}^{\vee} \rangle\frac{\overline{w}dw + wd\overline{w}}{|w|^{2}}\Bigg],
\end{equation}
Notice that, in this last setting, the complex structure $J \in {\text{End}}(TM)$ is obtained from a complex structure on ${\rm{Tot}}(\mathscr{O}_{\alpha}(-1)^{\times})$ after a suitable $\mathscr{D}$-homothetic transformation induced from
\begin{equation}
a = \frac{\langle \delta_{P_{\omega_{\alpha}}},h_{\alpha}^{\vee} \rangle}{\dim_{\mathbbm{C}}(X_{P_{\omega_{\alpha}}}) + 1},
\end{equation}
on the Sasaki structure defined on the unitary frame bundle of $\mathscr{O}_{\alpha}(-1)$, see for instance Eq. (\ref{complexchange}) and Remark \ref{homothetic}. 
\end{example}

The ideas above provide a constructive and systematic method to describe examples of Vaisman manifolds and homogeneous Hermitian-Einstein-Weyl manifolds associated to complex flag manifolds with Picard number equal to one. In what follows we shall further explore the application of the previous ideas in a class of concrete examples provided by complex Grassmannians.

\begin{example}[Complex Grassmannian] 
\label{Grassmanianexample}
Considering $G^{\mathbbm{C}} = {\rm{SL}}(n+1,\mathbbm{C})$ and fixing the Cartan subalgebra $\mathfrak{h} \subset \mathfrak{sl}(n+1,\mathbbm{C})$ given by diagonal matrices whose the trace is equal to zero, we have the set of simple roots given by
$$\Sigma = \Big \{ \alpha_{l} = \epsilon_{l} - \epsilon_{l+1} \ \Big | \ l = 1, \ldots,n\Big\},$$
here we consider $\epsilon_{l} \colon {\text{diag}}\{a_{1},\ldots,a_{n+1} \} \mapsto a_{l}$, $ \forall l = 1, \ldots,n+1$. Therefore, the set of positive roots is given by
$$\Pi^+ = \Big \{ \alpha_{ij} = \epsilon_{i} - \epsilon_{j} \ \Big | \ i<j  \Big\}, $$
notice that $\alpha_{l} = \alpha_{ll+1}$, $\forall l = 1, \ldots,n$. Considering $\Theta = \Sigma \backslash \{\alpha_{k}\}$ and $P = P_{\omega_{\alpha_{k}}}$, we have the complex Grassmannian manifold

\begin{center}

${\rm{Gr}}(k,\mathbbm{C}^{n+1}) = {\rm{SL}}(n+1,\mathbbm{C})/P_{\omega_{\alpha_{k}}}.$

\end{center}
A straightforward computation shows that ${\text{Pic}}({\rm{Gr}}(k,\mathbbm{C}^{n+1})) = \mathbbm{Z}c_{1}(\mathscr{O}_{\alpha_{k}}(1))$, and we can also show that
\begin{equation}
\label{fanograssmannian}
\big \langle \delta_{P_{\omega_{\alpha_{k}}}},h_{\alpha_{k}}^{\vee} \big \rangle = n+1 \Longrightarrow K_{{\rm{Gr}}(k,\mathbbm{C}^{n+1})}^{\otimes \frac{1}{n+1}} = \mathscr{O}_{\alpha_{k}}(-1).  
\end{equation}
Hence, given a negative line bundle $L \in {\text{Pic}}({\rm{Gr}}(k,\mathbbm{C}^{n+1}))$, it follows that $L = \mathscr{O}_{\alpha_{k}}(-\ell)$, for some integer $\ell >0$. From the above data and the previous example, we can consider the manifold defined by
\begin{center}
$M = {\rm{Tot}}(\mathscr{O}_{\alpha_{k}}(-\ell)^{\times})/\Gamma, \ \ \ {\text{where}}$ \ \ \ $\Gamma = \big \{ \lambda^{n} \in \mathbbm{C}^{\times} \ \big | \ n \in \mathbbm{Z}  \big \}$,
\end{center}
for some $\lambda \in \mathbbm{C}^{\times}$, such that $|\lambda| <1$. Applying Theorem \ref{Theo1} we obtain a (Vaisman) l.c.K. structure $(\Omega, J, \theta)$ on $M$ completely determined by a K\"{a}hler potential ${\rm{K}}_{H} \colon {\rm{Tot}}(L^{\times}) \to \mathbbm{R}^{+}$, which can be described in local coordinates $(z,w) \in  \mathscr{O}_{\alpha_{k}}(-\ell)^{\times}|_{U}$ as
\begin{center}
${\rm{K}}_{H}\big (z,w \big ) = \Big (\big | \big |s_{U}(z)v_{\omega_{\alpha_{k}}}^{+} \big | \big |^{2\ell} \Big)  w\overline{w}.$
\end{center}
In order to provide a concrete description for the above K\"{a}hler potential, we proceed as follows: Since we have 
\begin{center}
$V(\omega_{\alpha_{k}}) = \bigwedge^{k}(\mathbbm{C}^{n+1}),  \ \ \ {\text{and}} \ \ \  v_{\omega_{\alpha_{k}}}^{+} = e_{1} \wedge \ldots \wedge e_{k}$,
\end{center}
we fix the canonical basis $e_{i_{1}}\wedge \cdots \wedge e_{i_{k}}$, $\{i_{1} < \ldots < i_{k}\} \subset \{1, \ldots, n+1\}$, for $V(\omega_{\alpha_{k}}) = \bigwedge^{k}(\mathbbm{C}^{n+1})$. By taking the coordinate neighborhood defined by the opposite big cell $U =  R_{u}(P_{\omega_{\alpha_{k}}})^{-}x_{0} \subset {\rm{Gr}}(k,\mathbbm{C}^{n+1})$, from the parameterization

\begin{center}

$Z \in \mathbbm{C}^{(n+1-k)k} \mapsto n(Z)x_{0} = \begin{pmatrix}
 \ 1_{k} & 0_{k,n+1-k} \\
 Z & 1_{n+1-k}
\end{pmatrix}x_{0},$

\end{center}
here we have $Z = (z_{ij}) \in \mathbbm{C}^{(n+1-k)k} \cong {\rm{M}}_{n+1-k,k}(\mathbbm{C})$, we can take the local section $s_{U} \colon U \subset {\rm{Gr}}(k,\mathbbm{C}^{n+1})\to {\rm{SL}}(n+1,\mathbbm{C})$ defined by
\begin{center}
$s_{U}(n(Z)x_{0}) = n(Z) = \begin{pmatrix}
 \ 1_{k} & 0_{k,n+1-k} \\
 Z & 1_{n+1-k}
\end{pmatrix}.$
\end{center}
The local section above allows us to write locally
\begin{equation}
\label{potentialgrassman}
{\rm{K}}_{H}\big (Z,w \big ) = \Bigg (\sum_{I} \bigg | \det_{I} \begin{pmatrix}
 \ 1_{k} \\
 Z 
\end{pmatrix} \bigg |^{2} \Bigg )^{\ell}w\overline{w},
\end{equation}
where the sum above is taken over all $k \times k$ submatrices whose the lines are labeled by the set of index $I = \{i_{1} < \ldots < i_{k}\} \subset \{1, \ldots, n+1\}$. By means of the K\"{a}hler potential above we can describe the l.c.K. structure $(\Omega, J, \theta)$ on $M$ from the Lee form
\begin{equation}
\displaystyle \theta = -\ell\Bigg [d\log  \Bigg (\sum_{I} \bigg | \det_{I} \begin{pmatrix}
 \ 1_{k} \\
 Z 
\end{pmatrix} \bigg |^{2} \Bigg ) + \frac{1}{\ell}\frac{\overline{w}dw + wd\overline{w}}{|w|^{2}}\Bigg].
\end{equation} 
From Eq. (\ref{fanograssmannian}), and Theorem \ref{Theo2}, for the particular case that $\ell = 1$, if we consider the compact homogeneous Hermitian-Einstein-Weyl manifold of the form $M = {\rm{Tot}}(\mathscr{O}_{\alpha_{k}}(-1)^{\times})/\mathbbm{Z}$, such that $K_{\text{ss}} = G = {\rm{SU}}(n+1)$, the unique homogeneous Hermitian-Einstein-Weyl structure $(\Omega,J,\theta_{g})$ on $M$ is completely determined by the Lee form
\begin{equation}
\theta_{g} = -\Bigg [d\log \Bigg (\sum_{I} \bigg | \det_{I} \begin{pmatrix}
 \ 1_{k} \\
 Z 
\end{pmatrix} \bigg |^{2} \Bigg ) + \frac{\overline{w}dw + wd\overline{w}}{|w|^{2}}\Bigg].
\end{equation}
A particular low dimensional example of the construction above is provided by ${\rm{Gr}}(2,\mathbbm{C}^{4}) = {\rm{SL}}(4,\mathbbm{C})/P_{\omega_{\alpha_{2}}}$ (Klein quadric). In this case we have
\begin{center}
$V(\omega_{\alpha_{2}}) = \bigwedge^{2}(\mathbbm{C}^{4})$ \ \ and \ \  $v_{\omega_{\alpha_{2}}}^{+} =  e_{1} \wedge e_{2}$, 
\end{center}
here we fix the basis $\{e_{i} \wedge e_{j}\}_{i<j}$ for $V(\omega_{\alpha_{2}}) = \bigwedge^{2}(\mathbbm{C}^{4})$. Similarly to the previous computations, we consider the open set defined by the opposite big cell $U = B^{-}x_{0} \subset {\rm{Gr}}(2,\mathbbm{C}^{4})$. This open set is parameterized by local coordinates $z \mapsto n(z)x_{0} \in U$, given by

\begin{center}
$z = (z_{1},z_{2},z_{3},z_{4}) \in \mathbbm{C}^{4} \mapsto\begin{pmatrix}
1 & 0 & 0 & 0\\
0 & 1 & 0 & 0 \\                  
z_{1}  & z_{3} & 1 & 0 \\
z_{2}  & z_{4} & 0 & 1
 \end{pmatrix} x_{0} \in U = B^{-}x_{0}.$
\end{center}
Given a negative line bundle $L \in {\text{Pic}}({\rm{Gr}}(2,\mathbbm{C}^{4}))$, we have $L = \mathscr{O}_{\alpha_{2}}(-\ell)$, for some integer $\ell >0$. From the above data and the previous ideas we can consider the manifold
\begin{center}
$M = {\rm{Tot}}(\mathscr{O}_{\alpha_{2}}(-\ell)^{\times})/\Gamma, \ \ \ {\text{where}}$ \ \ \ $\Gamma = \big \{ \lambda^{n} \in \mathbbm{C}^{\times} \ \big | \ n \in \mathbbm{Z}  \big \}$,
\end{center}
for some $\lambda \in \mathbbm{C}^{\times}$, such that $|\lambda| <1$. Applying Theorem \ref{Theo1} we obtain a (Vaisman) l.c.K. structure $(\Omega, J, \theta)$ on $M$ completely determined by a K\"{a}hler potential ${\rm{K}}_{H} \colon {\rm{Tot}}(\mathscr{O}_{\alpha_{2}}(-\ell)^{\times}) \to \mathbbm{R}^{+}$, which in turn can be described in local coordinates $(z,w) \in  \mathscr{O}_{\alpha_{2}}(-\ell)^{\times}|_{U}$ as
\begin{center}
${\rm{K}}_{H}\big (z,w \big ) = \Big (\big | \big |s_{U}(z)v_{\omega_{\alpha_{2}}}^{+} \big | \big |^{2\ell} \Big)  w\overline{w}.$
\end{center}
Thus, by considering the local section $s_{U} \colon U \subset {\rm{Gr}}(2,\mathbbm{C}^{4})\to {\rm{SL}}(4,\mathbbm{C})$, such that $s_{U}(n(z)x_{0}) = n(z)$, we obtain the concrete expression 
\begin{equation}
\label{C8S8.3Sub8.3.2Eq8.3.21}
{\rm{K}}_{H}\big (z,w \big ) =  \displaystyle  \Bigg (1+ \sum_{k = 1}^{4}|z_{k}|^{2} + \bigg |\det \begin{pmatrix}
 z_{1} & z_{3} \\
 z_{2} & z_{4}
\end{pmatrix} \bigg |^{2} \Bigg)^{\ell}w\overline{w}.
\end{equation}
By means of the K\"{a}hler potential above we can describe the l.c.K. structure $(\Omega, J, \theta)$ on $M$ from the Lee form
\begin{equation}
\displaystyle \theta = -\ell\Bigg [d\log  \Bigg (1+ \sum_{k = 1}^{4}|z_{k}|^{2} + \bigg |\det \begin{pmatrix}
 z_{1} & z_{3} \\
 z_{2} & z_{4}
\end{pmatrix} \bigg |^{2} \Bigg) + \frac{1}{\ell}\frac{\overline{w}dw + wd\overline{w}}{|w|^{2}}\Bigg].
\end{equation} 
It is worthwhile to observe that in this case we have the Fano index of ${\rm{Gr}}(2,\mathbbm{C}^{4})$ given by $I({\rm{Gr}}(2,\mathbbm{C}^{4})) = 4$, thus we obtain
\begin{center}   
$K_{{\rm{Gr}}(2,\mathbbm{C}^{4})}^{\otimes \frac{1}{4}} = \mathscr{O}_{\alpha_{2}}(-1).$        
\end{center}
As before, from Theorem \ref{Theo2}, for the particular case $\ell = 1$, for the compact homogeneous Hermitian-Einstein-Weyl manifold of the form $M = {\rm{Tot}}(\mathscr{O}_{\alpha_{k}}(-1)^{\times})/\mathbbm{Z}$, such that $K_{\text{ss}} = G = {\rm{SU}}(4)$, the unique homogeneous Hermitian-Einstein-Weyl structure $(\Omega,J,\theta_{g})$ on $M$ is completely determined by the Lee form
\begin{equation}
\theta_{g} = -\Bigg [d\log  \Bigg (1+ \sum_{k = 1}^{4}|z_{k}|^{2} + \bigg |\det \begin{pmatrix}
 z_{1} & z_{3} \\
 z_{2} & z_{4}
\end{pmatrix} \bigg |^{2} \Bigg) + \frac{\overline{w}dw + wd\overline{w}}{|w|^{2}}\Bigg].
\end{equation}
In this particular case the compact simply connected Sasaki manifold defined by the sphere bundle of $\mathscr{O}_{\alpha_{2}}(-1)$ is given by the Stiefel manifold, namely, the underlying homogeneous Boothby-Wang fibration is given by the Tits fibration 
\begin{center}
$S^{1} \hookrightarrow \mathscr{V}_{2}(\mathbbm{R}^{6}) \to {\rm{Gr}}(2,\mathbbm{C}^{4}),$
\end{center}
such that $\mathscr{V}_{2}(\mathbbm{R}^{6})$ is the Stiefel manifold. Thus, we obtain an explicit description for the (unique) homogeneous Hermitian-Einstein-Weyl structure on $M = \mathscr{V}_{2}(\mathbbm{R}^{6}) \times S^{1}$.

\end{example}

\begin{example}[Hopf manifold]
By considering $\mathbbm{C}{\rm{P}}^{m} = {\rm{Gr}}(1,\mathbbm{C}^{m+1})$, from Example \ref{Grassmanianexample} we obtain

\begin{center}

$\mathbbm{C}{\rm{P}}^{m} = {\rm{SL}}(m+1,\mathbbm{C})/P_{\omega_{\alpha_{1}}}.$

\end{center}
Thus, it follows that
\begin{equation}
{\text{Pic}}(\mathbbm{C}{\rm{P}}^{m}) = \mathbbm{Z}c_{1}\big(\mathscr{O}_{\alpha_{1}}(1)\big) \Longrightarrow K_{\mathbbm{C}{\rm{P}}^{m}}^{\otimes \frac{1}{m+1}} = \mathscr{O}_{\alpha_{1}}(-1).  
\end{equation}
Since in this case we have $V(\omega_{\alpha_{1}}) = \mathbbm{C}^{m+1}$ and $v_{\omega_{\alpha_{1}}}^{+} = e_{1}$, if we take the open set defined by the opposite big cell $U =  R_{u}(P_{\omega_{\alpha_{1}}})^{-}x_{0} \subset \mathbbm{C}{\rm{P}}^{m}$, where $x_0=eP_{\omega_{\alpha_{1}}}$ (trivial coset), we have that $U =  R_{u}(P_{\Theta})^{-}x_{0}$ can be parameterized by
\begin{center}
$(z_{1},\ldots,z_{m}) \in \mathbbm{C}^{m} \mapsto \begin{pmatrix}
1 & 0 &\cdots & 0 \\
z_{1} & 1  &\cdots & 0 \\                  
\ \vdots  & \vdots &\ddots  & \vdots  \\
z_{m} & 0 & \cdots &1 
 \end{pmatrix}x_{0} \in U =  R_{u}(P_{\omega_{\alpha_{1}}})^{-}x_{0}$.
\end{center}
For the sake of simplicity, we denote the above matrix by $n(z) \in {\rm{SL}}(m+1,\mathbbm{C})$. From this, we can take a local section $s_{U} \colon U \subset \mathbbm{C}{\rm{P}}^{m} \to {\rm{SL}}(m+1,\mathbbm{C})$, defined by
$$s_{U}(n(z)x_{0}) = n(z) \in {\rm{SL}}(m+1,\mathbbm{C}),$$ 
such that $z = (z_{1},\ldots,z_{m}) \in \mathbbm{C}^{m}$. Thus, by considering a negative line bundle $L \in {\text{Pic}}(\mathbbm{C}{\rm{P}}^{m})$, since $L = \mathscr{O}_{\alpha_{1}}(-\ell)$, for some integer $\ell > 0$, we have ${\rm{K}}_{H} \colon {\rm{Tot}}(\mathscr{O}_{\alpha_{1}}(-\ell)^{\times}) \to \mathbbm{R}^{+}$, which is described in local coordinates $(z,w) \in  \mathscr{O}_{\alpha_{1}}(-\ell)^{\times}|_{U}$ as
\begin{equation}
\label{potentialtautological}
\displaystyle {\rm{K}}_{H}(z,w) = \bigg (1+ \sum_{k = 1}^{m}|z_{k}|^{2} \bigg )^{\ell}w\overline{w},
\end{equation}
see for instance Eq. (\ref{potentialgrassman}). Therefore, by considering the manifold
\begin{center}
$M = {\rm{Tot}}(\mathscr{O}_{\alpha_{1}}(-\ell)^{\times})/\Gamma, \ \ \ {\text{where}}$ \ \ \ $\Gamma = \big \{ \lambda^{n} \in \mathbbm{C}^{\times} \ \big | \ n \in \mathbbm{Z}  \big \}$,
\end{center}
for some $\lambda \in \mathbbm{C}^{\times}$, such that $|\lambda| <1$, we can describe the (Vaisman) l.c.K. structure $(\Omega, J, \theta)$ on $M$ from the Lee form
\begin{equation}
\displaystyle \theta = -\ell\Bigg [d\log \bigg (1+ \sum_{k = 1}^{m}|z_{k}|^{2} \bigg )+ \frac{1}{\ell}\frac{\overline{w}dw + wd\overline{w}}{|w|^{2}}\Bigg] =  - \ell \sum_{i = 1}^{m}\frac{z_{i}d\overline{z}_{i} +\overline{z}_{i}dz_{i}}{\big (1+ \sum_{k = 1}^{m}|z_{k}|^{2} \big)} - \frac{\overline{w}dw + wd\overline{w}}{|w|^{2}}.
\end{equation}
Similarly to the previous examples, from Theorem \ref{Theo2}, for the particular case $\ell = 1$, if we consider the compact homogeneous Hermitian-Einstein-Weyl manifold of the form $M = {\rm{Tot}}(\mathscr{O}_{\alpha_{1}}(-1)^{\times})/\mathbbm{Z}$, such that $K_{\text{ss}} = G = {\rm{SU}}(m+1)$, it follows that $M = {\rm{Tot}}(\mathscr{O}_{\alpha_{1}}(-1)^{\times})/\mathbbm{Z} = S^{2m+1} \times S^{1}$. Notice that the associated underlying homogeneous Boothby-Wang fibration defined by the sphere bundle of $\mathscr{O}_{\alpha_{1}}(-1)$ is given in this case by the complex Hopf fibration 
\begin{center}
$S^{1} \hookrightarrow S^{2m+1} \to \mathbbm{C}{\rm{P}}^{m}.$
\end{center}
Further, the unique homogeneous Hermitian-Einstein-Weyl structure $(\Omega,J,\theta_{g})$ on $M$ is completely determined by the Lee form
\begin{equation}
\displaystyle \theta_{g} =  - \sum_{i = 1}^{m}\frac{z_{i}d\overline{z}_{i} +\overline{z}_{i}dz_{i}}{\big (1+ \sum_{k = 1}^{m}|z_{k}|^{2} \big)} - \frac{\overline{w}dw + wd\overline{w}}{|w|^{2}}.
\end{equation}
From above we obtain a complete description for the unique homogeneous Hermitian-Einstein-Weyl structure on the Hopf manifold $S^{2m+1} \times S^{1}$ by means of elements of Lie theory.
\end{example}

\begin{remark}[Lens space] It is worth mentioning that the previous ideas used in the description of the Vaisman structure on $M = S^{2m+1} \times S^{1}$ also hold for any manifold of the form 
\begin{center}
$M = \mathbbm{L}(m;\ell) \times S^{1},$
\end{center}
where $\mathbbm{L}(m;\ell) = S^{2m+1}/\mathbbm{Z}_{\ell}$ (Lens space), $\forall \ell \in \mathbbm{Z}_{>0}$. Actually, by considering $\mathscr{O}_{\alpha_{1}}(-\ell) \to \mathbbm{C}{\rm{P}}^{m}$, it is not difficult to see that the underlying homogeneous Boothby-Wang fibration defined by the sphere bundle of $\mathscr{O}_{\alpha_{1}}(-\ell)$ is given by the complex Hopf fibration
\begin{center}
$S^{1} \hookrightarrow \mathbbm{L}(m;\ell) \to \mathbbm{C}{\rm{P}}^{m}.$
\end{center}
From above, one can equip $M = \mathbbm{L}(m;\ell) \times S^{1}$ with a (Vaisman) l.c.K. structure naturally induced from the cone $\mathbbm{L}(m;\ell) \times \mathbbm{R}^{+} \cong {\text{Tot}}(\mathscr{O}_{\alpha_{1}}(-\ell)^{\times})$ through the K\"{a}hler potential described in Eq. (\ref{potentialtautological}). 
\end{remark}

\begin{example}[Elliptic bundles over full flag manifolds] 
\label{examplefullflag}
Recall that a full flag manifold is defined as the homogeneous space given by $G/T$, where $G$ is a compact simple Lie group, and $T \subset G$ is a maximal torus. For the sake of simplicity, let us suppose also that $G$ is simply connected. Considering the root system $\Pi = \Pi^{+} \cup \Pi^{-}$ associated to the pair $(G,T)$ \cite{Knapp}, from the complexification $G^{\mathbbm{C}}$ of $G$ we have an identification
\begin{equation}
G/T \cong G^{\mathbbm{C}}/B,
\end{equation}
where $B \subset G^{\mathbbm{C}}$ is a Borel subgroup such that $B \cap G = T$. From Proposition \ref{C8S8.2Sub8.2.3P8.2.6}, we have
\begin{equation}
{\text{Pic}}(G/T) = \bigoplus_{\alpha \in \Sigma}\mathbbm{Z}c_{1}\big (\mathscr{O}_{\alpha}(1) \big ).
\end{equation}
Therefore, given a negative line bundle $L \in {\text{Pic}}(G/T)$, it follows that 
\begin{center}
$\displaystyle L = \bigotimes_{\alpha \in \Sigma}\mathscr{O}_{\alpha}(-\ell_{\alpha})$,
\end{center}
such that $\ell_{\alpha} > 0$, $\forall \alpha \in \Sigma$. From Theorem \ref{Theo2}, we have a K\"{a}hler potential ${\rm{K}}_{H} \colon {\rm{Tot}}(L^{\times}) \to \mathbbm{R}^{+}$, which can be described in local coordinates $(z,w) \in L^{\times}|_{U}$ as
\begin{equation}
\label{kahlerpotentialfull}
\displaystyle {\rm{K}}_{H}\big (z,w \big ) = \Big ( \prod_{\alpha \in \Sigma}\big | \big |s_{U}(z)v_{\omega_{\alpha}}^{+} \big | \big |^{2\ell_{\alpha}}\Big)w\overline{w} = \Big ( \big | \big |s_{U}(z)v_{\mu(L)}^{+} \big | \big |^{2}\Big)w\overline{w},
\end{equation}
for some local section $s_{U} \colon U \subset G^{\mathbbm{C}}/B \to G^{\mathbbm{C}}$, where $v_{\mu(L)}^{+}$ is the highest weight vector of weight $\mu(L)$ for the irreducible $\mathfrak{g}^{\mathbbm{C}}$-module $V(\mu(L)) = H^{0}(G/T,L^{-1})^{\ast}$, see Eq. (\ref{negativeweight}). Hence, by considering the manifold
\begin{center}
$M = {\rm{Tot}}(L^{\times})/\Gamma, \ \ \ {\text{where}}$ \ \ \ $\Gamma = \big \{ \lambda^{n} \in \mathbbm{C}^{\times} \ \big | \ n \in \mathbbm{Z}  \big \}$,
\end{center}
for some $\lambda \in \mathbbm{C}^{\times}$, such that $|\lambda| <1$, from the K\"{a}hler potential described above we have a (Vaisman) l.c.K. structure  $(\Omega, J, \theta)$ on $M$, such that the Lee form $\theta $ is locally given by
\begin{equation}
\displaystyle \theta = -\bigg [d\log \Big ( \big | \big |s_{U}v_{\mu(L)}^{+} \big | \big |^{2}\Big ) + \frac{\overline{w}dw + wd\overline{w}}{|w|^{2}}\bigg].
\end{equation}
If we consider
\begin{equation}
\label{fullweight}
\displaystyle \varrho = \frac{1}{2} \sum_{\alpha \in \Pi^{+}}\alpha,
\end{equation}
it follows that $\delta_{B} = 2\varrho$. Therefore, from Eq. (\ref{canonicalbundleflag}), since $I(G/T) = 2$ \cite[$\S$ 13.3]{Humphreys}, it follows that
\begin{equation}
\label{canonicalfull}
K_{G/T} = L_{\chi_{2\varrho}}^{-1} \Longrightarrow \mathscr{O}_{G/T}(-1) = L_{\chi_{\varrho}}^{-1},
\end{equation}
recall that $K_{G/T}^{\otimes \frac{1}{I(G/T)}} = \mathscr{O}_{G/T}(-1)$. Thus, if we consider a compact homogeneous Hermitian-Einstein-Weyl manifold of the form $M = {\rm{Tot}}( \mathscr{O}_{G/T}(-1)^{\times})/\mathbbm{Z}$, such that $K_{\text{ss}} = G$, the unique homogeneous Hermitian-Einstein-Weyl structure $(\Omega,J,\theta_{g})$ on $M$ is completely determined by the Lee form locally describe by\footnote{It is worthwhile to observe that $v_{2\varrho}^{+} = v_{\varrho}^{+} \otimes v_{\varrho}^{+}$, such that $v_{\varrho}^{+} \in V(\varrho)$, see for instance \cite[Page 345]{Procesi}.} 
\begin{equation}
\label{leefull}
\displaystyle \theta_{g} = -\frac{1}{2}\bigg [d\log \Big ( \big | \big |s_{U}v_{2\varrho}^{+} \big | \big |^{2}\Big ) + 2\frac{\overline{w}dw + wd\overline{w}}{|w|^{2}}\bigg] = -d\log \Big ( \big | \big |s_{U}v_{\varrho}^{+} \big | \big |^{2}\Big ) - \frac{\overline{w}dw + wd\overline{w}}{|w|^{2}},
\end{equation}
for some local section $s_{U} \colon U \subset G/T \to G^{\mathbbm{C}}$, where $v_{\varrho}^{+}$ denotes the highest weight vector of weight $\varrho$ for the irreducible $\mathfrak{g}^{\mathbbm{C}}$-module $V(\varrho)$. 

In order to illustrate the ideas described above by means of a concrete example, consider $G = {\rm{SU}}(n+1)$ and the full flag manifold
\begin{center}
${\rm{SU}}(n+1)/T^{n} = {\rm{SL}}(n+1,\mathbbm{C})/B,$
\end{center}
here we fix the same Lie-theoretical data for $\mathfrak{sl}(n+1,\mathbbm{C})$ as in Example \ref{Grassmanianexample}. By taking a negative line bundle $L \in {\text{Pic}}({\rm{SU}}(n+1)/T^{n})$, we can consider the manifold
\begin{center}
$M = {\rm{Tot}}(L^{\times})/\Gamma, \ \ \ {\text{where}}$ \ \ \ $\Gamma = \big \{ \lambda^{n} \in \mathbbm{C}^{\times} \ \big | \ n \in \mathbbm{Z}  \big \}$,
\end{center}
for some $\lambda \in \mathbbm{C}^{\times}$, such that $|\lambda| <1$. From Theorem \ref{Theo1}, we have a (Vaisman) l.c.K. structure $(\Omega, J, \theta)$ on $M$ completely determined by a K\"{a}hler potential ${\rm{K}}_{H} \colon {\rm{Tot}}(L^{\times}) \to \mathbbm{R}^{+}$, which can be described as follows. As we have seen in Example \ref{Grassmanianexample}, in  this particular case we have the irreducible fundamental $\mathfrak{sl}(n+1,\mathbbm{C})$-modules given by
\begin{center}
$V(\omega_{\alpha_{k}}) = \bigwedge^{k}(\mathbbm{C}^{n+1}),  \ \ \ {\text{and}} \ \ \  v_{\omega_{\alpha_{k}}}^{+} = e_{1} \wedge \ldots \wedge e_{k},$
\end{center}
such that $k = 1,\ldots,n$. Let $U = R_{u}(B)^{-}x_{0} \subset{\rm{SL}}(n+1,\mathbbm{C})/B$ be the opposite big cell. This open set is parameterized by the holomorphic coordinates 
$$z \in \mathbbm{C}^{\frac{n(n+1)}{2}} \mapsto n(z)x_{0} = \begin{pmatrix}
  1 & 0 & 0 & \cdots & 0 \\
  z_{21} & 1 & 0 & \cdots & 0  \\
  z_{31} & z_{32} & 1 &\cdots & 0 \\
  \vdots & \vdots & \vdots & \ddots & \vdots \\
  z_{n+1,1} & z_{n+1,2} & z_{n+1,3} & \cdots & 1 
 \end{pmatrix}x_{0},$$
where $ n = n(z) \in N^{-}$ and $z = (z_{ij}) \in \mathbbm{C}^{\frac{n(n+1)}{2}}$. From this, we define for each subset $I = \{i_{1} < \cdots < i_{k}\} \subset \{1,\cdots,n+1\}$, with $1 \leq k \leq n$, the polynomial functions $\det_{I} \colon {\rm{SL}}(n+1,\mathbbm{C}) \to \mathbbm{C}$, such that 

$$\textstyle{\det_{I}}(g) = \det \begin{pmatrix}
  g_{i_{1}1} & g_{i_{1}2} & \cdots & g_{i_{1} k} \\
  g_{i_{2}1} & g_{i_{2}2} & \cdots & g_{i_{2} k}  \\
  \vdots & \vdots& \ddots & \vdots \\
  g_{i_{k}1} & g_{i_{k}2} & \cdots & g_{i_{k}k} 
 \end{pmatrix},$$
for every $g \in {\rm{SL}}(n+1,\mathbbm{C})$. Notice that, for every $g \in {\rm{SL}}(n+1,\mathbbm{C})$, we have 

\begin{center}

$g(e_{1} \wedge \ldots \wedge e_{l}) =  \displaystyle \sum_{i_{1} < \cdots < i_{l}} \textstyle{\textstyle{\det_{I}}}(g)e_{i_{1}} \wedge \ldots \wedge e_{i_{l}},$

\end{center}
where the sum above is taken over $I = \{i_{1} < \cdots < i_{l}\} \subset \{1,\cdots,n+1\}$, with $1 \leq l \leq n$. By taking the local section $s_{U} \colon U \subset {\rm{SL}}(n+1,\mathbbm{C})/B \to {\rm{SL}}(n+1,\mathbbm{C})$, such that $s_{U}(n^{-}(z)x_{0}) = n(z)$, and supposing that 
\begin{center}
$\displaystyle L = \bigotimes_{k = 1}^{n}\mathscr{O}_{\alpha_{k}}(-\ell_{k}),$
\end{center}
with $\ell_{k} > 0$, for $k = 1,\ldots,n$, it follows from Eq. (\ref{kahlerpotentialfull}) that the K\"{a}hler potential ${\rm{K}}_{H} \colon {\rm{Tot}}(L^{\times}) \to \mathbbm{R}^{+}$, defined in local coordinates $(z,w) \in L^{\times}|_{U}$, is given by
\begin{equation}
\label{sl-potential}
\displaystyle {\rm{K}}_{H}\big (z,w \big ) = \Bigg [\prod_{k = 1}^{n}\Bigg (  \sum_{i_{1} < \cdots < i_{k}} \bigg | \textstyle{\det_{I}} \begin{pmatrix}
  1 & 0 & \cdots & 0 \\
  z_{21} & 1 & \cdots & 0  \\
  \vdots & \vdots & \ddots & \vdots \\
  z_{n+1,1} & z_{n+1,2} & \cdots & 1 
 \end{pmatrix} \bigg |^{2}\Bigg)^{\ell_{k}} \Bigg ] w\overline{w}.
\end{equation}
Hence, from the K\"{a}hler potential above, the Lee form associated to the (Vaisman) l.c.K. structure $(\Omega, J, \theta)$ on $M$ provided by Theorem \ref{Maintheo1} is given locally by
\begin{equation}
\displaystyle \theta = -\sum_{k = 1}^{n}\ell_{k}d\log\Bigg (  \sum_{i_{1} < \cdots < i_{k}} \bigg | \textstyle{\det_{I}} \begin{pmatrix}
  1 & 0 & \cdots & 0 \\
  z_{21} & 1 & \cdots & 0  \\
  \vdots & \vdots & \ddots & \vdots \\
  z_{n+1,1} & z_{n+1,2} & \cdots & 1 
 \end{pmatrix} \bigg |^{2}\Bigg) - \displaystyle \frac{\overline{w}dw + wd\overline{w}}{|w|^{2}}.
\end{equation}
Moreover, from Theorem \ref{Theo2}, if we consider a compact homogeneous Hermitian-Einstein-Weyl manifold of the form 
\begin{center}
$M = {\rm{Tot}}(\mathscr{O}_{{\rm{SU}}(n+1)/T^{n}}(-1)^{\times})/\mathbbm{Z}$, 
\end{center}
such that $K_{\text{ss}} = G = {\rm{SU}}(n+1)$, we have the unique homogeneous Hermitian-Einstein-Weyl structure $(\Omega,J,\theta_{g})$ on $M$ completely determined by the Lee form
\begin{equation}
\displaystyle \theta_{g} = -\sum_{k = 1}^{n}d\log\Bigg (  \sum_{i_{1} < \cdots < i_{k}} \bigg | \textstyle{\det_{I}} \begin{pmatrix}
  1 & 0 & \cdots & 0 \\
  z_{21} & 1 & \cdots & 0  \\
  \vdots & \vdots & \ddots & \vdots \\
  z_{n+1,1} & z_{n+1,2} & \cdots & 1 
 \end{pmatrix} \bigg |^{2}\Bigg) - \displaystyle \frac{\overline{w}dw + wd\overline{w}}{|w|^{2}},
\end{equation}
in the expression above we have used the description given in Eq. (\ref{canonicalfull}), and the fact that $\varrho = \omega_{\alpha_{1}} + \cdots + \omega_{\alpha_{n}}$, see for instance \cite[$\S$ 13.3]{Humphreys}. 

\subsection{Hyperhermitian-Weyl structures in dimension 8} In this subsection, we explore some applications of our main results in the setting of hyperhermitian-Weyl geometry (e.g. \cite{OrneaPiccinni}, \cite{PedersenPoonSwann}, \cite{Ishihara}, \cite{BGMann}). The main purpose is to provide a complete description for locally conformally hyperK\"{a}hler metrics on 8-dimensional locally conformal hyperK\"{a}hler compact manifolds. 

\begin{definition}
A hyperhermitian manifold $(M,g,J_{1},J_{2},J_{3})$ is called locally conformally hyperK\"{a}hler (l.c.h.K.) if there exists an open cover $\mathscr{U}$ of $M$ and a family of smooth functions $\{f_{U}\}_{U \in \mathscr{U}}$, $f_{U} \colon U \to \mathbbm{R}$, such that, for all $U \in \mathscr{U}$, the local metric $g_{U} = {\mathrm{e}}^{-f_{U}}g|_{U}$, is hyperK\"{a}hler.
\end{definition}

In the context of the above definition, we have, respectively, the (global) K\"{a}hler 4-form and the associated Lee $1$-form given by $\Upsilon = \sum_{i = 1}^{3} \Omega_{i} \wedge \Omega_{i} \in \Omega^{4}(M)$, such that $\Omega_{i} = g(J_{i} \otimes {\text{id}})$, for $i = 1,2,3$, and $\theta \in \Omega^{1}(M)$, such that $\theta|_{U} = df_{U}$, for all $U \in \mathscr{U}$. These differential forms satisfy the following relation
\begin{center}
$d \Upsilon = \theta \wedge \Upsilon$.
\end{center}
In what follows, we shall denote by $\mathcal{G}$ the subbundle of ${\rm{End}}(TM)$ spanned by $J_{1},J_{2},J_{3}$. Since hyperK\"{a}hler metrics are Einstein metrics, from \cite{Gauduchon2} we have the following result.

\begin{corollary}
Let $(M,g,J_{1},J_{2},J_{3})$ be a compact locally conformally hyperK\"{a}hler manifold and assume that none metric in the conformal class $[g]$ is hyperK\"{a}hler. Then there exists $g_{0} \in [g]$, such that the associated Lee form $\theta_{g_{0}}$ is $\nabla^{g_{0}}$-parallel.
\end{corollary}

Let $M$ be a compact locally conformally hyperK\"{a}hler manifold. From the above corollary we can chose $g$ such that $||A|| = 1$, where $A = \theta^{\sharp} \in \Gamma(TM)$, and $\nabla\theta = 0$. In this setting, as in Theorem \ref{flatvaisman}, we have the following result.
\begin{corollary}
\label{RegularlchK}
If $\mathcal{F}_{A}$ is regular, then $M$ is a flat $S^{1}$-principal bundle over a $3$-Sasakian manifold $Q = M/\mathcal{F}_{A}$. Moreover, the associated projection map is a Riemannian submersion.
\end{corollary}

In the above setting, we have that the vector fields $A,J_{1}A,J_{2}A,J_{3}A$, span an integrable $4$-dimensional distribution $\mathcal{D}$ on $M$. Moreover, we have the following result \cite[Theorem 5.2]{PedersenPoonSwann}.

\begin{proposition}
\label{twistorfoliation}
Let $M$ be a compact locally conformally hyperK\"{a}hler manifold of real $\dim \geq 12$ with $\mathcal{D}$ regular. Then the leaf space $N = M/\mathcal{D}$ inherits a structure of quaternion K\"{a}hler manifold with positive scalar curvature.
\end{proposition}

\begin{remark}
In the above proposition, by choosing a complex structure $J \in \mathcal{G}$, and denoting by $\mathcal{F}_{J}$ the complex analytic folliation spanned by $A$ and $JA$, from Theorem \ref{vaismanclass} we have that $Z_{J} = M/\mathcal{F}_{J}$ is a K\"{a}hler-Einstein manifold analytically equivalent to the twistor space of $N$, see for instance \cite{OrneaPiccinni}. Thus, we simply denote $\mathcal{F} = \mathcal{F}_{J}$.
\end{remark}

The results above also hold for $\dim_{\mathbbm{R}}(M) = 8$, and the complete description for locally conformal hyperK\"{a}hler metrics on compact 8-dimensional manifolds can be done as follows. Combing the result of Ishihara \cite{Ishihara} with the results of Hitchin who classified all 4-dimensional compact quaternionic K\"{a}hler manifolds of positive scalar curvature \cite{Hitchin}, and Poon and Salamon who extended this classification to dimension 8 \cite{PoonSalamon}, we have the following classification of all fibered Riemannian spaces with Sasakian 3-structure.

\begin{theorem}[\cite{BGMann}]
\label{7dim3sasaki}
Let $Q$ be a complete principal Riemannian fibration with 3-Sasakian structure of positive scalar curvature. Then,
\begin{enumerate}
\item if $Q$ has dimension $7$, then $Q$ is either $S^{7}$, $\mathbbm{R}\mathbbm{P}^{7}$, or ${\rm{SU}}(3)/{\rm{U}}(1)$;
\item if $Q$ has dimension $11$, then $Q$ is either $S^{11}$, $\mathbbm{R}\mathbbm{P}^{11}$, ${\rm{SU}}(4)/{\rm{S}}({\rm{U}}(2) \times {\rm{U}}(1))$, or ${\rm{G_{2}}}/{\rm{SU}}(2)$.
\end{enumerate}
In particular, every such fibered Riemannian manifold of dimension 7 or 11 with 3-Sasakian structure of positive scalar curvature is homogeneous.
\end{theorem}

From the above theorem, Corollary \ref{RegularlchK}, and Proposition \ref{twistorfoliation}, we have that a compact locally conformally hyperK\"{a}hler 8-dimensional manifold $M$, such that the foliations $\mathcal{F}_{A}$, $\mathcal{D}$, and $\mathcal{F}$ are regular, is one of the following manifolds\footnote{It is worth mentioning that the twisted product $S^{7} \times_{\mathbbm{Z}_{2}} S^{1} = (S^{7} \times S^{1})/\mathbbm{Z}_{2}$ is obtained by considering the action of $\pi_{1}(\mathbbm{R}\mathbbm{P}^{7}) = \mathbbm{Z}_{2}$ on $S^{7}$ by deck transformation, and the action of $ \mathbbm{Z}_{2}$ on $S^{1}$ through the holonomy representation, see for instance Remark \ref{monodromyvaisman}.}
\begin{center}
$S^{7} \times S^{1}$, \ $\mathbbm{R}\mathbbm{P}^{7} \times S^{1}$, \ $S^{7} \times_{\mathbbm{Z}_{2}} S^{1}$, \ $\big ({\rm{SU}}(3)/{\rm{U}}(1) \big) \times S^{1}$,
\end{center}
see for instance \cite[Corollary 4.3]{OrneaPiccinni}. The locally conformally hyperK\"{a}hler metric on $S^{7} \times S^{1}$, \ $\mathbbm{R}\mathbbm{P}^{7} \times S^{1}$, and $S^{7} \times_{\mathbbm{Z}_{2}} S^{1}$, can be obtained by applying the construction described in Example \ref{Grassmanianexample} in the particular setting of the complex Hopf fibration
%\footnote{Notice that $S^{7} \times S^{1} = \big ({\rm{Tot}}(\mathscr{O}_{\mathbbm{C}{\rm{P}}^{3}}(-1)^{\times}),\mathbbm{Z} \big )$ and $\mathbbm{R}\mathbbm{P}^{7} \times S^{1} = \big ({\rm{Tot}}(\mathscr{O}_{\mathbbm{C}{\rm{P}}^{3}}(-2)^{\times}),\mathbbm{Z} \big )$, see Eq. \ref{conepresentation}, where the underlying $\mathbbm{Z}$-action is defined by dilatation (cf. Eq. \ref{dilatation}).}
\begin{center}
$S^{1} \hookrightarrow S^{7} \to \mathbbm{C}{\rm{P}}^{3}.$
\end{center}
For the manifold $\big ({\rm{SU}}(3)/{\rm{U}}(1)\big) \times S^{1}$, let us explain in more details how one can describe the unique homogeneous locally conformally hyperK\"{a}hler metric. Consider the particular case of Example \ref{examplefullflag} provided by the the Wallach space \cite{Wallach}
\begin{equation}
W_{6} = {\rm{SU}}(3)/T^{2}.
\end{equation}
Since in this case we have ${\text{Pic}}(W_{6}) = \mathbbm{Z}c_{1}\big (\mathscr{O}_{\alpha_{1}}(1)\big) \oplus \mathbbm{Z}c_{1}\big (\mathscr{O}_{\alpha_{2}}(1)\big)$, for any negative line bundle $L \in {\text{Pic}}(W_{6})$ we have a K\"{a}hler potential ${\rm{K}}_{H} \colon {\rm{Tot}}(L^{\times}) \to \mathbbm{R}^{+}$, given by
\begin{center}
$\displaystyle {\rm{K}}_{H}\big (z,w \big ) = \Bigg [ \bigg ( 1 + \sum_{i = 2}^{3}|z_{i1}|^{2} \bigg )^{\ell_{1}} \bigg (1 + |z_{32}|^{2} + \bigg | \det \begin{pmatrix}
z_{21} & 1  \\                  
z_{31}  & z_{32} 
 \end{pmatrix} \bigg |^{2} \bigg )^{\ell_{2}} \Bigg ]w\overline{w},$
\end{center}
see for instance Eq. (\ref{sl-potential}). From this, by considering the manifold 
\begin{center}
$M = {\rm{Tot}}(L^{\times})/\Gamma, \ \ \ {\text{where}}$ \ \ \ $\Gamma = \big \{ \lambda^{n} \in \mathbbm{C}^{\times} \ \big | \ n \in \mathbbm{Z}  \big \}$,
\end{center}
for some $\lambda \in \mathbbm{C}^{\times}$, such that $|\lambda| <1$, we have from Theorem \ref{Theo1} a (Vaisman) l.c.K. structure $(\Omega, J, \theta)$ on $M$ with associated Lee form (locally) described as follows
\begin{equation}
\theta = -\ell_{1}d\log\bigg ( 1 + \sum_{i = 2}^{3}|z_{i1}|^{2} \bigg ) - \ell_{2}d\log  \bigg (1 + |z_{32}|^{2} + \bigg | \det \begin{pmatrix}
z_{21} & 1  \\                  
z_{31}  & z_{32} 
 \end{pmatrix} \bigg |^{2} \bigg )  - \displaystyle \frac{\overline{w}dw + wd\overline{w}}{|w|^{2}}.
\end{equation}
If we consider a compact homogeneous Hermitian-Einstein-Weyl manifold of the form 
\begin{center}
$M = {\rm{Tot}}\big (\mathscr{O}_{{\rm{SU}}(3)/T^{2}}(-1)^{\times}\big )/\mathbbm{Z}$, 
\end{center}
it follows that 
\begin{equation}
M = X_{1,1} \times S^{1},
\end{equation}
such that $X_{1,1}$ is the Aloff-Wallach space \cite{Aloff} defined by $X_{1,1} = {\rm{SU}}(3)/{\rm{U}}(1)$. Actually, recall that $M$ is a flat $S^{1}$-principal bundle over the simply connected Sasaki manifold defined by the sphere bundle of $\mathscr{O}_{{\rm{SU}}(3)/T^{2}}(-1)$, which is defined by
\begin{center}
$S^{1} \hookrightarrow X_{1,1} \to W_{6}.$
\end{center}
Therefore, from Theorem \ref{Theo2}, we have a Hermitian-Einstein-Weyl structure $(\Omega, J, \theta_{g})$ on $M$ completely determined by the Lee form
\begin{equation}
\label{hyperkahlerhiggsfield}
\theta_{g} = -d\log \bigg [ \Big ( 1 + \sum_{i = 2}^{3}|z_{i1}|^{2} \Big )\Big (1 + |z_{32}|^{2} + \Big | \det \begin{pmatrix}
z_{21} & 1  \\                  
z_{31}  & z_{32} 
 \end{pmatrix} \Big |^{2} \Big ) \bigg ]  - \displaystyle \frac{\overline{w}dw + wd\overline{w}}{|w|^{2}}.
\end{equation}
From the uniqueness (up to homothety) of Gauduchon's metric we have that the Lee form given in Eq. (\ref{hyperkahlerhiggsfield}) defines a locally conformally hyperK\"{a}hler metric on ${\rm{SU}}(3)/{\rm{U}}(1) \times S^{1}$ (cf. \cite[Proposition 5.1]{PedersenPoonSwann}). 
\end{example}


\begin{thebibliography}{BGGSM}



%\bibitem{Akhiezer1} Akhiezer, D. N.; Dense  orbits with two  ends , Math. USSR  Izvestija 11 (1977),  293-307.



\bibitem{Akhiezer} Akhiezer, D. N.; Lie Group Actions in Complex Analysis, Aspects of Mathematics, Vieweg, Braunschweig, Wiesbaden, 1995.

\bibitem{Alekseevsky} Alekseevsky, D. V.; Flag manifolds, in: Sbornik Radova, 11 Jugoslav. Seminr. Beograd 6(14) (1997) 3-35. MR1491979 (99b:53073).

\bibitem{Aloff} Aloff, S.; Wallach, N. R; An infinite family of distinct 7-manifolds admitting positively curved Riemannian structures, Bull. Amer. Math. Soc. 81 (1975) 93.

%\bibitem{Artin} Artin, M.; On isolated rational singularities of surfaces, Amer. J. Math 88 (1966), 129-136.

\bibitem{AZAD} Azad, H.; Biswas, I.; Quasi-potentials and K\"ahler-Einstein metrics on flag manifolds. II. J. Algebra, 269(2):480--491, 2003.

%\bibitem{Baston} Baston, R. J.; Eastwood, M. G.; The Penrose transform. Its interaction with representation theory, Oxford Mathematical Monographs, Clarendon Press, Oxford, 1989.

%\bibitem{Belgun} Belgun, F. A.; On the metric structure of non-K\"{a}hler complex surfaces, Math. Ann. 317 (2000) 1-40.

\bibitem{Besse} Besse, Arthur L.; Einstein Manifolds; Springer; Berlin Heidelberg New York 1987 edition (2007).

%\bibitem{BrieskornI} Brieskorn, E.: Rationale Singularit\"{a}ten komplexer Fl\"{a}nchen, Invent. Math. 4, 336-358 (1968).

\bibitem{Brieskorn} Brieskorn, E.; van de Ven, A.; Some complex structures on products of homotopy spheres, Topology 7 (1968), 389-393.

\bibitem{Boothby} Boothby, W. M.; Some fundamental formulas for Hermitian manifolds with vanishing torsion, American J. Math., 76 (1954), 509-534.

\bibitem{BW} Boothby, W. M.; Wang, H. C.; On contact manifolds, Ann. of Math., 68 (1958), 721-734.

\bibitem{BorelH} Borel, A.; Hirzebruch, F.; Characteristic classes and homogeneous spaces I, Amer. J. Math. 80 (1958) 458-538.

\bibitem{Borel} Borel, A.; Sur la cohomologie des espaces fibres principaux et des espaces homogfnes de groupes de Lie compacts, Ann. of Math., 57 (1953), 115-207.

\bibitem{BorelK} Borel, A.; Kahlerian coset spaces of semisimple Lie groups, Proc. Nat. Acad, Sci. U.S.A. 40 (1954), 1147-1151; MR 17 \# 1108.

\bibitem{Bott} Bott, R.; Homogeneous vector bundles, Ann. of Math. 66 (1957), 203-248.

%\bibitem{Bott2} Bott, R.; Milnor, J.; On the parallelizability of the spheres. Bull. Amer. Math. Soc., 64 : 87-89, 1958. OrneaPiccinni PedersenPoonSwann Ishihara BGMann

\bibitem{BGMann} Boyer, C. P., Galicki, K.; Mann, B. M.; Quaternionic reduction and Einstein manifolds, Communications in Analysis and Geometry, Vol. 1, number 2, 229-279 (1993).

\bibitem{BOYER} Boyer, C. P.; Galicki, K.; On Sasakian-Einstein geometry, Internat. J. Math. 11 (2000), 873-909.

\bibitem{BoyerGalicki} Boyer, C. P.; Galicki, K.; Sasakian Geometry, Oxford Mathematical Monographs, Oxford University Press; 1 edition (2008).

\bibitem{Blair} Blair,  David  E.;  Riemannian  Geometry  of  Contact  and  Symplectic  Manifolds,  Progress  in  Mathematics  203, Birkh\"{u}ser Basel (2010).

%\bibitem{Brieskorn2} Brieskorn, E. V.; Beispiele zur Differentialtopologie von Singularit\"{a}ten, Inventiones Mathematicae, 2 (1): 1-14 (1966).

%\bibitem{Brieskorn1} Brieskorn, E. V.; Examples of singular normal complex spaces which are topological manifolds, Proceedings of the National Academy of Sciences, 55 (6): 1395-1397 (1966).

%\bibitem{Brocker} Br\"{o}cker, T.; Dieck, T. tom; Representations of Compact Lie Groups, Graduate Texts in Mathematics, Springer (1985).

%\bibitem{CALABIANSATZ} Calabi, E.; Metriques K\"{a}hl\'{e}riennes et fibr\'{e}s holomorphes, Ann. Sci. \'{E}cole Norm. Sup. (4) 12 (1979), 269--294.

%\bibitem{Calin} Calin, O.; Chang, D.-C., Geometric Mechanics on Riemannian Manifolds: Applications to Partial Differential Equations, Applied and Numerical Harmonic Analysis, Birkh\"{a}user (2005).

%\bibitem{Candelas} Candelas. P.; de la Ossa, X.; Comments on conifolds, Nucl. Phys. B 342 (1990) 246.

\bibitem{PARABOLICTHEORY} Cap, A.; Slov\'{a}k, J.; Parabolic Geometries I: Background and General theory, Mathematical Surveys and Monographs, American Mathematical Society (2009).

%\bibitem{Cartan} Cartan, H.; Quotient d'un espace analytique par un groupe d'automorphisms. Princeton Univ. Press, pages 90-102, 1957. In ``Algebraic geometry and topology" (Lefschetz symposium volume), edit. by R.H. Fox, D.C. Spencer and A.W. Tucker.

%\bibitem{Cheeger} Cheeger, J., Gromoll, D.; The splitting theorem for manifolds of nonnegative Ricci curvature. J. Diff. Geom. 6 (1971), 119-128

\bibitem{Chen} Chen, B. Y.; Piccinni, P.; The canonical foliations of a locally conformal K\"{a}hler manifold. Ann. Mat. Pura Appl., 141, 289-305 (1985).

\bibitem{CONTACTCORREA} Correa, E. M.; Homogeneous Contact Manifolds and Resolutions of Calabi-Yau Cones. Commun. Math. Phys. 367, 1095-1151 (2019).

%\bibitem{Demazure} Demazure, M.; A very simple proof of Bott's theorem, Invent. Math. 33, no. 3 (1976) 271-272.

%\bibitem{DiazReventos} Diaz-Miranda, A.; Reventos, A.; Homogeneous contact compact manifolds and homogeneous symplectic manifolds, Bull. Sci. Math., 106 (1982),337-350.

%\bibitem{DuVal} Du Val, P.; On isolated singularities of surfaces which do not affect the conditions of adjunction, Proceedings of the Cambridge Philosophical Society, vol. 30 (1933-34), 453-465; 483-491.

\bibitem{Dragomir} Dragomir, S.; Ornea, L.; Locally Conformal K\"{a}hler Geometry; Progress in Mathematics, Birkh\"{a}user; 1998 edition (1997).

%\bibitem{EGUCHIHANSON} Eguchi, T.; Hanson, A. J.;  Asymptotically flat solutions to Euclidean gravity. Physics Letters, 74B:249-251, 1978.

%\bibitem{Eschenburg} Eschenburg, J.-H., Heintze, E.; An elementary proof of the Cheeger - Gromoll splitting theorem. Ann. Glob. Analysis and Geometry, vol. 2, n. 2 (1984), 141-151.

%\bibitem{Tensorcategory} Etingof, P.; Gelaki, S.; Nikshych, D.; Ostrik, V.; Tensor Categories, Mathematical Surveys and Monographs, American Mathematical Society (2016).

%\bibitem{Falcitelli} Falcitelli, M.; Pastore, A. M.; Ianus; S.; Riemannian Submersions and Related Topics, World Scientific Pub Co Inc (2004).

\bibitem{Algmodels} F\'{e}lix, Y.; Oprea, J.; Tanr\'{e}, D.; Algebraic Models in Geometry. Oxford University Press, USA (2008).

\bibitem{Folland} Folland, G. B.; Weyl manifolds, J. Diff. Geom., 4 (1970), 145-173.

%\bibitem{Pham} Fr\'{e}d\'{e}ric, P.; Formules de Picard-Lefschetz g\'{e}n\'{e}ralis\'{e}es et ramification des int\'{e}grales, Bulletin de la Soci\'{e}t\'{e} Math\'{e}matique de France, 93: 333-367 (1965).

\bibitem{NEGATIVELINEBUNDLE} Fritzsche, K.; Grauert, H.; From Holomorphic Functions to Complex Manifolds, Graduate Texts in Mathematics, Springer (2002).

\bibitem{Gauduchon} Gauduchon, P.; La 1-forme de torsion d'une vari\'{e}t\'{e} e hermitienne compacte, Math. Ann. 267 (1984), 495-518.

\bibitem{Gauduchon1} Gauduchon, P.; Moroianu, A.; Ornea, L.; Compact homogeneous lcK manifolds are Vaisman. Math. Ann. 361, 1043-1048 (2015).

%\bibitem{GauduchonOrnea} Gauduchon. P; Ornea, L.; Locally conformally K\"{a}hler metrics on Hopf surfaces, Ann. Inst. Fourier. 48 (1998), 1107-1127.

\bibitem{Gauduchon2} Gauduchon, P.; Structures de Weyl-Einstein, espaces de twisteurs et vari\'{e}t\'{e} de type $S^{1} \times S^{3}$, J. Reine Angew. Math. 469 (1995), 1-50.

%\bibitem{Gauntlett} Gauntlett, J. P.; Martelli, D.; Sparks, J.; Waldram, D.; Sasaki-Einstein metrics on $S^2 \times S^3$. Adv. Theor. Math. Phys. 8, 711-734 (2004).

%\bibitem{Gibbons} Gibbons, G. W.; Hawking, S. W.; Gravitational multi-instantons, Phys. Lett. B 78 (1978) 430-432.

\bibitem{Gini} Gini, R.; Ornea, L.; Parton, M.; Locally conformal K\"{a}hler reduction. J. Reine Angew. Math., 581:1-21, 2005.

\bibitem{GiniI} Gini, R.; Ornea, L.; Parton, M.; Piccinni, P.; Reduction of Vaisman structures in complex and quaternionic geometry, J. Geom. Phys. 56 (2006) 2501-2522.

%\bibitem{GRAUERT1} Grauert, H; Remmert, R.; Theory of Stein Spaces, Classics in Mathematics, Springer (2004).

%\bibitem{GRAUERT} Grauert, H; \"{U}ber Modifikationen und exzeptionelle analytische Mengen. Math. Ann. 146, 331-368 (1962).

\bibitem{Greub} Greub, W.; Multilinear Algebra, Universitext, Springer; 2nd edition (1978).

\bibitem{Gund} Gundogan, H.; Classification and Structure Theory of Lie Algebras of Smooth Sections, Logos Verlag Berlin GmbH (2011).


\bibitem{HATAKEYMA} Hatakeyama, Y.; Some notes on differentiable manifolds with almost contact structures, Osaka Math. J. (2) 15 (1963), 176-181. MR 27 \#705. 

%\bibitem{Hatcher} Hatcher, A; Algebraic Topology. Cambridge University Press (2002).

\bibitem{Higa} Higa, T.; Weyl manifolds and Einstein-Weyl manifolds. Commun. Math. Univ. Sancti Pauli 42, 143-160 (1993). 

\bibitem{Hitchin} Hitchin, N. J.; K\"{a}hlerian twistor spaces. Proc. Lond. Math. Soc., 43, 133-150 (1981).

%\bibitem{Hirzebruch} Hirzebruch, F.; Singularities and exotic spheres. S\'{e}minaire Bourbaki, Vol. 10, Soc. Math. France, Paris, 1995, pp. 13-32.

%\bibitem{Hitchin} Hitchin, N. J.; Polygons and gravitons, Math. Proc. Cambridge Philos. Soc. 83 (1969) 465-476.

%\bibitem{HOPF} Hopf, H.; Zur Topologie der komplexen Mannigfaltigkeiten, Courant Anniversary Volume, New York, 1948.

\bibitem{Humphreys} Humphreys, J. E.; Introduction to Lie algebras and representation theory, Graduate Texts in Mathematics, no. 9, Springer-Verlag, Berlin-New York (1972).

%\bibitem{HumphreysLAG} Humphreys, J. E.; Linear algebraic groups, Springer-Verlag, 1975.

\bibitem{DANIEL} Huybrechts, D.: Complex geometry: An introduction. Universitext. Springer-Verlag, Berlin (2005).

\bibitem{Ishihara} Ishihara, S.; Quaternion K\"{a}hlerian manifolds and fibered Riemannian spaces with Sasakian 3-structure, Kodai Math. Sem. Rep. 25 (1973), 321-329.

%\bibitem{KK} Ianus, S.; Visinescu, M.; Kaluza-Klein theory with scalar fields and generalized Hopf manifolds. Class. Quantum Gravity 4, 1317-1325 (1987). 

\bibitem{Kamishima} Kamishima, Y.; Ornea, L.; Geometric flow on compact locally conformally K\"{a}hler manifolds. Tohoku Math. J. 57, 201-221 (2005).

%\bibitem{Kempf} Kempf, G. R.; Algebraic Varieties, London Mathematical Society Lecture Note Series (Book 172), Cambridge University Press; 1 ed. (1993).

%\bibitem{Kervaire} Kervaire, M.; Milnor, J.; Groups of Homotopy Spheres I, Ann. Math. 77 (1963) 504.

%\bibitem{Klebanov} Klebanov, I. R.; Witten, E.;  Superconformal field theory on threebranes at a Calabi-Yau singularity, Nucl. Phys. B 536 , 199 (1998).

%\bibitem{Klein} Klein, F.; Vorlesen  \"{u}ber das Ikosaeder und die Aufl\"{o}sung der Gleichun gen vom f\"{u}unften Grade, Teubner, Leipzig 1884. 

\bibitem{Knapp} Knapp, A. W.; Lie Groups: Beyond an Introduction, second edition, Birkh\"{a}user (2002).

\bibitem{KN} Kobayashi, S.; Nomizu, K.; Foundations of differential geometry, vol. 1. John Willey \& Sons (1963).
 
%\bibitem{Kobayashi} Kobayashi, S.; On compact K\"{a}hler manifolds with positive definite Ricci tensor. Ann. Math. 74, 381-385 (1961).

%\bibitem{Kobayashi1} Kobayashi, S.; Principal fibre bundles with the 1-dimensional toroidal group, Tohoku Math. J. (2) 8 (1956), 29-45

%\bibitem{Kodaira} Kodaira, K.; Complex structures on $S^{1}\times S^{3}$, Proc. Nat. Acad. Sci. U.S.A., 55 (1966), 240-243.

%\bibitem{Khunel} Kh\"{u}nel, W.; Differential Geometry: Curves - Surfaces - Manifolds, American Mathematical Society; 2 edition (2005).

%\bibitem{Kronheimer} Kronheimer, P.; The construction of ALE spaces as hyper-kahler quotients, J. Diff. Geom. 28 (1989) 665.

\bibitem{Flaginterplay} Lakshmibai, V.; Flag Varieties: An Interplay of Geometry, Combinatorics, and Representation Theory; Texts and Readings in Mathematics (Book 53), Hindustan Book Agency (2009).

\bibitem{MONOMIAL} Lakshmibai, V.; Raghavan, K. N.; Standard monomial theory, Encyclopaedia of Mathematical Sciences 137, Berlin, New York: 213 Springer-Verlag (2008).

\bibitem{Lazarsfeld} Lazarsfeld, R.; Positivity in Algebraic Geometry I - Classical Setting: Line Bundles and Linear Series, Springer-Verlag Berlin Heidelberg (2004).

%\bibitem{Lamotke} Lamotke, K.; Regular solids and isolated singularities, Vieweg Braunschweig (1986).

%\bibitem{Lawson} Lawson, B.; Michelson, M.; Spin Geometry, Princeton University Press, 1989.

%\bibitem{Lee} Lee, H. C.; A kind of even-dimensional differential geometry and its application to exterior calculus. Am. J. Math., 65:433-438, 1943.

\bibitem{Libermann1} Libermann, P.; Sur le probleme d'equivalance de certaines structures infinitesimales regulieres, Annali Mat. Pura Appl., 36 (1954), 27-120.

\bibitem{Libermann2} Libermann, P.;  Sur les structures presque complexes et autres structures infinitdsimales reguliers, Bull. Soc. Math. Plane, 83 (1955), 195-224.

%\bibitem{Quadric} Lichtenstein, W.; A system of quadrics describing the orbit of the highest weight vector. Proc. Am. Math. Soc. 84 (4), 605-608 (1982).

%\bibitem{Lindstrom} Lindstrom, U.; Rocek, M.; Properties of hyperk\"{a}hler manifolds and their twistor spaces, Commun. Math. Phys. 293, 257 (2010).

%\bibitem{Martelli} Martelli, D.; Sparks, J.; Toric geometry, SasakiEinstein manifolds and a new infinite class of AdS/CFT duals. Commun. Math. Phys. 262, 51 (2006).

\bibitem{MATSUSHIMA} Matsushima, Y.; Remarks on K\"{a}hler-Einstein manifolds, Nagoya Math. J. 46 (1972), 161-173.

%\bibitem{Meinrenken} Meinrenken, E.; Clifford algebras and Lie theory, Springer, Berlin (2013).

%\bibitem{Miller} Miller, G. A.; Blichfeldt, H. F.; Dickson, L. E.; Theory and applications of finite groups, Dover, New York, 1916.

%\bibitem{Milnor} Milnor, J.; On manifolds homeomorphic to the seven sphere, Ann. of Math. vol. 64 (1956) pp. 399-405.

\bibitem{MORIMOTO} Morimoto, A.; On normal almost contact structures, J. Math. Soc. Japan 15 (1963) 420-436.

\bibitem{Moroianukahler} Moroianu, A.; Lectures on K\"{a}hler Geometry, Cambridge University Press; 1 edition (2007).

%\bibitem{Moroianu} Moroianu, A.; Ornea, L.; Homogeneous locally conformally K\"{a}hler manifolds, arXiv:1311.0671v1, 4 Nov 2013.

\bibitem{Nomizu} Nomizu, K.; Lie groups and differential geometry; Publications of the Mathematical Society of Japan, Mathematical Society of Japan; 1st edition (1956).

%\bibitem{ONEILL} O'Neill, B.; Semi-Riemannian Geometry, Pure and Applied Math. 103, Academic Press, New York 1983.

%\bibitem{Onishchik} Onishchik, A. L.; Vinberg, E. B.; Minachin, V.; Gorbatsevich, V. V.; Lie Groups and Lie Algebras III: Structure of Lie Groups and Lie Algebras, Encyclopaedia of Mathematical Sciences, Springer (1994). 

%\bibitem{Vaismanimmersion}  Ornea, L.; Verbitsky, M.; An immersion theorem for compact Vaisman manifolds. Math. Ann. 332, 121-143 (2005).

\bibitem{OrneaPiccinni} Ornea, L.; Piccinni,P.; Locally conformal K\"{a}hler structures in quaternionic geometry, Trans. Amer. Math. Soc. (1997).

\bibitem{OVconjecture} Ornea, L.; Verbitsky, M.; Einstein-Weyl structures on complex manifolds and conformal version of Monge-Amp\`{e}re equation, Bull. Math. Soc. Sci. Math. Roumanie (N.S.) 51 (99) (2008), 339-353.

%\bibitem{Vaismanpotential} Ornea, L.; Verbitsky, M.; Locally conformal K\"{a}hler manifolds with potential, Math. Ann. (2009).

\bibitem{Shells} Ornea, L.; Verbitsky, M.; Locally conformally K\"{a}hler metrics obtained from pseudoconvex shells, Proc. Amer. Math. Soc. 144 (2016), 325-335.

%\bibitem{OrneaPiccinni} Ornea, L.; Piccinni, P.; Induced Hopf bundles and Einstein metrics, in: New Developments in Differential Geometry, Kluwer Academic, Budapest, 1966, pp. 295-306.


\bibitem{Palais} Palais, R. S.; A global formulation of the Lie theory of transformation groups, Memoirs Am. Math. Soc., no. 22 (1957). 

\bibitem{PedersenPoonSwann} Pedersen, H.; Poon, Y. S.; Swann, A. F.; The  Einstein-Weyl  Equations  in  Complex  and  Quaternionic  Geometry, Differential Geometry and its Applications, 3(4) (1993), 309-322.

\bibitem{PoonSalamon} Poon, Y. S.; Salamon, S.; Eight-dimensional quaternionic K\"{a}hler manifolds with positive scalar curvature. J. Diff. Geom., 33, 363-378 (1990).

%\bibitem{Petersen} Petersen, P.; Riemannian Geometry; Graduate Texts in Mathematics, Springer; 2nd edition (2006). 

\bibitem{Procesi} Procesi, C.; Lie Groups: An Approach through Invariants and Representations, Universitext, Springer, New York (2007).

%\bibitem{Seade} Seade, J.; On the topology of isolated singularities in analytic spaces. Birkh\"{a}user, Progress in Mathematics, vol. 241 (2006)

%\bibitem{Sepanski} Sepanski, M. R.; Compact Lie Groups, Graduate Texts in Mathematics, Springer (2007).

%\bibitem{Serre} Serre, J.-P.; G\'{e}om\'{e}trie alg\'{e}brique et g\'{e}om\'{e}trie analytique, Universit\'{e} de Grenoble. Annales de l'Institut Fourier 6: 1-42, (1956).

%\bibitem{Stenzel} Stenzel, M. B.; Ricci-flat metrics on the complexification of a compact rank one symmetric space. Manuscripta Mathematica 80, 151-163 (1993).


\bibitem{Sparks} Sparks, J.; Sasaki-Einstein manifolds, Surveys Diff. Geom. 16 (2011) 265.

\bibitem{Tanno} Tanno, S.; The topology of contact Riemannian manifolds, Illinois J. Math. 12 (1968), 700-717.

\bibitem{TAYLOR} Taylor, J. L.; Several Complex Variables with Connections to Algebraic Geometry and Lie Groups, Graduate Studies in Mathematics (Book 46), American Mathematical Society (2002).

%\bibitem{Tischler} Tischler, D.; On fibering certain foliated manifolds over $S^{1}$, Topology 9 (1970), 153-154. 

\bibitem{Tod} Tod, K. P.; Compact 3-dimensional Einstein-Weyl structures, J. London Math. Soc. (2) 45 (1992), 341-351.

\bibitem{Tondeur} Tondeur, P.; Foliations on Riemannian Manifolds, Series Universitext, Springer-Verlag New York (1988).

\bibitem{Tsukada} Tsukada, K.; The canonical foliation of a compact generalized Hopf manifold, Differential Geom. Appl. 11 (1) (1999) 13-28.

%\bibitem{Udriste} Udri\c{s}te, C.; Convex functions and optimization methods on Riemannian manifolds, Mathematics and Its Applications, Vol. 297, Kluwer Academic, Dordrecht, 1994.

\bibitem{VaismanII} Vaisman, I.; A survey  of  generalized  hopf manifolds. Rend. Sem. Mat. Univ. Politecn. Torino. (1984). 

\bibitem{Vaisman} Vaisman, I.; Generalized Hopf manifolds, Geom. Dedicata 13 (1982), 231-255.

\bibitem{VaismanI} Vaisman, I.; Locally Conformal K\"{a}hler Manifolds with Parallel Lee Form. Rend. Mat. Roma 12 (1979), 263-284. 

%\bibitem{Vinberg} Vinberg, E. B.; Popov, V. L.; On a class of quasihomogeneous affine varieties, Izv. USSR Math., 6 (1972), 743-758.

%\bibitem{Voisin} Voisin, C.; Schneps, L.; Hodge Theory and Complex Algebraic Geometry I, Volume 1, Cambridge Studies in Advanced Mathematics, Cambridge University Press; 1 edition (2008).

\bibitem{Wallach} Wallach, N. R. R.; Compact homogeneous Riemannian manifolds with strictly positive curvature, Ann. Math. 96 (1972), 277-295.

%\bibitem{WallachI} Wallach, N. R. R.; Geometric Invariant Theory: Over the Real and Complex Numbers, Springer (2017).

%\bibitem{WangZiller} Wang, M. Y.; Ziller, W.; Einstein metrics on principal torus bundles, J. Differential Geom. 31 (1990), no. 1, 215-248. MR 91f:53041.

\bibitem{Weyl} Weyl, H.; Space, Time, Matter; Dover Publications (1952).


\end{thebibliography}
\end{document}